\documentclass[a4paper]{amsart}
\usepackage[foot]{amsaddr}

\usepackage[latin1]{inputenc}
\usepackage{amssymb,amsmath,amsthm}
\usepackage{amsmath}
\usepackage{amsfonts}
\usepackage{enumitem}
\usepackage{booktabs}
\usepackage{ifthen}
\usepackage{tikz}
\usetikzlibrary{math,calc}
\usepackage{url}
\usepackage{hyphenat}
\usepackage{hyperref}
\usepackage{graphicx}
\usepackage{calc,etoolbox}
\usepackage{rotating}

\theoremstyle{plain}
\newtheorem{theorem}{Theorem}
\newtheorem{proposition}[theorem]{Proposition}
\newtheorem{lemma}[theorem]{Lemma}

\theoremstyle{definition}
\newtheorem{definition}[theorem]{Definition}
\newtheorem{example}[theorem]{Example}

\newtheorem{remark}[theorem]{Remark}

\numberwithin{figure}{section}
\numberwithin{table}{section}

\newcommand{\IN}{\ensuremath{\mathbb{N}}}

\newcommand{\kLoc}[1][k]{\ensuremath{\mathord{{#1}\text{-Loc}}}}

\newcommand{\tulee}{\ensuremath{\mathrel{\triangleleft}}}
\newcommand{\minorant}{\ensuremath{\leqq}}
\newcommand{\minor}{\ensuremath{\leq}}
\newcommand{\minmin}{\ensuremath{\sqsubseteq}}
\newcommand{\RelCl}[1]{\ensuremath{\mathrel{\mathrm{C}_{#1}}}}
\newcommand{\eqminmin}{\ensuremath{\equiv}}
\newcommand{\nset}[1]{\ensuremath{[{#1}]}}
\newcommand{\card}[1]{\ensuremath{\lvert{#1}\rvert}}
\newcommand{\gendefault}{}
\newcommand{\gen}[2][\gendefault]{\ensuremath{\langle{#2}\rangle_{#1}}}
\newcommand{\clonegen}[1]{\gen[]{#1}}
\newcommand{\vect}[1]{\ensuremath{\mathbf{#1}}}
\newcommand{\lhs}{\hspace{2em}&\hspace{-2em}}

\newcommand{\clIntVal}[3]{\ensuremath{#1_{\ifthenelse{\equal{#2}{}}{\mathord{*}}{#2}\ifthenelse{\equal{#3}{}}{\mathord{*}}{#3}}}}
\newcommand{\clIntEq}[1]{\ensuremath{#1_{=}}}
\newcommand{\clIntNeq}[1]{\ensuremath{#1_{\neq}}}
\newcommand{\clIntLeq}[1]{\ensuremath{#1_{\leq}}}
\newcommand{\clIntGeq}[1]{\ensuremath{#1_{\geq}}}
\newcommand{\clIntNeqOO}[1]{\ensuremath{#1_{{\neq},00}}}
\newcommand{\clIntNeqII}[1]{\ensuremath{#1_{{\neq},11}}}

\newcommand{\clAll}{\ensuremath{\mathsf{\Omega}}}
\newcommand{\clEioo}{\ensuremath{\clIntNeqII{\clAll}}}
\newcommand{\clEioi}{\ensuremath{\clIntGeq{\clAll}}}
\newcommand{\clEiio}{\ensuremath{\clIntLeq{\clAll}}}
\newcommand{\clEiii}{\ensuremath{\clIntNeqOO{\clAll}}}
\newcommand{\clEq}{\ensuremath{\clIntEq{\clAll}}}

\newcommand{\clOX}{\ensuremath{{\clIntVal{\clAll}{0}{}}}}
\newcommand{\clIX}{\ensuremath{{\clIntVal{\clAll}{1}{}}}}
\newcommand{\clXO}{\ensuremath{{\clIntVal{\clAll}{}{0}}}}
\newcommand{\clXI}{\ensuremath{{\clIntVal{\clAll}{}{1}}}}
\newcommand{\clOXC}{\ensuremath{\clOX \cup \clVak}}
\newcommand{\clIXC}{\ensuremath{\clIX \cup \clVak}}
\newcommand{\clXOC}{\ensuremath{\clXO \cup \clVak}}
\newcommand{\clXIC}{\ensuremath{\clXI \cup \clVak}}
\newcommand{\clOO}{\ensuremath{\clIntVal{\clAll}{0}{0}}}
\newcommand{\clII}{\ensuremath{\clIntVal{\clAll}{1}{1}}}
\newcommand{\clOI}{\ensuremath{\clIntVal{\clAll}{0}{1}}}
\newcommand{\clIO}{\ensuremath{\clIntVal{\clAll}{1}{0}}}
\newcommand{\clOOC}{\ensuremath{\clOO \cup \clVak}}
\newcommand{\clIIC}{\ensuremath{\clII \cup \clVak}}
\newcommand{\clOICO}{\ensuremath{\clOI \cup \clVako}}
\newcommand{\clOICI}{\ensuremath{\clOI \cup \clVaki}}
\newcommand{\clOIC}{\ensuremath{\clOI \cup \clVak}}
\newcommand{\clIOCO}{\ensuremath{\clIO \cup \clVako}}
\newcommand{\clIOCI}{\ensuremath{\clIO \cup \clVaki}}
\newcommand{\clIOC}{\ensuremath{\clIO \cup \clVak}}
\newcommand{\clS}{\ensuremath{\mathsf{S}}}
\newcommand{\clSc}{\ensuremath{\clIntVal{\clS}{0}{1}}}

\newcommand{\clSM}{\ensuremath{\mathsf{SM}}}

\newcommand{\clSmin}{\ensuremath{\clS^{-}}}
\newcommand{\clSmaj}{\ensuremath{\clS^{+}}}

\newcommand{\clSminOX}{\ensuremath{\clIntVal{\clSmin}{0}{}}}
\newcommand{\clSminXO}{\ensuremath{\clIntVal{\clSmin}{}{0}}}
\newcommand{\clSminOO}{\ensuremath{\clIntVal{\clSmin}{0}{0}}}
\newcommand{\clSminOI}{\ensuremath{\clIntVal{\clSmin}{0}{1}}}
\newcommand{\clSminIO}{\ensuremath{\clIntVal{\clSmin}{1}{0}}}
\newcommand{\clSminOICO}{\ensuremath{\clSminOI \cup \clVako}}
\newcommand{\clSminIOCO}{\ensuremath{\clSminIO \cup \clVako}}

\newcommand{\clM}{\ensuremath{\mathsf{M}}}
\newcommand{\clMo}{\ensuremath{\clIntVal{\clM}{0}{}}}
\newcommand{\clMi}{\ensuremath{\clIntVal{\clM}{}{1}}}
\newcommand{\clMc}{\ensuremath{\clIntVal{\clM}{0}{1}}}
\newcommand{\clMneg}{\ensuremath{{\overline{\clM}}}}
\newcommand{\clMoneg}{\ensuremath{\clIntVal{\clMneg}{1}{}}}
\newcommand{\clMineg}{\ensuremath{\clIntVal{\clMneg}{}{0}}}
\newcommand{\clMcneg}{\ensuremath{\clIntVal{\clMneg}{1}{0}}}
\newcommand{\clUk}[1]{\ensuremath{\mathsf{U}^{#1}}}
\newcommand{\clU}{\ensuremath{\clUk{2}}}
\newcommand{\clTcU}{\ensuremath{\clIntVal{\clU}{0}{1}}}
\newcommand{\clTcUk}[1]{\ensuremath{\clIntVal{\clUk{#1}}{0}{1}}}
\newcommand{\clTcUCO}{\ensuremath{\clTcU \cup \clVako}}
\newcommand{\clTcUkCO}[1]{\ensuremath{\clTcUk{#1} \cup \clVako}}
\newcommand{\clMU}{\ensuremath{\mathsf{M}\clU}}
\newcommand{\clMcU}{\ensuremath{\clIntVal{\clMU}{0}{1}}}
\newcommand{\clMUk}[1]{\ensuremath{\mathsf{M}\clUk{#1}}}
\newcommand{\clMcUk}[1]{\ensuremath{\clIntVal{\clMUk{#1}}{0}{1}}}
\newcommand{\clWk}[1]{\ensuremath{\mathsf{W}^{#1}}}
\newcommand{\clW}{\ensuremath{\clWk{2}}}
\newcommand{\clTcW}{\ensuremath{\clIntVal{\clW}{0}{1}}}
\newcommand{\clTcWk}[1]{\ensuremath{\clIntVal{\clWk{#1}}{0}{1}}}

\newcommand{\clMW}{\ensuremath{\mathsf{M}\clW}}
\newcommand{\clMcW}{\ensuremath{\clIntVal{\clMW}{0}{1}}}
\newcommand{\clMWk}[1]{\ensuremath{\mathsf{M}\clWk{#1}}}
\newcommand{\clMcWk}[1]{\ensuremath{\clIntVal{\clMWk{#1}}{0}{1}}}

\newcommand{\clWneg}{\ensuremath{\overline{\clW}}}
\newcommand{\clWkneg}[1]{\ensuremath{\overline{\clWk{#1}}}}
\newcommand{\clTcWneg}{\ensuremath{\clIntVal{\clWneg}{1}{0}}}
\newcommand{\clTcWkneg}[1]{\ensuremath{\clIntVal{\clWkneg{#1}}{1}{0}}}
\newcommand{\clTcWnegCO}{\ensuremath{\clTcWneg \cup \clVako}}
\newcommand{\clTcWknegCO}[1]{\ensuremath{\clTcWkneg{#1} \cup \clVako}}
\newcommand{\clMWneg}{\ensuremath{\overline{\clMW}}}
\newcommand{\clMWkneg}[1]{\ensuremath{\overline{\clMWk{#1}}}}
\newcommand{\clMcWneg}{\ensuremath{\clIntVal{\clMWneg}{1}{0}}}
\newcommand{\clMcWkneg}[1]{\ensuremath{\clIntVal{\clMWkneg{#1}}{1}{0}}}
\newcommand{\clUOO}{\ensuremath{\clIntVal{\clU}{0}{0}}}
\newcommand{\clUkOO}[1]{\ensuremath{\clIntVal{\clUk{#1}}{0}{0}}}

\newcommand{\clWnegOO}{\ensuremath{\clIntVal{\clWneg}{0}{0}}}
\newcommand{\clWknegOO}[1]{\ensuremath{\clIntVal{\clWkneg{#1}}{0}{0}}}
\newcommand{\clUWneg}{\ensuremath{\clU \cap \clWneg}}

\newcommand{\clRefl}{\ensuremath{\mathsf{R}}}  
\newcommand{\clReflOO}{\ensuremath{\clIntVal{\clRefl}{0}{0}}}
\newcommand{\clReflII}{\ensuremath{\clIntVal{\clRefl}{1}{1}}}
\newcommand{\clReflOOC}{\ensuremath{\clReflOO \cup \clVak}}
\newcommand{\clReflIIC}{\ensuremath{\clReflII \cup \clVak}}
\newcommand{\clVak}{\ensuremath{\mathsf{C}}}
\newcommand{\clVaka}[1]{\ensuremath{\clVak_{#1}}}
\newcommand{\clVako}{\ensuremath{\clVaka{0}}}
\newcommand{\clVaki}{\ensuremath{\clVaka{1}}}
\newcommand{\clEmpty}{\ensuremath{\mathsf{\emptyset}}}
\newcommand{\clKlik}[2]{\ensuremath{\mathsf{U}^{#1}_{#2}}}

\newcommand{\clIc}{\ensuremath{\mathsf{J}}}

\newcommand{\clAllleq}[1]{\ensuremath{\clAll^{[\leq {#1}]}}}
\newcommand{\posetAllleq}[1]{\ensuremath{\boldsymbol{\clAll}^{[\leq {#1}]}}}
\newcommand{\Ideals}{\ensuremath{\mathrm{Id}}}

\newcommand{\id}{\ensuremath{\mathrm{id}}}
\newcommand{\threshold}[2]{\ensuremath{\mathrm{th}^{#1}_{#2}}}

\DeclareMathOperator{\pr}{pr}
\newcommand{\arity}[1]{\ensuremath{\mathrm{ar}({#1})}}
\newcommand{\closys}[1]{\ensuremath{\mathcal{L}_{#1}}}
\newcommand{\clProj}[1]{\ensuremath{\mathsf{J}_{#1}}}
\DeclareMathOperator{\dom}{dom}

\begin{document}
\title{Near-unanimity-closed minions of Boolean functions}

\author{Erkko Lehtonen}

\address{%
    Department of Mathematics \\
    Khalifa University of Science and Technology \\
    P.O. Box 127788 \\
    Abu Dhabi \\
    United Arab Emirates
}

\date{28 June 2024}

\begin{abstract}
The near\hyp{}unanimity\hyp{}closed minions of Boolean functions, i.e., the clonoids whose target algebra contains a near\hyp{}unanimity function, are completely described.
The key concept towards this result is the minorant\hyp{}minor partial order and its order ideals.
\end{abstract}

\maketitle


\section{Introduction}
\label{sec:introduction}
\numberwithin{theorem}{section}

A set $K$ of functions of several arguments from a set $A$ to a set $B$ is called a \emph{minion} if it is closed under taking \emph{minors}, i.e., functions obtained by composing from the right with projections.
Given clones $C_1$ and $C_2$ on sets $A$ and $B$, respectively, a set $K$ of functions of several arguments from $A$ to $B$ is called a \emph{$(C_1,C_2)$\hyp{}clonoid} if $K$ is stable under right composition with $C_1$ and stable under left composition with $C_2$.
Thus, $(C_1,C_2)$\hyp{}clonoids are minions with additional closure conditions.

Minors, minions, and clonoids have been a topic of much research in the past years.
To the best of the author's knowledge, the term ``minion'' was coined by Opr\v{s}al around the year 2018 (see \cite{BulKroOpr} and \cite[Definition~2.20]{BarBulKroOpr}), and the term ``clonoid'' first appears in the 2016 paper of Aichinger and Mayr \cite{AicMay}.
These ideas have nevertheless appeared in the literature long before -- as an example of such earlier work, we would like to highlight Galois theories characterizing minions and $(C_1,C_2)$\hyp{}clonoids as sets of polymorphisms of relation pairs that are due to Pippenger~\cite{Pippenger} and Couceiro and Foldes \cite{CouFol-2005,CouFol-2009}.
Minions and clonoids arise naturally in universal algebra, and they have recently played an important role in the analysis of the computational complexity of constraint satisfaction problems (CSP), especially in the variant known as promise CSP (see the survey article by Barto et al.\ \cite{BarBulKroOpr}).

The problem of describing the lattice of $(C_1,C_2)$\hyp{}clonoids (for fixed clones $C_1$ and $C_2$), or at least parts thereof, or at least determining its cardinality, is most natural, in the same way as the main problem in clone theory is to describe all clones.
Such classification results, especially concerning clonoids between clones of term operations of modules, have been presented, e.g., in the recent papers of Fioravanti \cite{Fioravanti-AU,Fioravanti-IJAC}, Kreinecker \cite{Kreinecker}, and Mayr and Wynne \cite{MayWyn}.

The following remarkable result by Sparks has sparked further studies on clonoids.
Here, $\closys{(C_1,C_2)}$ denotes the lattice of $(C_1,C_2)$\hyp{}clonoids, and $\clProj{A}$ denotes the clone of projections on $A$.
Recall also that an operation $f \colon A^n \to A$ is a \emph{near\hyp{}unanimity operation} if it satisfies the identities $f(x, \dots, x, y, x, \dots, x) \approx x$, where the single $y$ may be at any argument position.
A ternary near\hyp{}unanimity operation is called a \emph{majority operation}.
A \emph{Mal'cev operation} is a ternary operation $f$ satisfying $f(x,x,y) \approx f(y,x,x) \approx y$.

\begin{theorem}[{Sparks~\cite[Theorem~1.3]{Sparks-2019}}]
\label{thm:Sparks}
Let $A$ be a finite set with $\card{A} > 1$, and let $B = \{0,1\}$.
Let $C$ be a clone on $B$.
Then the following statements hold.
\begin{enumerate}[label={\upshape(\roman*)}]
\item\label{thm:Sparks:finite} $\closys{(\clProj{A},C)}$ is finite if and only if $C$ contains a near\hyp{}unanimity operation.
\item\label{thm:Sparks:countable} $\closys{(\clProj{A},C)}$ is countably infinite if and only if $C$ contains a Mal'cev operation but no majority operation.
\item\label{thm:Sparks:uncountable} $\closys{(\clProj{A},C)}$ has the cardinality of the continuum if and only if $C$ contains neither a near\hyp{}unanimity operation nor a Mal'cev operation.
\end{enumerate}
\end{theorem}

Apart from the cardinality categorization, Theorem~\ref{thm:Sparks} unfortunately offers no information whatsoever about the structure of the lattice $\closys{(\clProj{A},C)}$, nor about the $(\clProj{A},C)$\hyp{}clonoids themselves.
In order to shed more light on this matter,
the current author put forward the idea of systematically classifying $(C_1,C_2)$\hyp{}clonoids for each pair of clones $C_1$ and $C_2$ of Boolean functions.
The case when $C_1$ is arbitrary and $C_2$ contains a Mal'cev operation (the triple sum operation) -- covering item \ref{thm:Sparks:countable} of Theorem~\ref{thm:Sparks} -- was fully treated in \cite{CouLeh-Lcstability}.
The same was done in \cite{Lehtonen-SM} for the case when $C_1$ is arbitrary and $C_2$ contains a \emph{majority operation} (a ternary near\hyp{}unanimity operation) -- this, however, covers only a small part of item \ref{thm:Sparks:finite} of Theorem~\ref{thm:Sparks}.
At the time of writing the paper \cite{Lehtonen-SM}, the author realized that the situation is considerably more complicated in the case when the clone $C_2$ contains a near\hyp{}unanimity operation of arity greater than $3$ but no majority operation -- even though $\closys{(C_1,C_2)}$ is finite in this case, according to Sparks's Theorem~\ref{thm:Sparks}.
This was left as an open problem.

It is the goal of the current paper to complete the description of the $(\clIc,C)$\hyp{}clonoids of Boolean functions in the cases when $C$ satisfies condition \ref{thm:Sparks:finite} of Theorem~\ref{thm:Sparks} and $\closys{(\clIc,C)}$ is finite.
After providing the basic notions and preliminary results that will be needed throughout the paper (Section~\ref{sec:preliminaries}), we define properties of Boolean functions (Section~\ref{sec:Boolean}), and we introduce our main tools: \emph{minorant minors}, i.e., minions closed under minorants (Section~\ref{sec:minmin}) and $(G,k)$\hyp{}semibisectable functions (Section~\ref{sec:semibisectable}).
Section~\ref{sec:IcMcUk} is dedicated to our main result and its proof: a complete description of the $(\clIc,\clMcUk{k})$\hyp{}clonoids of Boolean functions, for $k \geq 2$.
(Here, $\clMcUk{k}$ denotes the clone of idempotent monotone $1$\hyp{}separating functions of rank $k$; see Definition~\ref{def:separating}.)
With this result at hand, we then describe in Section~\ref{sec:C1C2} the $(\clIc,C)$\hyp{}clonoids, where $C$ is a clone containing $\clMcUk{k}$.
We conclude the paper with some final remarks and open problems in Section~\ref{sec:final}.


\section{Preliminaries}
\label{sec:preliminaries}

\numberwithin{theorem}{subsection}
\renewcommand{\gendefault}{(C_1,C_2)}

\subsection{General}

We denote by $\IN$ and $\IN_{+}$ the set of nonnegative integers and the set of positive integers, respectively.
For $n \in \IN$, let $\nset{n} := \{ \, i \in \nset{N} \mid 1 \leq i \leq n \, \}$.

Let $(P, \mathord{\leq})$ be a partially ordered set, and let $X \subseteq P$.
The order ideal and the order filter generated by $X$ in $P$ are
\begin{align*}
{\downarrow} X &:= \{ \, y \in P \mid \exists x \in X \colon y \leq x \, \}, &
{\uparrow} X &:= \{ \, y \in P \mid \exists x \in X \colon x \leq y \, \},
\end{align*}
respectively.
For $x \in P$, the principal ideal and the principal filter of $x$ are ${\downarrow} x := {\downarrow} \{x\}$ and ${\uparrow} x := {\uparrow} \{x\}$, respectively.

\subsection{Minors, minions, clones, clonoids}

A \emph{function \textup{(}of several arguments\textup{)}} from a set $A$ to a set $B$ is a mapping $f \colon A^n \to B$ for some $n \in \IN$, called the \emph{arity} of $f$.
In the case when $A = B$, we speak of \emph{operations} on $A$.
We denote by $\mathcal{F}^{(n)}_{AB}$ the set $B^{A^n}$ of all $n$\hyp{}ary functions from $A$ to $B$, and we let $\mathcal{F}_{AB} := \bigcup_{n \in \IN} \mathcal{F}^{(n)}_{AB}$.
Similarly, we denote by $\mathcal{O}^{(n)}_A$ the set $A^{A^n}$ of all $n$\hyp{}ary operations on $A$, and we let $\mathcal{O}_A := \bigcup_{n \in \IN} \mathcal{O}^{(n)}_A$.
For any $C \subseteq \mathcal{F}_{AB}$, the \emph{$n$\hyp{}ary part} of $C$ is $C^{(n)} := C \cap \mathcal{F}^{(n)}_{AB}$.

The $i$\hyp{}th $n$\hyp{}ary \emph{projection} on $A$ is the operation $\pr^{(n)}_i \colon A^n \to A$ given by the rule $(a_1, \dots, a_n) \mapsto a_i$.
We denote by $\clProj{A}$ the set of all projections on $A$.

The \emph{composition} of $f \in \mathcal{F}^{(n)}_{BC}$ (the \emph{outer function}) with $g_1, \dots, g_n \in \mathcal{F}^{(m)}_{AB}$ (the \emph{inner functions})
is the function $f(g_1, \dots, g_n) \in \mathcal{F}^{(m)}_{AC}$ given by the rule
\[
f(g_1, \dots, g_n)(\vect{a})
= f(g_1(\vect{a}), \dots, g_n(\vect{a})),
\]
for all $\vect{a} \in A^m$.
The notion of functional composition extends to function classes.
If $C \subseteq \mathcal{F}_{BC}$ and $K \subseteq \mathcal{F}_{AB}$, then the \emph{composition} of $C$ with $K$ is
\[
C K := \{ \, f(g_1, \dots, g_n) \mid \exists m, n \in \IN_{+} \colon \, f \in C^{(n)}, \, g_1, \dots, g_n \in K^{(m)} \,\}.
\]

A set $C \subseteq \mathcal{O}_A$ is called a \emph{clone} on $A$ if $\clProj{A} \subseteq C$ and $C C \subseteq C$.
The set of all clones on $A$ is a closure system.
For $F \subseteq \mathcal{O}_A$, we denote by $\clonegen{F}$ the clone generated by $F$, i.e., the smallest clone on $A$ that contains $F$.
The clones on a two\hyp{}element set were described by Post~\cite{Post}; and the lattice of these clones, known as \emph{Post's lattice}, is presented in Figure~\ref{fig:Post} (see Section~\ref{sec:Boolean} for an explanation of the notation for these clones).

Let $f, g \in \mathcal{F}_{AB}$, say $f$ is $n$\hyp{}ary and $g$ is $m$\hyp{}ary.
We say that $f$ is a \emph{minor} of $g$ if $f \in \{g\} \clProj{A}$, or, more explicitly, if there is a $\sigma \colon \nset{m} \to \nset{n}$ such that $f = g(\pr^{(n)}_{\sigma(1)}, \dots, \pr^{(n)}_{\sigma(m)})$.
Such a map $\sigma$ is referred to as a \emph{minor formation map}, and we denote by $g_\sigma$ the function $g(\pr^{(n)}_{\sigma(1)}, \dots, \pr^{(n)}_{\sigma(m)})$.
It holds that $g_\sigma(\vect{a}) = g(\vect{a} \sigma)$ for all $\vect{a} \in A^n$, where $\vect{a} \sigma = \vect{a} \circ \sigma$ (recall that $\vect{a}$ is a mapping $\nset{n} \to A$).
A set $C \subseteq \mathcal{F}_{AB}$ is a \emph{minion} (or a \emph{minor\hyp{}closed class}) (from $A$ to $B$) if $C \clProj{A} \subseteq C$.
The set of all minions from $A$ to $B$ constitutes a closure system.
It is easy to see that the intersection and the union of minions are minions, so the lattice of all minions is a sublattice of the power set lattice of $\mathcal{F}_{AB}$.

For minors of functions, especially of a small arity, we sometimes use the following shorthand notation.
If $g$ is $m$\hyp{}ary and $\sigma \colon \nset{m} \to \nset{n}$, we write $g_{\sigma(1) \sigma(2) \dots \sigma(m)}$ for $g_\sigma$.
For example, we may write ${+}_{ij}$ to denote $\mathord{+}(\pr^{(n)}_i, \pr^{(n)}_j)$.

Let $C_1$ and $C_2$ be clones on $A$ and $B$, respectively.
A set $K \subseteq \mathcal{F}_{AB}$ is
\emph{stable under right composition with $C_1$} if $K C_1 \subseteq K$,
and it is
\emph{stable under left composition with $C_2$} if $C_2 K \subseteq K$.
We say that $K$ is \emph{$(C_1,C_2)$\hyp{}stable} or that $K$ is a \emph{$(C_1,C_2)$\hyp{}clonoid} if $K C_1 \subseteq K$ and $C_2 K \subseteq K$.
The set of all $(C_1,C_2)$\hyp{}clonoids constitutes a closure system, denoted by $\closys{(C_1,C_2)}$.
For $K \subseteq \mathcal{F}_{AB}$, we denote by $\gen{K}$ the $(C_1,C_2)$\hyp{}clonoid generated by $K$, i.e., the smallest $(C_1,C_2)$\hyp{}clonoid containing $K$.
It is known that $\gen{K} = C_2 ( K C_1 )$ (see \cite[Lemma~2.5]{Lehtonen-SM}).

\begin{lemma}[{\cite[Lemma~2.1]{Lehtonen-SM}}]
\label{lem:stable-monotonicity}
Let $C_1$ and $C'_1$ be clones on $A$ and let $C_2$ and $C'_2$ be clones on $B$.
If $C_1 \subseteq C'_1$ and $C_2 \subseteq C'_2$, then every $(C'_1,C'_2)$\hyp{}clonoid is a $(C_1,C_2)$\hyp{}clonoid.
\end{lemma}

\begin{lemma}[{\cite[Lemma~2.2]{Lehtonen-SM}}]
\label{lem:stable-clones}
Let $C$ and $K$ be clones on $A$.
Then the following statements hold.
\begin{enumerate}[label={\upshape(\roman*)}]
\item $KC \subseteq K$ if and only if $C \subseteq K$.
\item $CK \subseteq K$ if and only if $C \subseteq K$.
\end{enumerate}
\end{lemma}

\begin{lemma}[{Couceiro, Foldes \cite[Associativity Lemma]{CouFol-2007,CouFol-2009}}]
Let $A$, $B$, $C$, and $D$ be arbitrary nonempty sets, and let $I \subseteq \mathcal{F}_{CD}$, $J \subseteq \mathcal{F}_{BC}$, $K \subseteq \mathcal{F}_{AB}$.
Then the following statements hold.
\begin{enumerate}[label={\upshape(\roman*)}]
\item $(IJ)K \subseteq I(JK)$.
\item If $J$ is stable under right composition with the clone of projections on $B$ \textup{(}i.e., $J$ is minor-closed\textup{)}, then $(IJ)K = I(JK)$.
\end{enumerate}
\end{lemma}

\begin{lemma}[{\cite[Lemma~2.5]{Lehtonen-SM}}]
\label{lem:F-closure}
Let $F \subseteq \mathcal{F}_{AB}$ and let $C_1$ and $C_2$ be clones on $A$ and $B$, respectively.
Then $\gen[(C_1,C_2)]{F} = C_2 ( F C_1 )$.
\end{lemma}

In the next lemma we employ the binary composition operation $\ast$ defined as follows:
if $f \in \mathcal{O}_A^{(n)}$ and $g \in \mathcal{O}_A^{(m)}$, then $f \ast g \in \mathcal{O}_A^{(m+n-1)}$ is defined by
\[
(f \ast g)(a_1, \dots, a_{m+n-1}) :=
f(g(a_1, \dots, a_m), a_{m+1}, \dots, a_{m+n-1}),
\]
for all $a_1, \dots, a_{m+n-1} \in A$.

\begin{lemma}[{\cite[Lemma~3.2]{CouLeh-Lcstability}}]
\label{lem:right-stable}
Let $F \subseteq \mathcal{O}_A$, let $C$ be a clone on $A$, and let $G$ be a generating set of $C$.
Then the following conditions are equivalent.
\begin{enumerate}[label={\upshape(\roman*)}]
\item $FC \subseteq F$.
\item $F$ is minor\hyp{}closed and $f \ast g \in F$ whenever $f \in F$ and $g \in C$.
\item $F$ is minor\hyp{}closed and $f \ast g \in F$ whenever $f \in F$ and $g \in G$.
\end{enumerate}
\end{lemma}

\begin{lemma}[{\cite[Lemma~3.3]{CouLeh-Lcstability}}]
\label{lem:left-stable}
Let $F \subseteq \mathcal{O}_A$, let $C$ be a clone on $A$, and let $G$ be a generating set of $C$.
Then the following conditions are equivalent.
\begin{enumerate}[label={\upshape(\roman*)}]
\item $CF \subseteq F$.
\item $g(f_1, \dots, f_n) \in F$ whenever $g \in C^{(n)}$ and $f_1, \dots, f_n \in F^{(m)}$ for some $n, m \in \IN$.
\item $g(f_1, \dots, f_n) \in F$ whenever $g \in G^{(n)}$ and $f_1, \dots, f_n \in F^{(m)}$ for some $n, m \in \IN$.
\end{enumerate}
\end{lemma}


\section{Properties and classes of Boolean functions}
\label{sec:Boolean}

\subsection{Boolean functions and their properties}

Operations on $\{0,1\}$ are called \emph{Boolean functions.}
The operation tables in Figure~\ref{fig:BFs} define some well\hyp{}known Boolean functions:
the constant functions $0$ and $1$,
identity $\id$,
negation $\neg$,
conjunction $\mathord{\wedge}$,
disjunction $\mathord{\vee}$,
addition $\mathord{+}$,
equivalence $\mathord{\leftrightarrow}$,
nonimplication $\mathord{\nrightarrow}$,
majority $\mu$,
and triple sum $\mathord{\oplus_3}$.

\begin{figure}
\newcommand{\FIGunary}{$\begin{array}[t]{c|cccc}
x_1 & 0 & 1 & \id & \neg \\
\hline
0 & 0 & 1 & 0 & 1 \\
1 & 0 & 1 & 1 & 0
\end{array}$}
\newcommand{\FIGbinary}{$\begin{array}[t]{cc|ccccc}
x_1 & x_2 & \wedge & \vee & + & \leftrightarrow & \nrightarrow \\
\hline
0 & 0 & 0 & 0 & 0 & 1 & 0 \\
0 & 1 & 0 & 1 & 1 & 0 & 0 \\
1 & 0 & 0 & 1 & 1 & 0 & 1 \\
1 & 1 & 1 & 1 & 0 & 1 & 0
\end{array}$}
\newcommand{\FIGternary}{$\begin{array}[t]{ccc|cc}
x_1 & x_2 & x_3 & \mu & \mathord{\oplus_3} \\
\hline
0 & 0 & 0 & 0 & 0 \\
0 & 0 & 1 & 0 & 1 \\
0 & 1 & 0 & 0 & 1 \\
0 & 1 & 1 & 1 & 0 \\
1 & 0 & 0 & 0 & 1 \\
1 & 0 & 1 & 1 & 0 \\
1 & 1 & 0 & 1 & 0 \\
1 & 1 & 1 & 1 & 1
\end{array}$}
\newlength{\FIGternaryheight}
\settototalheight{\FIGternaryheight}{\FIGternary}
\newlength{\FIGbinarywidth}
\settowidth{\FIGbinarywidth}{\FIGbinary}
\begin{minipage}[t][\FIGternaryheight][t]{\FIGbinarywidth}
\FIGunary
\vfill
\FIGbinary
\end{minipage}
\qquad\qquad
\FIGternary

\caption{Some well\hyp{}known Boolean functions.}
\label{fig:BFs}
\end{figure}

For $a \in \{0,1\}^n$, let $\overline{a} :=  1 - a$.
For $\vect{a} = (a_1, \dots, a_n) \in \{0,1\}^n$, let $\overline{\vect{a}} := (\overline{a_1}, \dots, \overline{a_n})$.
We regard the set $\{0,1\}$ totally ordered by the natural order $0 < 1$, which induces the direct product order on $\{0,1\}^n$.
The poset $(\{0,1\}^n, {\leq})$ constitutes a Boolean lattice, i.e., a complemented distributive lattice with least and greatest elements $\vect{0} = (0, \dots 0)$ and $\vect{1} = (1, \dots, 1)$, and with the map $\vect{a} \mapsto \overline{\vect{a}}$ being the complementation.

Denote by $\vect{e}_i$ the $n$\hyp{}tuple $(0, \dots, 0, 1, 0, \dots, 0)$, where the single $1$ is at the $i$\hyp{}th position.

We are going to define several properties of Boolean functions and introduce notation for classes of Boolean functions satisfying these properties.
For the reader's convenience, we summarize the notation in Table~\ref{table:Bf-notation}.

\begin{table}
\begin{center}
\begin{tabular}{lll}
\toprule
Notation & Description & Definition \\
\midrule
$\clAll$ & all Boolean functions & \ref{def:all} \\
$\clRefl$ & reflexive functions & \ref{def:neg} \\
$\clS$ & self-dual functions & \ref{def:neg} \\
$\clSmin$ & minorants of self-dual functions & \ref{def:minorant} \\
$\clSmaj$ & majorants of self-dual functions & \ref{def:minorant} \\
$\clM$ & monotone functions & \ref{def:monotone} \\
$\clSM$ & $\clS \cap \clM$ & \ref{def:monotone} \\
$\clUk{k}$ & $1$\hyp{}separating functions of rank $k$ & \ref{def:separating} \\
$\clWk{k}$ & $0$\hyp{}separating functions of rank $k$ & \ref{def:separating} \\
$\clMUk{k}$ & $\clM \cap \clUk{k}$ & \ref{def:separating} \\
$\clMWk{k}$ & $\clM \cap \clWk{k}$ & \ref{def:separating} \\
$\clIc$ & projections & \ref{def:proj} \\
$\clVak$ & constant functions & \ref{def:constant} \\
$\clVako$ & constant $0$ functions & \ref{def:constant} \\
$\clVaki$ & constant $1$ functions & \ref{def:constant} \\
$\clKlik{k}{\Theta}$ & $k$\hyp{}local closure of the minorant minion ${\downarrow} \Theta$ & \ref{def:UkTheta} \\
\bottomrule
\end{tabular}
\end{center}

\bigskip
\caption{Notation for classes of Boolean functions.}
\label{table:Bf-notation}
\end{table}

\begin{definition}
\label{def:all}
Denote by $\clAll$ the class of all Boolean functions.
For any $C \subseteq \clAll$ and
any $a, b \in \{0,1\}$, let
\begin{align*}
\clIntVal{C}{a}{} & := \{ \, f \in C \mid f(\vect{0}) = a \, \}, \\
\clIntVal{C}{}{b} & := \{ \, f \in C \mid f(\vect{1}) = b \, \}, \\
\clIntVal{C}{a}{b} & := \clIntVal{C}{a}{} \cap \clIntVal{C}{}{b}, \\
\clIntEq{C} & := \{ \, f \in C \mid f(\vect{0}) = f(\vect{1}) \, \} = \clIntVal{C}{0}{0} \cup \clIntVal{C}{1}{1}, \\
\clIntNeq{C} & := \{ \, f \in C \mid f(\vect{0}) \neq f(\vect{1}) \, \} = \clIntVal{C}{0}{1} \cup \clIntVal{C}{1}{0}, \\
\clIntLeq{C} & := \{ \, f \in C \mid f(\vect{0}) \leq f(\vect{1}) \, \} = \clIntVal{C}{0}{0} \cup \clIntVal{C}{0}{1} \cup \clIntVal{C}{1}{1}, \\
\clIntGeq{C} & := \{ \, f \in C \mid f(\vect{0}) \geq f(\vect{1}) \, \} = \clIntVal{C}{0}{0} \cup \clIntVal{C}{1}{0} \cup \clIntVal{C}{1}{1}, \\
\clIntNeqOO{C} & := \clIntNeq{C} \cup \clIntVal{C}{0}{0} = \clIntVal{C}{0}{1} \cup \clIntVal{C}{1}{0} \cup \clIntVal{C}{0}{0}, \\
\clIntNeqII{C} & := \clIntNeq{C} \cup \clIntVal{C}{1}{1} = \clIntVal{C}{0}{1} \cup \clIntVal{C}{1}{0} \cup \clIntVal{C}{1}{1}.
\end{align*}
\end{definition}

\begin{definition}
\label{def:constant}
A function $f \colon \{0,1\}^n \to \{0,1\}$ is \emph{constant} if $f(\vect{a}) = f(\vect{b})$ for all $\vect{a}, \vect{b} \in \{0,1\}^n$.
We denote by $\clVak$ the class of all constant Boolean functions, and we introduce the shorthands $\clVako := \clIntVal{\clVak}{0}{0}$ and $\clVaki := \clIntVal{\clVak}{1}{1}$.
\end{definition}

\begin{definition}
\label{def:neg}
Let $f \colon \{0,1\}^n \to \{0,1\}$.
The \emph{negation} $\overline{f}$, the \emph{inner negation} $f^\mathrm{n}$, and the \emph{dual} $f^\mathrm{d}$ of $f$ are the $n$\hyp{}ary Boolean functions given by the rules
$\overline{f}(\vect{a}) := \overline{f(\vect{a})}$, $f^\mathrm{n}(\vect{a}) := f(\overline{\vect{a}})$, and $f^\mathrm{d}(\vect{a}) := \overline{f(\overline{\vect{a}})}$, for all $\vect{a} \in \{0,1\}^n$.
In other words, $\overline{f} = \neg(f)$, $f^\mathrm{n} = f(\neg^{(n)}_1, \dots, \neg^{(n)}_n)$, and $f^\mathrm{d} = \neg(f(\neg^{(n)}_1, \dots, \neg^{(n)}_n))$.
A function $f \colon \{0,1\}^n \to \{0,1\}$ is \emph{reflexive} if $f = f^\mathrm{n}$, i.e., $f(\vect{a}) = f(\overline{\vect{a}})$ for all $\vect{a} \in \{0,1\}^n$, and $f$ is \emph{self\hyp{}dual} if $f = f^\mathrm{d}$, i.e., $f(\vect{a}) = \overline{f(\overline{\vect{a}})}$ for all $\vect{a} \in \{0,1\}^n$.
We denote by $\clRefl$ and $\clS$ the class of all reflexive functions and the class of all self\hyp{}dual functions, respectively.
For any $C \subseteq \clAll$, let $\overline{C} := \{ \, \overline{f} \mid f \in C \, \}$, $C^\mathrm{n} := \{ \, f^\mathrm{n} \mid f \in C \, \}$, and $C^\mathrm{d} := \{ \, f^\mathrm{d} \mid f \in C \, \}$.
\end{definition}

\begin{lemma}\label{lem:C1nC2}
\label{lem:stability-nd}
Let $C_1$ and $C_2$ be clones.
If $C_1 = C_1^\mathrm{d}$, then the inner negation $K \mapsto K^\mathrm{n}$ is an automorphism of $\closys{(C_1,C_2)}$.
\end{lemma}

\begin{proof}
Assume that $C_1 = C_1^\mathrm{d}$, and let $K$ be a $(C_1,C_2)$\hyp{}clonoid.
We need to show that $K^\mathrm{n}$ is a $(C_1,C_2)$\hyp{}clonoid.

First we show that $K^\mathrm{n} C_1 \subseteq K^\mathrm{n}$.
Let $f \in K^\mathrm{n} C_1$.
Then $f = g(h_1, \dots, h_n)$ with $g \in K^\mathrm{n}$ and $h_1, \dots, h_n \in C_1$.
Thus $g^\mathrm{n} \in K$ and $h_1^\mathrm{d}, \dots, h_n^\mathrm{d} \in C_1^\mathrm{d} = C_1$, and we have
\begin{align*}
f
&
= g(h_1, \dots, h_n)
= (g^\mathrm{n})^\mathrm{n}(h_1, \dots, h_n)
= (g^\mathrm{n}(\neg^{(n)}_1, \dots, \neg^{(n)}_n))(h_1, \dots, h_n)
\\ &
= g^\mathrm{n}(\overline{h_1}, \dots, \overline{h_n})
= g^\mathrm{n}((h_1^\mathrm{d})^\mathrm{n}, \dots, (h_n^\mathrm{d})^\mathrm{n})
= (g^\mathrm{n}(h_1^\mathrm{d}, \dots, h_n^\mathrm{d}))^\mathrm{n}
\in (K C_1)^\mathrm{n}
\subseteq K^\mathrm{n}.
\end{align*}

Next we show that $C_2 K^\mathrm{n} \subseteq K^\mathrm{n}$.
Let $f \in C_2 K^\mathrm{n}$.
Then $f = g(h_1, \dots, h_n)$ with $g \in C$ and $h_1, \dots, h_n \in K^\mathrm{n}$, so $h_1^\mathrm{n}, \dots, h_n^\mathrm{n} \in K$.
Since $C_2 K \subseteq K$, we have $g(h_1^\mathrm{n}, \dots, h_n^\mathrm{n}) \in K$, and hence
\[
f = g(h_1, \dots, h_n)
= g((h_1^\mathrm{n})^\mathrm{n}, \dots, (h_n^\mathrm{n})^\mathrm{n})
= (g(h_1^\mathrm{n}, \dots, h_n^\mathrm{n}))^\mathrm{n}
\in K^\mathrm{n}.
\]

It remains to note that $K \mapsto K^\mathrm{n}$ is an involution.
It follows that this map is an automorphism of $\closys{(C_1,C_2)}$.
\end{proof}

\begin{lemma}
\label{lem:duality}
Let $C_1$ and $C_2$ be clones on $\{0,1\}$, and let $K \subseteq \clAll$.
Then $K$ is a $(C_1,C_2)$\hyp{}clonoid if and only if $K^\mathrm{d}$ is a $(C_1^\mathrm{d},C_2^\mathrm{d})$\hyp{}clonoid.
\end{lemma}

\begin{proof}
Observe first that for all $f \in \clAll^{(n)}$, $g_1, \dots, g_n \in \clAll^{(m)}$, we have, for all $\vect{a} \in \{0,1\}^m$,
\begin{align*}
\lhs
(f(g_1, \dots, g_n))^\mathrm{d}(\vect{a})
= \overline{f(g_1, \dots, g_n)(\overline{\vect{a}})}
= \overline{f(g_1(\overline{\vect{a}}), \dots, g_n(\overline{\vect{a}}))}
\\
& =
\overline{f(\overline{\overline{g_1(\overline{\vect{a}})}}, \dots, \overline{\overline{g_n(\overline{\vect{a}})}})}
= f^\mathrm{d}(g_1^\mathrm{d}(\vect{a}), \dots, g_n^\mathrm{d}(\vect{a}))
= f^\mathrm{d}(g_1^\mathrm{d}, \dots, g_n^\mathrm{d})(\vect{a}),
\end{align*}
so $(f(g_1, \dots, g_n))^\mathrm{d} = f^\mathrm{d}(g_1^\mathrm{d}, \dots, g_n^\mathrm{d}$.
From this it follows that for all $I, J \subseteq \clAll$, $(IJ)^\mathrm{d} = I^\mathrm{d} J^\mathrm{d}$.

Now, if $K$ is a $(C_1,C_2)$\hyp{}clonoid, we have that $K = \gen[(C_1,C_2)]{K} = C_2 ( K C_1 )$.
Therefore, $K^\mathrm{d} = (C_2 ( K C_1 ))^\mathrm{d} = C_2^\mathrm{d} ( K^\mathrm{d} C_1^\mathrm{d} ) = \gen[(C_1^\mathrm{d},C_2^\mathrm{d})]{K^\mathrm{d}}$, which means that $K^\mathrm{d}$ is a $(C_1^\mathrm{d},C_2^\mathrm{d})$\hyp{}clonoid.
The converse implication follows by observing that $(C^\mathrm{d})^\mathrm{d} = C$ for all $C \subseteq \clAll$.
\end{proof}

\begin{definition}
\label{def:minorant}
Let $f, g \colon \{0,1\}^n \to \{0,1\}$.
We say that $f$ is a \emph{minorant} of $g$, or that $g$ is a \emph{majorant} of $f$, and we write $f \minorant g$, if $f(\vect{a}) \leq g(\vect{a})$ for all $\vect{a} \in \{0,1\}^n$.
We denote by $\clSmin$ and $\clSmaj$ the classes of all minorants of self\hyp{}dual functions and all majorants of self\hyp{}dual functions, respectively.
Note that $\clS = \clSmin \cap \clSmaj$ and that $f \in \clSmin$ if and only if $f(\vect{a}) \wedge f(\overline{\vect{a}}) = 0$ for all $\vect{a} \in \{0,1\}^n$, and $f \in \clSmaj$ if and only if $f(\vect{a}) \vee f(\overline{\vect{a}}) = 1$ for all $\vect{a} \in \{0,1\}^n$.
\end{definition}

\begin{definition}
\label{def:monotone}
A function $f \in \{0,1\}^n \to \{0,1\}$ is \emph{monotone} if $f(\vect{a}) \leq f(\vect{b})$ whenever $\vect{a} \leq \vect{b}$.
We denote by $\clM$ the class of all monotone functions.
Let $\clSM := \clS \cap \clM$.
\end{definition}

\begin{definition}
\label{def:separating}
For $k \in \{2, \dots, \infty\}$,
a function $f \in \{0,1\}^n \to \{0,1\}$ is \emph{$1$\hyp{}separating of rank $k$} if for all $T \subseteq f^{-1}(1)$ with $\card{T} \leq k$, it holds that $\bigwedge T \neq \vect{0}$,
and $f$ is \emph{$0$\hyp{}separating of rank $k$} if for all $F \subseteq f^{-1}(0)$ with $\card{F} \leq k$, it holds that $\bigvee F \neq \vect{1}$.
We denote by $\clUk{k}$ and $\clWk{k}$ the classes of all $1$\hyp{}separating functions of rank $k$ and all $0$\hyp{}separating functions of rank $k$.
We introduce the shorthands $\clMUk{k}$ for $\clM \cap \clUk{k}$ and $\clMWk{k}$ for $\clM \cap \clWk{k}$.
\end{definition}

\begin{definition}
\label{def:proj}
Denote by $\clIc$ the class of all projections on $\{0,1\}$.
\end{definition}

The clones of Boolean functions were described by Post~\cite{Post}.
The lattice of clones of Boolean functions, also known as \emph{Post's lattice,} is presented in Figure~\ref{fig:Post}.
The classes $\clAll$, $\clOX$, $\clXI$, $\clOI$, $\clM$, $\clMo$, $\clMi$, $\clMc$, $\clS$, $\clSc$, $\clSM$, $\clUk{k}$, $\clTcUk{k}$, $\clMUk{k}$, $\clMcUk{k}$, $\clWk{k}$, $\clTcWk{k}$, $\clMWk{k}$, $\clMcWk{k}$, and $\clIc$ ($k \geq 2$) that were defined above are clones, and they are indicated in Figure~\ref{fig:Post}.

\begin{figure}
\begin{center}
\scalebox{0.5}{%
\tikzstyle{every node}=[circle, draw, fill=black, scale=1, font=\normalsize]
\begin{tikzpicture}[baseline, scale=0.8]
   \node [label = below:$\clIc$] (Ic) at (0,-1) {};
   \node (Istar) at (0,0.5) {};
   \node (I0) at (4.5,0.5) {};
   \node (I1) at (-4.5,0.5) {};
   \node (I) at (0,2) {};
   \node (Omega1) at (0,5) {};
   \node (Lc) at (0,7.5) {};
   \node (LS) at (0,9) {};
   \node (L0) at (3,9) {};
   \node (L1) at (-3,9) {};
   \node (L) at (0,10.5) {};
   \node [label = below:$\clSM\,\,$] (SM) at (0,13.5) {};
   \node [label = left:$\clSc$] (Sc) at (0,15) {};
   \node [label = above:$\clS$] (S) at (0,16.5) {};
   \node [label = below:$\clMc$] (Mc) at (0,23) {};
   \node [label = left:$\clMo\,\,$] (M0) at (2,24) {};
   \node [label = right:$\,\,\clMi$] (M1) at (-2,24) {};
   \node [label = above:$\clM\,\,$] (M) at (0,25) {};
   \node (Lamc) at (7.2,6.7) {};
   \node (Lam1) at (5,7.5) {};
   \node (Lam0) at (8.7,7.5) {};
   \node (Lam) at (6.5,8.3) {};
   \node [label = left:$\clMcUk{\infty}$] (McUi) at (7.2,11.5) {};
   \node [label = left:$\clMUk{\infty}$] (MUi) at (8.7,13) {};
   \node [label = right:$\clTcUk{\infty}$] (TcUi) at (10.2,12) {};
   \node [label = right:$\clUk{\infty}$] (Ui) at (11.7,13.5) {};
   \node [label = left:$\clMcUk{3}$](McU3) at (7.2,16) {};
   \node [label = left:$\clMUk{3}$] (MU3) at (8.7,17.5) {};
   \node [label = right:$\clTcUk{3}$] (TcU3) at (10.2,16.5) {};
   \node [label = right:$\clUk{3}$] (U3) at (11.7,18) {};
   \node [label = left:$\clMcU\,$] (McU2) at (7.2,19) {};
   \node [label = left:$\clMU\,$] (MU2) at (8.7,20.5) {};
   \node [label = right:$\clTcU$] (TcU2) at (10.2,19.5) {};
   \node [label = right:$\clU$] (U2) at (11.7,21) {};
   \node (Vc) at (-7.2,6.7) {};
   \node (V0) at (-5,7.5) {};
   \node (V1) at (-8.7,7.5) {};
   \node (V) at (-6.5,8.3) {};
   \node [label = right:$\clMcWk{\infty}$] (McWi) at (-7.2,11.5) {};
   \node [label = right:$\clMWk{\infty}$] (MWi) at (-8.7,13) {};
   \node [label = left:$\clTcWk{\infty}$] (TcWi) at (-10.2,12) {};
   \node [label = left:$\clWk{\infty}$] (Wi) at (-11.7,13.5) {};
   \node [label = right:$\clMcWk{3}$] (McW3) at (-7.2,16) {};
   \node [label = right:$\clMWk{3}$] (MW3) at (-8.7,17.5) {};
   \node [label = left:$\clTcWk{3}$] (TcW3) at (-10.2,16.5) {};
   \node [label = left:$\clWk{3}$] (W3) at (-11.7,18) {};
   \node [label = right:$\,\,\clMcW$] (McW2) at (-7.2,19) {};
   \node [label = right:$\clMW$] (MW2) at (-8.7,20.5) {};
   \node [label = left:$\clTcW$] (TcW2) at (-10.2,19.5) {};
   \node [label = left:$\clW$] (W2) at (-11.7,21) {};
   \node [label = above:$\clOI$] (Tc) at (0,28) {};
   \node [label = right:$\clOX$] (T0) at (5,29.5) {};
   \node [label = left:$\clXI$] (T1) at (-5,29.5) {};
   \node [label = above:$\clAll$] (Omega) at (0,31) {};
   \draw [thick] (Ic) -- (Istar) to[out=135,in=-135] (Omega1);
   \draw [thick] (I) -- (Omega1);
   \draw [thick] (Omega1) to[out=135,in=-135] (L);
   \draw [thick] (Ic) -- (I0) -- (I);
   \draw [thick] (Ic) -- (I1) -- (I);
   \draw [thick] (Ic) to[out=128,in=-134] (Lc);
   \draw [thick] (Ic) to[out=58,in=-58] (SM);
   \draw [thick] (I0) -- (L0);
   \draw [thick] (I1) -- (L1);
   \draw [thick] (Istar) to[out=60,in=-60] (LS);
   \draw [thick] (Ic) -- (Lamc);
   \draw [thick] (I0) -- (Lam0);
   \draw [thick] (I1) -- (Lam1);
   \draw [thick] (I) -- (Lam);
   \draw [thick] (Ic) -- (Vc);
   \draw [thick] (I0) -- (V0);
   \draw [thick] (I1) -- (V1);
   \draw [thick] (I) -- (V);
   \draw [thick] (Lamc) -- (Lam0) -- (Lam);
   \draw [thick] (Lamc) -- (Lam1) -- (Lam);
   \draw [thick] (Lamc) -- (McUi);
   \draw [thick] (Lam0) -- (MUi);
   \draw [thick] (Lam1) -- (M1);
   \draw [thick] (Lam) -- (M);
   \draw [thick] (Vc) -- (V0) -- (V);
   \draw [thick] (Vc) -- (V1) -- (V);
   \draw [thick] (Vc) -- (McWi);
   \draw [thick] (V0) -- (M0);
   \draw [thick] (V1) -- (MWi);
   \draw [thick] (V) -- (M);
   \draw [thick] (McUi) -- (TcUi) -- (Ui);
   \draw [thick] (McUi) -- (MUi) -- (Ui);
   \draw [thick,loosely dashed] (McUi) -- (McU3);
   \draw [thick,loosely dashed] (MUi) -- (MU3);
   \draw [thick,loosely dashed] (TcUi) -- (TcU3);
   \draw [thick,loosely dashed] (Ui) -- (U3);
   \draw [thick] (McU3) -- (TcU3) -- (U3);
   \draw [thick] (McU3) -- (MU3) -- (U3);
   \draw [thick] (McU3) -- (McU2);
   \draw [thick] (MU3) -- (MU2);
   \draw [thick] (TcU3) -- (TcU2);
   \draw [thick] (U3) -- (U2);
   \draw [thick] (McU2) -- (TcU2) -- (U2);
   \draw [thick] (McU2) -- (MU2) -- (U2);
   \draw [thick] (McU2) -- (Mc);
   \draw [thick] (MU2) -- (M0);
   \draw [thick] (TcU2) to[out=120,in=-25] (Tc);
   \draw [thick] (U2) -- (T0);

   \draw [thick] (McWi) -- (TcWi) -- (Wi);
   \draw [thick] (McWi) -- (MWi) -- (Wi);
   \draw [thick,loosely dashed] (McWi) -- (McW3);
   \draw [thick,loosely dashed] (MWi) -- (MW3);
   \draw [thick,loosely dashed] (TcWi) -- (TcW3);
   \draw [thick,loosely dashed] (Wi) -- (W3);
   \draw [thick] (McW3) -- (TcW3) -- (W3);
   \draw [thick] (McW3) -- (MW3) -- (W3);
   \draw [thick] (McW3) -- (McW2);
   \draw [thick] (MW3) -- (MW2);
   \draw [thick] (TcW3) -- (TcW2);
   \draw [thick] (W3) -- (W2);
   \draw [thick] (McW2) -- (TcW2) -- (W2);
   \draw [thick] (McW2) -- (MW2) -- (W2);
   \draw [thick] (McW2) -- (Mc);
   \draw [thick] (MW2) -- (M1);
   \draw [thick] (TcW2) to[out=60,in=-155] (Tc);
   \draw [thick] (W2) -- (T1);

   \draw [thick] (SM) -- (McU2);
   \draw [thick] (SM) -- (McW2);

   \draw [thick] (Lc) -- (LS) -- (L);
   \draw [thick] (Lc) -- (L0) -- (L);
   \draw [thick] (Lc) -- (L1) -- (L);
   \draw [thick] (Lc) to[out=120,in=-120] (Sc);
   \draw [thick] (LS) to[out=60,in=-60] (S);
   \draw [thick] (L0) -- (T0);
   \draw [thick] (L1) -- (T1);
   \draw [thick] (L) to[out=125,in=-125] (Omega);
   \draw [thick] (SM) -- (Sc) -- (S);
   \draw [thick] (Sc) to[out=142,in=-134] (Tc);
   \draw [thick] (S) to[out=42,in=-42] (Omega);
   \draw [thick] (Mc) -- (M0) -- (M);
   \draw [thick] (Mc) -- (M1) -- (M);
   \draw [thick] (Mc) to[out=120,in=-120] (Tc);
   \draw [thick] (M0) -- (T0);
   \draw [thick] (M1) -- (T1);
   \draw [thick] (M) to[out=55,in=-55] (Omega);
   \draw [thick] (Tc) -- (T0) -- (Omega);
   \draw [thick] (Tc) -- (T1) -- (Omega);
\end{tikzpicture}
}
\end{center}
\caption{Post's lattice.}
\label{fig:Post}
\end{figure}

\subsection{Threshold functions and generators of near\hyp{}unanimity clones}
\label{subsec:gen-nu}

\begin{definition}
A Boolean function $f \colon \{0,1\}^n \to \{0,1\}$ is called a \emph{threshold function} (or a \emph{linearly separable function}) if there exist \emph{weights} $w_1, \dots, w_n \in \mathbb{R}$ and a \emph{threshold} $t \in \mathbb{R}$ such that $f(a_1, \dots, a_n) = 1$ if and only if $\sum_{i=1}^n w_i a_i \geq t$ (see, e.g., Muroga~\cite{Muroga}).
We use the notation $\threshold{n}{t}$ for the $n$\hyp{}ary threshold function with weights $w_1 = w_2 = \dots = w_n = 1$ and threshold $t$.
Thus, $\threshold{n}{t}(a_1, \dots, a_n) = 1$ if and only if $\card{\{ \, i \in \nset{n} \mid a_i = 1 \,\}} \geq t$.
\end{definition}

For $1 < t \leq n-1$, the threshold function $\threshold{n}{t}$ is a near\hyp{}unanimity operation.
These functions can be used to generate the clones of $0$- or $1$\hyp{}separating functions.
In what follows, we will make use of the following generating sets of clones.

\begin{itemize}
\item $\clMcUk{k} = \clonegen{\threshold{k+1}{k}}$ for $k \geq 3$,
      $\clMcUk{2} = \clonegen{\threshold{3}{2}, x \wedge (y \vee z)}$
\item $\clMcUk{\infty} = \clonegen{x \wedge (y \vee z)}$
\item $\clMcWk{k} = \clonegen{\threshold{k+1}{2}}$ for $k \geq 3$,
      $\clMcWk{2} = \clonegen{\threshold{3}{2}, x \vee (y \wedge z)}$
\item $\clMcWk{\infty} = \clonegen{x \vee (y \wedge z)}$
\item $\clMUk{k} = \clonegen{\threshold{k+1}{k}, 0}$ for $k \geq 2$
\item $\clMUk{\infty} = \clonegen{x \wedge (y \vee z), 0}$
\item $\clMWk{k} = \clonegen{\threshold{k+1}{2}, 1}$ for $k \geq 2$
\item $\clMWk{\infty} = \clonegen{x \vee (y \wedge z), 1}$
\item $\clTcUk{k} = \clonegen{\threshold{k+1}{k}, x \wedge (y \rightarrow z)}$ for $k \geq 2$
\item $\clTcUk{\infty} = \clonegen{x \wedge (y \rightarrow z)}$
\item $\clTcWk{k} = \clonegen{\threshold{k+1}{2}, x \vee (y \nrightarrow z)}$ for $k \geq 2$
\item $\clTcWk{\infty} = \clonegen{x \vee (y \nrightarrow z)}$
\item $\clUk{k} = \clonegen{\threshold{k+1}{k}, \mathord{\nrightarrow}}$ for $k \geq 2$
\item $\clUk{\infty} = \clonegen{\mathord{\nrightarrow}}$
\item $\clWk{k} = \clonegen{\threshold{k+1}{2}, \mathord{\rightarrow}}$ for $k \geq 2$
\item $\clWk{\infty} = \clonegen{\mathord{\rightarrow}}$
\end{itemize}


\section{Minorant minions and $\ell$\hyp{}local closures}
\label{sec:minmin}

In this section, we introduce the most important notions that we need for describing the $(\clIc,\clMcUk{k})$\hyp{}clonoids: $\ell$\hyp{}locally closed minorant minions.
As it turns out, such function classes are $(\clIc,\clMcUk{k})$\hyp{}clonoids whenever $\ell \leq k \leq \infty$ and $k \geq 2$ (see Proposition~\ref{prop:UTheta-stability}).
We will use the notation $\clKlik{\ell}{\Theta}$ for the $\ell$\hyp{}local closure of the minorant minion ${\downarrow} \Theta$ (see Definition~\ref{def:UkTheta}), and we develop tools for working efficiently with these classes.

\subsection{Minorant minions}

\begin{definition}
Let $f \in \clAll$.
The elements of $f^{-1}(1)$ are called the \emph{true points} of $f$, and the elements of $f^{-1}(0)$ are called the \emph{false points} of $f$.
For $T \subseteq \{0,1\}^n$, the unique $n$\hyp{}ary Boolean function whose set of true points is $T$ is called the \emph{characteristic function} of $T$ and is denoted by $\chi_T$.
\end{definition}

\begin{definition}
\label{def:minmin}
The \emph{minorant\hyp{}minor} relation $\minmin$ of Boolean functions is the transitive closure of the union of the minorant and the minor relations on $\clAll$, in symbols, $\mathord{\minmin} = (\mathord{\minorant} \cup \mathord{\minor})^\mathrm{tr}$.
The minorant\hyp{}minor relation $\minmin$ is a quasiorder (a reflexive and transitive relation) on $\clAll$; it is reflexive because both $\minorant$ and $\minor$ are, and it is transitive by definition.
It induces an equivalence relation $\eqminmin$ on $\clAll$ and a partial order $\minmin$ on $\clAll / \mathord{\eqminmin}$ in the usual way: $f \eqminmin g$ if and only if $f \minmin g$ and $g \minmin f$, and $f / \mathord{\eqminmin} \minmin g / \mathord{\eqminmin}$ if and only if $f \minmin g$.

A downset of $(\clAll, \mathord{\minmin})$ is called a \emph{minorant minion}.
By definition, a minorant minion is a function class that is closed under minors and minorants, i.e., it is a minorant\hyp{}closed minion.
\end{definition}

The following lemma will be very helpful for working with the minorant\hyp{}minor relation.
(Note that we compose relations from left to right, so that $x \mathrel{(R \circ S)} y$ if and only if there is a $z$ such that $x \mathrel{R} z \mathrel{S} y$.)

\begin{lemma}
\label{lem:minmin}
\leavevmode
\begin{enumerate}[label={\upshape(\roman*)}]
\item\label{lem:minmin:move}
$\mathord{\leq} \circ \mathord{\minorant} \subseteq \mathord{\minorant} \circ \mathord{\minor}$.
\item\label{lem:minmin:comp}
$\mathord{\minmin} = \mathord{\minorant} \circ \mathord{\minor}$.
\item\label{lem:minmin:cond}
For Boolean functions $f$ and $g$ of arities $n$ and $m$, respectively,
$f \minmin g$ if and only if there exists a $\sigma \colon \nset{m} \to \nset{n}$ such that $f \minorant g_\sigma$.
\end{enumerate}
\end{lemma}

\begin{proof}
\ref{lem:minmin:move}
Assume $f \mathrel{(\mathord{\minor} \circ \mathord{\minorant})} g$.
Then there exists an $h$ such that $f \minor h \minorant g$.
Since $h \minorant g$, we have $h(\vect{b}) \leq g(\vect{b})$ for all $\vect{b} \in \{0,1\}^m$, and
since $f \leq h$, there exists a $\tau \colon \nset{m} \to \nset{n}$ such that $f(\vect{a}) = h(\vect{a} \tau)$ for all $\vect{a} \in \{0,1\}^n$.
Then $f(\vect{a}) = h(\vect{a} \tau) \leq g(\vect{a} \tau) = g_\tau(\vect{a})$ for all $\vect{a} \in \{0,1\}^n$, which means that $f \minorant g_\tau \leq g$, i.e., $f \mathrel{(\mathord{\minorant} \circ \mathord{\leq})} g$.

\ref{lem:minmin:comp}
By definition, $\mathord{\minmin} = (\mathord{\minorant} \cup \mathord{\minor})^\mathrm{tr}$, so the inclusion $\mathord{\minorant} \circ \mathord{\minor} \subseteq \mathord{\minmin}$ is obvious.
Assume now that $f \minmin g$.
Then $f \mathrel{(\beta_1 \circ \beta_2 \circ \dots \circ \beta_r)} g$ for some $r \in \IN_{+}$, where each $\beta_i$ is either $\minorant$ or $\minor$.
By the reflexivity of partial orders, we may assume that each of $\minor$ and $\minorant$ occurs at least once.
Because $\mathord{\minor} \circ \mathord{\minorant} \subseteq \mathord{\minorant} \circ \mathord{\minor}$ by part \ref{lem:minmin:move}, we can bring all $\minorant$'s to the left and all $\minor$'s to the right.
By transitivity, we can remove repetitions.
It follows that $f \mathrel{(\mathord{\minorant} \circ \mathord{\minor})} g$, as claimed.

\ref{lem:minmin:cond}
This is an immediate consequence of part \ref{lem:minmin:comp}.
\end{proof}

We now present a few basic results on the structure of the minorant\hyp{}minor poset $(\clAll / \mathord{\eqminmin}, \mathord{\minmin})$.
The poset has a least and a greatest element, namely (the $\eqminmin$\hyp{}equivalence classes of) the constant functions $0$ and $1$, respectively (see Proposition~\ref{prop:minmin-10-atom}).
It has one atom but no coatoms (Propositions~\ref{prop:minmin-10-atom} and \ref{prop:minmin-asc-chain}).
It contains an infinite ascending chain and an infinite antichain (Propositions~\ref{prop:minmin-asc-chain} and \ref{prop:minmin-antichain}).
A schematic Hasse diagram of the minorant\hyp{}minor poset is shown in Figure~\ref{fig:minmin-poset}.

\begin{proposition}
\label{prop:minmin-10-atom}
The least and greatest elements of the minorant\hyp{}minor poset $(\clAll / \mathord{\eqminmin}, \mathord{\minmin})$ are $0$ and $1$, respectively.
The poset has one atom, namely $\nrightarrow$, and hence $\nrightarrow$ is a minorant\hyp{}minor of every function that is not equivalent to $0$.
\end{proposition}

\begin{proof}
It is easy to verify the statements about the least and greatest elements; one just needs to observe that the minors of a constant function are constant functions with the same value and that $0 \minorant f \minorant 1$ for every Boolean function (here the constant functions $0$ and $1$ have the same arity as $f$).

Let now $f$ be a function that is not constant $0$.
Then $f$ has a true point $\vect{u} = (u_1, \dots, u_n)$.
By introducing a fictitious argument, if necessary, we may assume that $\vect{u} \notin \{\vect{0}, \vect{1}\}$.
By identifying, on the one hand, the arguments corresponding to those positions $i$ for which $u_i = 1$, and, on the other hand, those for which $u_i = 0$, we obtain a binary minor of $f$ in which $10$ is a true point; $\nrightarrow$ is a minorant of this function, so $\mathord{\nrightarrow} \minmin f$.
From this it also follows that $\nrightarrow$ is an atom of $(\clAll / \mathord{\eqminmin}, \mathord{\minmin})$.
\end{proof}

\begin{figure}
\scalebox{0.6}{%
\tikzstyle{every node}=[circle, draw, fill=black, scale=1, font=\huge]
\pgfdeclarelayer{poset}
\pgfdeclarelayer{blobs}
\pgfsetlayers{blobs,poset}
\begin{tikzpicture}
\begin{pgfonlayer}{poset}
\node (0) at (0,0) {};
\draw ($(0)+(0,-0.7)$) node[draw=none,fill=none]{$0$};
\node (atom) at (0,1) {};
\draw ($(atom)+(-0.7,0)$) node[draw=none,fill=none]{$\nrightarrow$};
\node (1) at (0,12) {};
\draw ($(1)+(0,0.7)$) node[draw=none,fill=none]{$1$};
\draw[thick] (0) -- (atom);
\node (f3) at (-1.5,4.5) {};
\draw ($(f3)+(0,-0.7)$) node[draw=none,fill=none]{$f_3$};
\node (f4) at (-0.5,4.5) {};
\draw ($(f4)+(0,-0.7)$) node[draw=none,fill=none]{$f_4$};
\node (f5) at (0.5,4.5) {};
\draw ($(f5)+(0,-0.7)$) node[draw=none,fill=none]{$f_5$};
\node[draw=none,fill=none] (fdots) at (1.5,4.5) {$\cdots$};
\node (v2) at (-2,6) {};
\draw ($(v2)+(0.7,0)$) node[draw=none,fill=none]{$\vee_2$};
\node (v3) at (-1.75,7) {};
\draw ($(v3)+(0.7,0)$) node[draw=none,fill=none]{$\vee_3$};
\node (v4) at (-1.5,8) {};
\draw ($(v4)+(0.7,0)$) node[draw=none,fill=none]{$\vee_4$};
\node (v5) at (-1.25,9) {};
\draw ($(v5)+(0.7,0)$) node[draw=none,fill=none]{$\vee_5$};
\node[draw=none,fill=none] (vdots) at (-1,10) {\begin{turn}{75.9}$\cdots$\end{turn}};
\draw[thick] (v2) -- (v3) -- (v4) -- (v5);
\node (A2) at (2,6) {};
\draw ($(A2)+(-0.7,0)$) node[draw=none,fill=none]{$\mathord{\uparrow}_2$};
\node (A3) at (1.75,7) {};
\draw ($(A3)+(-0.7,0)$) node[draw=none,fill=none]{$\mathord{\uparrow}_3$};
\node (A4) at (1.5,8) {};
\draw ($(A4)+(-0.7,0)$) node[draw=none,fill=none]{$\mathord{\uparrow}_4$};
\node (A5) at (1.25,9) {};
\draw ($(A5)+(-0.7,0)$) node[draw=none,fill=none]{$\mathord{\uparrow}_5$};
\node[draw=none,fill=none] (Adots) at (1,10) {\begin{turn}{-75.9}$\cdots$\end{turn}};
\draw[thick] (A2) -- (A3) -- (A4) -- (A5);
\end{pgfonlayer}
\begin{pgfonlayer}{blobs}
\draw[thin,fill=blue!22] (atom.center) to[out=145,in=215] (1.center) to[out=325,in=35] (atom.center) -- cycle;
\end{pgfonlayer}
\end{tikzpicture}
}
\caption{Schematic Hasse diagram of the minorant\hyp{}minor poset $(\clAll / \mathord{\eqminmin}, \mathord{\minmin})$.}
\label{fig:minmin-poset}
\end{figure}

For $n \in \IN_{+}$, let $\mathord{\vee_n} := \pr_1^{(n)} \vee \dots \vee \pr_n^{(n)}$ and $\mathord{\uparrow_n} := \neg(\pr_1^{(1)} \wedge \dots \wedge \pr_n^{(n)})$.

\begin{proposition}
\label{prop:minmin-asc-chain}
\leavevmode
\begin{enumerate}[label={\upshape{(\roman*)}}]
\item\label{prop:minmin-asc-chain:1}
The families $\{\vee_n\}_{n \in \IN_{+}}$ and $\{\mathord{\uparrow}_n\}_{n \in \IN_{+}}$ are infinite ascending chains in the poset $(\clAll / \mathord{\eqminmin}, \mathord{\minmin})$.
\item\label{prop:minmin-asc-chain:2}
Furthermore, there is no coatom in $(\clAll / \mathord{\eqminmin}, \mathord{\minmin})$.
\end{enumerate}
\end{proposition}

\begin{proof}
\ref{prop:minmin-asc-chain:1}
Let $n, m \in \IN_{+}$ with $n < m$.
It is easy to verify that $\vee_n = (\vee_m)_\sigma$ with $\sigma \colon \nset{m} \to \nset{n}$, $i \mapsto \min(i,n)$, so $\vee_n \minor \vee_m$ and hence $\vee_n \minmin \vee_m$.

Suppose now, to the contrary, that $\vee_m \minmin \vee_n$.
Then, by Lemma~\ref{lem:minmin}, $\vee_m \minorant (\vee_n)_\sigma$ for some $\sigma \colon \nset{n} \to \nset{m}$.
Since $n < m$, $\sigma$ is not surjective, so there is an element $p \in \nset{m}$ that is not in the range of $\sigma$, and we have $\vect{e}_p \sigma = \vect{0}$.
Then $\vee_m(\vect{e}_p) = 1$ but $\vee_n(\vect{e}_p \sigma) = \vee_n(\vect{0}) = 0$, a contradiction.

This shows that $\vee_n \minmin \vee_m$ if and only if $n \leq m$.
A similar argument shows that $\mathord{\uparrow}_n \minmin \mathord{\uparrow}_m$ if and only if $n \leq m$.

\ref{prop:minmin-asc-chain:2}
It is clear that every function in $\clII$ is $\minmin$\hyp{}equivalent to $1$.
On the other hand, every at most $n$\hyp{}ary function in $\clOX$ (in $\clXO$, resp.)\ is a minorant\hyp{}minor of $\vee_n$ (of $\mathord{\uparrow}_n$, resp.)\ and is hence below one of the elements of the two infinite ascending chains of part \ref{prop:minmin-asc-chain:1}.
Therefore, $(\clAll / \mathord{\eqminmin}, \mathord{\minmin})$ has no coatom.
\end{proof}

\begin{proposition}
\label{prop:minmin-antichain}
The minorant\hyp{}minor poset $(\clAll / \mathord{\eqminmin}, \mathord{\minmin})$ contains an infinite antichain.
Consequently, there are uncountably many minorant minions of Boolean functions.
\end{proposition}

\begin{proof}
For $n \geq 3$, let $f_n \colon \{0,1\}^n \to \{0,1\}$ be defined by the rule $f_n(\vect{a}) = 1$ if and only if $w(\vect{a}) \in \{1, n-1\}$, where $w(\vect{a})$ is the \emph{Hamming weight} of $\vect{a}$, i.e., $w(a_1, \dots, a_n) = \card{\{ i \in \nset{n} \mid a_i \neq 0\}}$.

We claim that for all $n, m \in \nset{N}$ with $n, m \geq 3$, we have $f_n \minmin f_m$ if and only if $n = m$.
Sufficiency is obvious. It remains to prove necessity.
Assume that $f_n \minmin f_m$.
Then $f_n \minorant (f_m)_\sigma$ for some $\sigma \colon \nset{m} \to \nset{n}$.
Suppose first, to the contrary, that $\sigma$ is not surjective; say, $p \in \nset{n}$ has no preimage under $\sigma$.
Then $\vect{e}_p \sigma = \vect{0}$, and we have $f_n(\vect{e}_p) = 1$ but $(f_m)_\sigma(\vect{e}_p) = f_m(\vect{e}_p \sigma) = f_m(\vect{0}) = 0$, a contradiction.
Therefore, $\sigma$ must be surjective.

Suppose now, to the contrary, that $\sigma$ is not injective; say, $r, s \in \nset{m}$, $q \in \nset{n}$ are elements such that $r \neq s$ and $\sigma(r) = \sigma(s) = q$.
Since $\sigma$ is surjective, there are at least $n-1$ elements $i \in \nset{m}$ such that $\sigma(i) \neq q$.
Then $2 \leq \card{\sigma^{-1}(q)} \leq m - (n-1) \leq m - 2$ because $n \geq 3$, and we have $f_n(\vect{e}_q) = 1$ but, because $w(\vect{e}_q \sigma) = \card{\sigma^{-1}(q)}$, we have $(f_m)_\sigma(\vect{e}_q) = f_m(\vect{e}_q \sigma) = 0$, a contradiction.

Therefore, $\sigma$ is a bijection, and, consequently, $n = m$.
We conclude that the functions $f_n$ ($n \geq 3$) constitute an infinite antichain in the minorant\hyp{}minor poset.
It follows that there are uncountably many minorant minions of Boolean functions, because $\{\, f_n \mid n \geq 3 \,\}$ has uncountably many subsets and distinct subsets thereof generate distinct minorant minions.
\end{proof}

\begin{remark}
The same family of functions $f_n$ as in the proof of Proposition~\ref{prop:minmin-antichain} was used by Pippenger~\cite[Proposition~3.4]{Pippenger} to show that the minor poset of Boolean functions contains an infinite antichain and, consequently, that there are uncountably many minions.
\end{remark}

\subsection{$\ell$\hyp{}local closures}

\begin{definition}
Let $\Theta \subseteq \clAll$.
For $\ell \in \IN \cup\{\infty\}$, the \emph{$\ell$\hyp{}local closure} of $\Theta$ is
\[
\kLoc[\ell](\Theta) := \{ \, f \in \clAll \mid \forall S \subseteq \dom(f) \, (\card{S} \leq \ell \rightarrow \exists g \in \Theta \, (f|_S = g|_S )) \, \}.
\]
\end{definition}

\begin{lemma}
\label{lem:lLoc-closureop}
For every $\ell \in \IN \cup \{\infty\}$,
the mapping $\kLoc[\ell] \colon \mathcal{P}(\clAll) \to \mathcal{P}(\clAll)$ is a closure operator on $\clAll$.
\end{lemma}

\begin{proof}
We need to show that $\kLoc[\ell]$ is extensive, order\hyp{}preserving, and idempotent.

\emph{Extensivity:}
Let $X \subseteq \clAll$.
For every $f \in X$ and for every $S \subseteq \dom(f)$ with $\card{S} \leq \ell$, we clearly have $f|_S = f|_S$; hence, $f \in \kLoc[\ell](X)$.
Therefore, $X \subseteq \kLoc[\ell](X)$.

\emph{Order preservation:}
Assume $X, Y \subseteq \clAll$ satisfy $X \subseteq Y$.
If $f \in \kLoc[\ell](X)$, then for all $S \subseteq \dom(f)$ with $\card{S} \leq \ell$, there exists a $g \in X$ such that $f|_S = g|_S$.
Since $X \subseteq Y$, the function $g$ mentioned above belongs to $Y$, and we conclude that $f \in \kLoc[\ell](Y)$.
Therefore, $\kLoc[\ell](X) \subseteq \kLoc[\ell](Y)$.

\emph{Idempotence:}
Because $\kLoc[\ell]$ is extensive, we have $X \subseteq \kLoc[\ell](X)$, and because $\kLoc[\ell]$ is order\hyp{}preserving, it follows that $\kLoc[\ell](X) \subseteq \kLoc[\ell](\kLoc[\ell](X))$.

It remains to show that the converse inclusion $\kLoc[\ell](\kLoc[\ell](X)) \subseteq \kLoc[\ell](X)$ holds.
Let $f \in \kLoc[\ell](\kLoc[\ell](X))$.
In order to show that $f \in \kLoc[\ell](X)$, let $S \subseteq \dom(f)$ with $\card{S} \leq \ell$.
Since $f \in \kLoc[\ell](\kLoc[\ell](X))$, there exists a $g \in \kLoc[\ell](X)$ such that $g|_S = f|_S$.
Since $g \in \kLoc[\ell](X)$, there exists an $h \in X$ such that $h|_S = g|_S = f|_S$.
Therefore, $f \in \kLoc[\ell](X)$.
\end{proof}

\begin{definition}
For $\ell \in \IN \cup \{\infty\}$, let
\[
\clAllleq{\ell} :=
\{ \, f \in \clAll \mid \card{f^{-1}(1)} \leq \ell \,\},
\]
the set of Boolean functions with at most $\ell$ true points.
For $\Theta \subseteq \clAll$, let $\Theta^{[\leq \ell]} := \Theta \cap \clAllleq{\ell}$.
\end{definition}

\begin{lemma}
\label{lem:minmin-kLoc}
Let $\ell \in \IN$ with $\ell \geq 2$, let $\Theta \subseteq \clAll$, assume that $\Theta$ is closed under minorants, and let $f \in \clAll$.
Then the following conditions are equivalent.
\begin{enumerate}[label={\upshape(\roman*)}]
\item\label{lem:minmin-kLoc:1}
$f \in \kLoc[\ell](\Theta)$.
\item\label{lem:minmin-kLoc:2}
For all $T \subseteq f^{-1}(1)$ with $\card{T} \leq \ell$, we have $\chi_T \in \Theta$.
\item\label{lem:minmin-kLoc:3}
$f \in \kLoc[\ell](\Theta^{[\leq \ell]})$.
\end{enumerate}
\end{lemma}

\begin{proof}
\ref{lem:minmin-kLoc:1} $\implies$ \ref{lem:minmin-kLoc:2}
Let $T \subseteq f^{-1}(1)$ with $\card{T} \leq \ell$.
By the definition of $\ell$\hyp{}local closure, there exists a $g \in \Theta$ such that $f|_T = g|_T$.
It clearly holds that $\chi_T \minorant g$.
Since $\Theta$ is closed under minorants, we have $\chi_T \in \Theta$.

\ref{lem:minmin-kLoc:2} $\implies$ \ref{lem:minmin-kLoc:3}
Let $S \subseteq \dom(f)$ with $\card{S} \leq \ell$, and let $T := f^{-1}(1) \cap S$.
Clearly, $T \subseteq f^{-1}(1)$ and $\card{T} \leq \card{S} \leq \ell$, so $\chi_T \in \Theta$ by \ref{lem:minmin-kLoc:2}, and, in fact,
$\chi_T \in \Theta^{[\leq \ell]}$ because $\card{\chi_T^{-1}(1)} = \card{T} \leq \ell$.
Furthermore, $f|_S = (\chi_T)|_S$, and we conclude that $f \in \kLoc[\ell](\Theta^{[\leq \ell]})$.

\ref{lem:minmin-kLoc:3} $\implies$ \ref{lem:minmin-kLoc:1}
Since $\Theta^{[\leq \ell]} \subseteq \Theta$, we have $\kLoc[\ell](\Theta^{[\leq \ell]}) \subseteq \kLoc[\ell](\Theta)$ because $\kLoc[\ell]$ is monotone (see Lemma~\ref{lem:lLoc-closureop}).
\end{proof}

\subsection{$\ell$\hyp{}local closures of minorant minions}

Recall from Definition~\ref{def:separating} that a function $f$ belongs to $\clUk{\ell}$ if and only if for all $T \subseteq f^{-1}(1)$ with $\card{T} \leq \ell$,
we have $\bigwedge T \neq \vect{0}$, i.e.,
there is an $i \in \nset{n}$ such that $a_i = 1$ for all $\vect{a} = (a_1, \dots, a_n) \in T$.
Note that the part ``$a_i = 1$'' of this condition can be rewritten as ``$\id_i(\vect{a}) = 1$'', or, more explicitly, as ``$\id_{\tau_i}(\vect{a}) = \id(\vect{a} \tau_i) = 1$'', where $\tau_i \colon \nset{1} \to \nset{n}$ is the minor formation map $1 \mapsto i$.
This leads to an interesting natural generalization.

\begin{definition}
\label{def:UkTheta}
For $T \subseteq \{0,1\}^n$ and $\Theta \subseteq \clAll$, we write $T \tulee \Theta$ if there exists a $\theta \in \Theta$, say of arity $h$, and $\tau \colon \nset{h} \to \nset{n}$ such that $\theta_\tau(\vect{a}) = \theta(\vect{a} \tau) = 1$ for all $\vect{a} \in T$.
In other words, $T \tulee \Theta$ if there exists a $\theta \in \Theta$ such that $\chi_T \minmin \theta$.
Equivalently, $T \tulee \Theta$ if $\chi_T \in {\downarrow} \Theta$.

Let $\Theta \subseteq \clAll$, $\ell \in \IN \cup \{\infty\}$, and define $\clKlik{\ell}{\Theta}$ to be the set of all Boolean functions $f$ satisfying
$T \tulee \Theta$ (i.e., $\chi_T \in {\downarrow} \Theta$) for all $T \subseteq f^{-1}(1)$ with $\card{T} \leq \ell$.
More explicitly, $f \in \clKlik{\ell}{\Theta}$ if
for all $T \subseteq f^{-1}(1)$ with $\card{T} \leq \ell$, there exists a $\theta \in \Theta$ and $\tau \colon \nset{\arity \theta} \to \nset{n}$ such that $\theta_\tau(\vect{a}) = \theta(\vect{a} \tau) = 1$ for all $\vect{a} \in T$.
\end{definition}

Note that we obtain $\clUk{\ell}$ as a special instance of $\clKlik{\ell}{\Theta}$, namely, for $\ell \geq 2$, $\clUk{\ell} = \clKlik{\ell}{\{\id\}} = \clKlik{\ell}{\clIc}$.

The classes of the form $\clKlik{\ell}{\Theta}$ are precisely the $\ell$\hyp{}local closures of minorant minions.

\begin{proposition}
\label{prop:Uk-kLoc-equal}
Let $\Theta \subseteq \clAll$, and let $\ell \in \mathbb{N} \cup \{\infty\}$.
Then $\clKlik{\ell}{\Theta} = \kLoc[\ell]({\downarrow} \Theta)$.
\end{proposition}

\begin{proof}
Assume first that $f \in \clKlik{\ell}{\Theta}$.
Let $S \subseteq \dom(f)$ with $\card{S} \leq \ell$, and let $T := S \cap f^{-1}(1)$; we clearly have $f|_S = (\chi_T)|_S$.
Furthermore, since $\card{T} \leq \ell$, we have $\chi_T \in {\downarrow} \Theta$ by the definition of the class $\clKlik{\ell}{\Theta}$ (see Definition~\ref{def:UkTheta}).
We conclude that $f \in \kLoc[\ell]({\downarrow} \Theta)$.

Assume now that $f \in \kLoc[\ell]({\downarrow} \Theta)$.
Let $T \subseteq f^{-1}(1)$ with $\card{T} \leq \ell$.
Since $f \in \kLoc[\ell]({\downarrow} \Theta)$, there exists a $\varphi \in {\downarrow} \Theta$ such that $\varphi|_T = f|_T$.
But then $\chi_T \minorant \varphi$ and hence $\chi_T \in {\downarrow} \Theta$.
Therefore $f \in \clKlik{\ell}{\Theta}$.
\end{proof}

We are going to use the following fact in many proofs that follow.

\begin{lemma}
\label{lem:chiT-chiTsigma}
Let $n, m \in \IN_{+}$, $T \subseteq \{0,1\}^n$, and $\sigma \colon \nset{n} \to \nset{m}$, and let $T \sigma := \{ \, \vect{a} \sigma \mid \vect{a} \in T \,\}$.
Then $\chi_T \minorant (\chi_{T \sigma})_\sigma$, and consequently $\chi_T \minmin \chi_{T \sigma}$.
\end{lemma}

\begin{proof}
Let $\vect{a} \in \chi_T^{-1}(1) = T$.
Then $\vect{a} \sigma \in T \sigma$,
so $(\chi_{T \sigma})_\sigma (\vect{a}) = \chi_{T \sigma} (\vect{a} \sigma) = 1$.
Therefore $\chi_T \minorant (\chi_{T \sigma})_\sigma \minor \chi_{T \sigma}$, so $\chi_T \minmin \chi_{T \sigma}$.
\end{proof}

The classes $\clKlik{\ell}{\Theta}$ have remarkable closure properties, as revealed by the following proposition.
Note that it follows, in particular, that the $\ell$\hyp{}local closure of a minorant minion is a minorant minion.

\begin{proposition}
\label{prop:UTheta-stability}
Let $\ell \in \IN \cup \{\infty\}$ and $\Theta \subseteq \clAll$.
\begin{enumerate}[label={\upshape(\roman*)}]
\item\label{prop:UTheta-stability:minorant}
The class $\clKlik{\ell}{\Theta}$ is closed under taking minorants.

\item\label{prop:UTheta-stability:minor}
The class $\clKlik{\ell}{\Theta}$ is closed under taking minors.

\item\label{prop:UTheta-stability:left-stable}
For any $k \in \IN \cup \{\infty\}$ with $\max(2,\ell) \leq k$, the class $\clKlik{\ell}{\Theta}$ is stable under left composition with $\clMcUk{k}$.

\item\label{prop:UTheta-stability:stable}
For any $k \in \IN \cup \{\infty\}$ with $\max(2,\ell) \leq k$, the class $\clKlik{\ell}{\Theta}$ is a minorant\hyp{}closed $(\clIc,\clMcUk{k})$\hyp{}clonoid and, in particular, a minorant minion.
\end{enumerate}
\end{proposition}

\begin{proof}
\ref{prop:UTheta-stability:minorant}
Assume $f \in \clKlik{\ell}{\Theta}$, and let $g$ be a minorant of $f$.
Let $T \subseteq g^{-1}(1)$ with $\card{T} \leq \ell$.
Because $g \minorant f$, we have $T \subseteq f^{-1}(1)$, and so
we have $\chi_T \in {\downarrow} \Theta$ by the definition of $\clKlik{\ell}{\Theta}$.
We conclude that $g \in \clKlik{\ell}{\Theta}$.

\ref{prop:UTheta-stability:minor}
Let $f \in \clKlik{\ell}{\Theta}$, and let $g$ be a minor of $f$.
Then $g = f_\sigma$ for some $\sigma \colon \nset{n} \to \nset{m}$.
Let $T \subseteq g^{-1}(1)$ with $\card{T} \leq \ell$, and let $T \sigma := \{ \, \vect{a} \sigma \mid \vect{a} \in T \, \}$.
Then $1 = g(\vect{a}) = f_\sigma(\vect{a}) = f(\vect{a} \sigma)$, so $T \sigma \subseteq f^{-1}(1)$.
Since $\card{T \sigma} \leq \card{T} \leq \ell$, we have that $\chi_{T \sigma} \in {\downarrow} \Theta$.
Since $\chi_T \minmin \chi_{T \sigma}$ by Lemma~\ref{lem:chiT-chiTsigma}, it follows that $\chi_T \in {\downarrow} \Theta$ as well.
Therefore $g \in \clKlik{\ell}{\Theta}$.

\ref{prop:UTheta-stability:left-stable}
Because $\clMcUk{2} = \clonegen{\threshold{3}{2}, x \wedge (y \vee z)}$, $\clMcUk{k} = \clonegen{\threshold{k+1}{k}}$ for $3 \leq k < \infty$, and $\clMcUk{\infty} = \clonegen{x \wedge (y \vee z)}$, it suffices to show, by Lemma~\ref{lem:left-stable}, that  for all $f_1, \dots, f_{k+1} \in \clKlik{\ell}{\Theta}$ we have $\threshold{k+1}{k}(f_1, \dots, f_{k+1}) \in \clKlik{\ell}{\Theta}$ and  for all $f_1, f_2, f_3 \in \clKlik{\ell}{\Theta}$ we have $(x_1 \wedge (x_2 \vee x_3))(f_1, f_2, f_3) \in \clKlik{\ell}{\Theta}$.

Let $f_1, \dots, f_{k+1} \in \clKlik{\ell}{\Theta}$, and let $\varphi := \threshold{k+1}{k}(f_1, \dots, f_{k+1})$.
Let $T = \{ \vect{a}_1, \dots, \vect{a}_s \} \subseteq \varphi^{-1}(1)$ with $s \leq \ell$.
Then, for each $j \in \nset{s}$, at least $k$ of the values $f_1(\vect{a}_j), \dots, f_{k+1}(\vect{a}_j)$ are equal to $1$.
Since $s \leq \ell \leq k$, there exists a $p \in \nset{k+1}$ such that $f_p(\vect{a}_1) = f_p(\vect{a}_2) = \dots = f_p(\vect{a}_s) = 1$.
Thus $\chi_T \minorant f_p$.
Since $f_p \in \clKlik{\ell}{\Theta}$ and $\card{T} \leq \ell$, we have $\chi_T \in {\downarrow} \Theta$.
We conclude that $\varphi \in \clKlik{\ell}{\Theta}$.

Now let $f_1, f_2, f_3 \in \clKlik{\ell}{\Theta}$, and let $\psi := (x_1 \wedge (x_2 \vee x_3))(f_1, f_2, f_3)$.
Let $T \subseteq \psi^{-1}(1)$ with $\card{T} \leq \ell$.
Clearly, for $\psi(\vect{a}) = 1$, it is necessary that $f_1(\vect{a}) = 1$;
hence $T \subseteq f_1^{-1}(1)$.
Since $f_1 \in \clKlik{\ell}{\Theta}$,
we have $\chi_T \in {\downarrow} \Theta$.
Therefore $\psi \in \clKlik{\ell}{\Theta}$.

\ref{prop:UTheta-stability:stable}
This is a corollary of parts \ref{prop:UTheta-stability:minorant}--\ref{prop:UTheta-stability:left-stable}.
\end{proof}

\begin{definition}
In what follows,
we are going to make use of the Boolean functions defined in Table~\ref{table:BFs}.
The functions are specified by listing their true points.
Note that each of these functions has at most three true points.
\end{definition}

\begin{table}
\begingroup
\small
\begin{tabular}{ccccccc}
\toprule
$\begin{array}[t]{c} 0 \\ \midrule \text{---} \end{array}$
&
$\begin{array}[t]{c} \id \\ \midrule 1 \end{array}$
&
$\begin{array}[t]{c} \neg \\ \midrule 0 \end{array}$
&
$\begin{array}[t]{c} \nrightarrow \\ \midrule 10 \end{array}$
&
$\begin{array}[t]{c} 1 \\ \midrule 1 \\ 0 \end{array}$
&
$\begin{array}[t]{c} + \\ \midrule 01 \\ 10 \end{array}$
\\
\midrule
$\begin{array}[t]{c} \lambda_{30} \\ \midrule 001 \\ 010 \end{array}$
&
$\begin{array}[t]{c} \lambda_{31} \\ \midrule 110 \\ 101 \end{array}$
&
$\begin{array}[t]{c} \vee \\ \midrule 11 \\ 01 \\ 10 \end{array}$
&
$\begin{array}[t]{c} \uparrow \\ \midrule 00 \\ 01 \\ 10 \end{array}$
&
$\begin{array}[t]{c} \Gamma_0 \\ \midrule 000 \\ 110 \\ 101 \end{array}$
&
$\begin{array}[t]{c} \Gamma_1 \\ \midrule 111 \\ 001 \\ 010 \end{array}$
&
$\begin{array}[t]{c} \Gamma_{01} \\ \midrule 1001 \\ 0111 \\ 0010 \end{array}$
\\
\midrule
$\begin{array}[t]{c} \delta_0 \\ \midrule 110 \\ 101 \\ 011 \end{array}$
&
$\begin{array}[t]{c} \delta_0^0 \\ \midrule 1100 \\ 1010 \\ 0110 \end{array}$
&
$\begin{array}[t]{c} \delta_0^1 \\ \midrule 1101 \\ 1011 \\ 0111 \end{array}$
&
$\begin{array}[t]{c} \delta_0^{+} \\ \midrule 1101 \\ 1010 \\ 0110 \end{array}$
&
$\begin{array}[t]{c} \delta_0^{0{+}} \\ \midrule 11010 \\ 10100 \\ 01100 \end{array}$
&
$\begin{array}[t]{c} \lambda_{10} \\ \midrule 10011 \\ 01101 \\ 00110 \end{array}$
&
$\begin{array}[t]{c} \lambda_{10}^0 \\ \midrule 100110 \\ 011010 \\ 001100 \end{array}$
\\
\midrule
$\begin{array}[t]{c} \delta_1 \\ \midrule 001 \\ 010 \\ 100 \end{array}$
&
$\begin{array}[t]{c} \delta_1^0 \\ \midrule 0010 \\ 0100 \\ 1000 \end{array}$
&
$\begin{array}[t]{c} \delta_1^1 \\ \midrule 0011 \\ 0101 \\ 1001 \end{array}$
&
$\begin{array}[t]{c} \delta_1^{+} \\ \midrule 0010 \\ 0101 \\ 1001 \end{array}$
&
$\begin{array}[t]{c} \delta_1^{1{+}} \\ \midrule 00101 \\ 01011 \\ 10011 \end{array}$
&
$\begin{array}[t]{c} \lambda_{11} \\ \midrule 10001 \\ 01010 \\ 00111 \end{array}$
&
$\begin{array}[t]{c} \lambda_{11}^1 \\ \midrule 100011 \\ 010101 \\ 001111 \end{array}$
\\
\midrule
$\begin{array}[t]{c} \lambda_2 \\ \midrule 100011 \\ 010101 \\ 001110 \end{array}$
&
$\begin{array}[t]{c} \lambda_2^0 \\ \midrule 1000110 \\ 0101010 \\ 0011100 \end{array}$
&
$\begin{array}[t]{c} \lambda_2^1 \\ \midrule 1000111 \\ 0101011 \\ 0011101 \end{array}$
&
$\begin{array}[t]{c} A_0 \\ \midrule 110 \\ 100 \\ 010 \end{array}$
&
$\begin{array}[t]{c} A_0^{+} \\ \midrule 1110 \\ 0100 \\ 0010 \end{array}$
&
$\begin{array}[t]{c} A_1 \\ \midrule 001 \\ 011 \\ 101 \end{array}$
&
$\begin{array}[t]{c} A_1^{+} \\ \midrule 0001 \\ 1011 \\ 1101 \end{array}$
\\
\bottomrule
\end{tabular}
\endgroup
\medskip
\caption{The Boolean functions with at most three true points, up to $\eqminmin$\hyp{}equivalence.}
\label{table:BFs}
\end{table}

\begin{example}
\label{ex:classes}
Many of the classes of Boolean functions introduced in Section~\ref{sec:Boolean} are of the form $\clKlik{\ell}{\Theta}$ for some $\ell \in \mathbb{N} \cup \{\infty\}$ and $\Theta \subseteq \clAll$.
The following equalities are easy to verify (see Table~\ref{table:BFs} for the notation):
\begin{align*}
& \clAll = \clKlik{1}{\{\id, \neg\}} = \clKlik{2}{\{1\}}, &
& \clEiii = \clKlik{2}{\{\id, \neg, \mathord{+}\}}, &
& \clOX = \clKlik{1}{\{\id\}}, &
& \clXO = \clKlik{1}{\{\neg\}}, \\
& \clOO = \clKlik{1}{\{\mathord{\nrightarrow}\}} = \clKlik{2}{\{\mathord{+}\}}, &
& \clSmin = \clKlik{2}{\{\id, \neg\}}, &
& \clSminOX = \clKlik{2}{\{\id, \lambda_{30}\}}, &
& \clSminXO = \clKlik{2}{\{\neg, \lambda_{31}\}}, \\
& \clSminOO = \clKlik{2}{\{\lambda_{30}, \lambda_{31}\}}, &
& \clUk{\ell} = \clKlik{\ell}{\{\id\}}, &
& \clWkneg{\ell} = \clKlik{\ell}{\{\neg\}}, &
& \clUkOO{\ell} = \clKlik{\ell}{\{\varphi_\ell\}}, \\
& \clWknegOO{\ell} = \clKlik{\ell}{\{\varphi'_\ell\}}, &
& \clUk{\ell} \cap \clWkneg{\ell} = \clKlik{\ell}{\{\mathord{\nrightarrow}\}}, &
& \clVako = \clKlik{1}{\{0\}}, &
& \clEmpty = \clKlik{1}{\emptyset},
\end{align*}
where $\varphi_\ell \colon \{0,1\}^{\ell + 1} \to \{0,1\}$ with $\varphi_\ell(a_1, \dots, a_{\ell + 1}) = 1$ if and only if $a_1 = 1$ and $(a_2, \dots, a_{\ell + 1}) \neq \vect{1}$
and
$\varphi'_\ell \colon \{0,1\}^{\ell + 1} \to \{0,1\}$ with $\varphi'_\ell(a_1, \dots, a_{\ell + 1}) = 1$ if and only if $a_1 = 0$ and $(a_2, \dots, a_{\ell + 1}) \neq \vect{0}$.
\end{example}

\begin{proposition}
Every Boolean function in $\clAllleq{3}$ is $\eqminmin$\hyp{}equivalent to one of the functions defined in Table~\ref{table:BFs}.
\end{proposition}

\begin{proof}
We leave the details of the proof as an exercise to the reader.
A possible approach to proving this is to think of the true points of a function as the rows of a matrix with entries in $\{0,1\}$.
The order of rows clearly does not matter, and we can also assume that the rows are pairwise distinct.
Neither does the order of columns matter -- functions obtained from each other by permutation of arguments are $\eqminmin$\hyp{}equivalent.
Furthermore, we may assume that the columns are pairwise distinct, because matrices with the same set of columns represent functions that are $\eqminmin$\hyp{}equivalent; we can eliminate identical columns by identifying arguments, and we can bring such columns back by introducing fictitious arguments and then taking a minorant.
Thus, in order to find a transversal of the $\eqminmin$\hyp{}equivalence classes in $\clAllleq{3}$, it suffices to consider, for each $n \in \nset{3}$, each subset of $\{0,1\}^n$, which would give us the columns of a matrix with $n$ rows.
The task is then to consider functions corresponding to such matrices and to eliminate possible duplicate representatives of the $\eqminmin$\hyp{}equivalence classes.
Some duplicates arise simply by permutation of rows and columns, but other equivalences may arise in more involved ways.
One such set of representatives (without duplicates) is given in Table~\ref{table:BFs}.

Note that if the matrix corresponding to the true points of $f$ has $n$ rows ($n \geq 1$) and $2^n$ distinct columns, then $f \eqminmin \mathord{\nrightarrow}$. (To see that $\mathord{\nrightarrow} \minmin f$, take any true point $\vect{a}$ of $f$ and identify those arguments $i$ with $a_i = 0$ and those arguments $j$ with $a_j = 1$. To see that $f \minmin \mathord{\nrightarrow}$, introduce $2^n - 2$ fictitious arguments and take a minorant.)
\end{proof}


\subsection{Working with the classes $\clKlik{\ell}{\Theta}$}

We now develop a few facts that help us work with classes of the form $\clKlik{\ell}{\Theta}$.
Our first lemma provides a basic inclusion relationship between classes of the form $\clKlik{\ell}{\Theta}$.

\begin{lemma}
\label{lem:kllk-inclusion}
Let $\ell, \ell' \in \IN \cup \{\infty\}$ and $\Theta, \Theta' \subseteq \clAll$.
If $\ell \geq \ell'$ and ${\downarrow} \Theta \subseteq {\downarrow} \Theta'$, then $\clKlik{\ell}{\Theta} \subseteq \clKlik{\ell'}{\Theta'}$.
\end{lemma}

\begin{proof}
Let $f \in \clKlik{\ell}{\Theta}$.
Let $T \subseteq f^{-1}(1)$ with $\card{T} \leq \ell'$.
Since $\ell' \leq \ell$ and $f \in \clKlik{\ell}{\Theta}$,
we have $T \tulee \Theta$.
Because ${\downarrow} \Theta \subseteq {\downarrow} \Theta'$, it follows that $T \tulee \Theta'$.
We conclude that
$f \in \clKlik{\ell'}{\Theta'}$.
\end{proof}

When considering a class of the form $\clKlik{\ell}{\Theta}$, we may assume that $\Theta$ contains only functions with at most $\ell$ true points.

\begin{lemma}
\label{lem:klik-atmostk}
Let $\ell \in \IN \cup \{\infty\}$ and $C \subseteq \clAll$.
Then the following conditions are equivalent.
\begin{enumerate}[label={\upshape(\roman*)}]
\item\label{lem:klik-atmostk-any}
$C = \clKlik{\ell}{\Theta}$ for some $\Theta \subseteq \clAll$.
\item\label{lem:klik-atmostk-k}
$C = \clKlik{\ell}{\Psi}$ for some $\Psi \subseteq \clAll^{[\leq \ell]}$.
\end{enumerate}
\end{lemma}

\begin{proof}
\ref{lem:klik-atmostk-k} $\implies$ \ref{lem:klik-atmostk-any}: Trivial.

\ref{lem:klik-atmostk-any} $\implies$ \ref{lem:klik-atmostk-k}:
Assume $C = \clKlik{\ell}{\Theta}$, where $\Theta \subseteq \clAll$.
Let $\Psi$ be the set of all minorants of members of $\Theta$ with at most $\ell$ true points, i.e.,
\[
\Psi :=
\{ \, \psi \in \clAll \mid \exists \theta \in \Theta \colon \text{$\psi \minorant \theta$ and $\card{\psi^{-1}(1)} \leq \ell$} \, \}.
\]
We have $\Psi \subseteq \clAll^{[\leq \ell]}$ by construction.

We need to show that $\clKlik{\ell}{\Theta} = \clKlik{\ell}{\Psi}$.
Since ${\downarrow} \Psi \subseteq {\downarrow} \Theta$ by definition, we have $\clKlik{\ell}{\Psi} \subseteq \clKlik{\ell}{\Theta}$ by Lemma~\ref{lem:kllk-inclusion}.
For the converse inclusion,
assume that $f \in \clKlik{\ell}{\Theta}$, say $f$ is $n$\hyp{}ary.
Let $T \subseteq f^{-1}(1)$ with $\card{T} \leq \ell$.
Then there exists a $\theta \in \Theta$ and $\tau \colon \nset{\arity \theta} \to \nset{n}$ such that $\theta(\vect{a} \tau) = 1$ for all $\vect{a} \in T$.
Let $T \tau := \{ \, \vect{a} \tau \mid \vect{a} \in T \, \}$.
Note that $\chi_{T \tau} \minorant \theta$ and $\card{T \tau} \leq \card{T} \leq \ell$.
Therefore, we have $\chi_{T \tau} \in \Psi$ and $\chi_{T \tau}(\vect{a} \tau) = 1$ for all $\vect{a} \in T$.
We conclude that $f \in \clKlik{\ell}{\Psi}$; thus $\clKlik{\ell}{\Theta} \subseteq \clKlik{\ell}{\Psi}$.
\end{proof}

By Proposition~\ref{prop:UTheta-stability}\ref{prop:UTheta-stability:stable}, a set $\clKlik{\ell}{\Theta}$ is a $(\clIc,\clMcUk{k})$\hyp{}clonoid whenever $\Theta \subseteq \clAll$ and $\max(2,\ell) \leq k$.
In order to describe $(\clIc,\clMcUk{k})$\hyp{}clonoids of this form, we can, in fact, assume that $\ell = k$.

\begin{lemma}
\label{lem:klik-k}
Let $k, \ell \in \IN \cup \{\infty\}$ with $\max(2,\ell) \leq k$.
For any $\Theta \subseteq \clAll$, there exists a $\Psi \subseteq \clAll$ such that $\clKlik{\ell}{\Theta} = \clKlik{k}{\Psi}$.
\end{lemma}

\begin{proof}
Assume that $\max(2,\ell) \leq k$ and $\Theta \subseteq \clAll$.
By Lemma~\ref{lem:klik-atmostk}, we may furthermore assume that $\Theta \subseteq \clAll^{[\leq \ell]}$.
Let $\Psi$ be the set of all Boolean functions $\psi$ such that $\card{\psi^{-1}(1)} \leq k$ and for all $T \subseteq \psi^{-1}(1)$ with $\card{T} \leq \ell$ we have $T \tulee \Theta$.
We need to show that $\clKlik{\ell}{\Theta} = \clKlik{k}{\Psi}$.

Assume first that $f \in \clKlik{\ell}{\Theta}$, say $f$ is $n$\hyp{}ary.
Let $T \subseteq f^{-1}(1)$ with $\card{T} \leq k$.
Since $f \in \clKlik{\ell}{\Theta}$, we have that $S \tulee \Theta$ for all $S \subseteq T$ with $\card{S} \leq \ell$.
By the definition of $\Psi$, we have $\chi_T \in \Psi$.
Consequently, $T \tulee \Psi$.
We conclude that $f \in \clKlik{k}{\Psi}$.
This shows that $\clKlik{\ell}{\Theta} \subseteq \clKlik{k}{\Psi}$.

Assume now that $f \in \clKlik{k}{\Psi}$.
Let $T \subseteq f^{-1}(1)$ with $\card{T} \leq \ell$.
Since $f \in \clKlik{k}{\Psi}$, we have that $S \tulee \Psi$ for all $S \subseteq T$ with $\card{S} \leq \ell$; in particular, $T \tulee \Psi$ because $\ell \leq k$, so there are a $\psi \in \Psi$, say $h$\hyp{}ary, and $\sigma \colon \nset{h} \to \nset{n}$ such that $\psi(\vect{a} \sigma) = 1$ for all $\vect{a} \in T$.
Thus, $T \sigma := \{ \, \vect{a} \sigma \mid \vect{a} \in T \, \} \subseteq \psi^{-1}(1)$ and $\card{T \sigma} \leq \card{T} \leq \ell$.
By the definition of $\Psi$, this implies that $T \sigma \tulee \Theta$, i.e., there are a $\theta \in \Theta$, say $h'$\hyp{}ary, and $\tau \colon \nset{h'} \to \nset{h}$ such that $\theta(\vect{b} \tau) = 1$ for all $\vect{b} \in T \sigma$, that is, $\theta(\vect{a} \sigma \tau) = 1$ for all $\vect{a} \in T$.
But this means that $T \tulee \Theta$.
We conclude that $f \in \clKlik{\ell}{\Theta}$.
This shows that $\clKlik{k}{\Psi} \subseteq \clKlik{\ell}{\Theta}$.
\end{proof}

By putting together Lemmata~\ref{lem:klik-atmostk} and \ref{lem:klik-k}, we conclude that whenever we consider $(\clIc,\clMcUk{k})$\hyp{}clonoids of the form $\clKlik{\ell}{\Theta}$, we may assume that $\ell = k$ and $\Theta \subseteq \clAll^{[\leq k]}$.
We may furthermore assume that if $k \in \IN$, then the arity of each $\theta \in \Theta$ is at most $2^k$ (in fact, less than $2^k$ if $k \geq 2$).
This last assumption can be made for the following reason.
Assume that $\theta^{-1}(1) = \{\vect{a}_1, \dots, \vect{a}_k\}$ and consider the matrix $M$ whose rows are the $n$\hyp{}tuples $\vect{a}_1, \dots, \vect{a}_k$ (in an arbitrary order).
If some columns of $M$ are identical, then by identifying the arguments of $\theta$ corresponding to the identical columns we obtain a function $\theta'$ that is equivalent to $\theta$ with respect to the minorant\hyp{}minor order. Namely, we can get from $\theta'$ back to $\theta$ by introducing fictitious arguments and taking a minorant of the result.
Since there are $2^k$ different columns of length $k$, we conclude that it suffices to consider functions of arity at most $2^k$.
In fact, whenever $k \geq 2$, we can do with functions of arity at most $2^k - 1$, because if $M$ contains all $2^k$ possible different columns, then the function is actually equivalent to $\mathord{\nrightarrow}$, as can be easily checked.

Motivated by the above, we introduce some notation and concepts that help us study $(\clIc,\clMcUk{k})$\hyp{}clonoids of the form $\clKlik{k}{\Theta}$.

\begin{definition}
\label{def:Id-minmin}
Consider the poset $\posetAllleq{k} = (\clAllleq{k} / \mathord{\eqminmin}, \mathord{\minmin}|_{\clAllleq{k}})$ of $\eqminmin$\hyp{}equivalence classes of $\clAllleq{k}$, ordered by $\minmin$.
(Note that we do indeed mean the restriction of the relation $\minmin$ to $\clAllleq{k}$: $\mathord{\minmin}|_{\clAllleq{k}}$.
We do \textbf{not} mean taking the transitive closure of the restrictions of $\minorant$ and $\minor$ to $\clAllleq{k}$: $(\mathord{\minorant}|_{\clAllleq{k}} \cup \mathord{\minor}|_{\clAllleq{k}})^\mathrm{tr}$.
Therefore, even if $f \minmin g$ with $f, g \in \clAllleq{k}$, it may be the case that every function $h$ satisfying $f \minorant h \minor g$ is in $\clAll \setminus \clAllleq{k}$.
This is fine and consistent with our definition.)

Let $\Ideals(\posetAllleq{k})$ be the lattice of ideals of $\posetAllleq{k}$,
\[
\Ideals(\posetAllleq{k}) :=
\{ \, {\downarrow^{[\leq k]}} X \mid X \subseteq \clAllleq{k} / \mathord{\eqminmin} \, \},
\]
where, for the sake of clarity, ${\downarrow^{[\leq k]}} X$ denotes the downset of $X$ in $\posetAllleq{k}$.
We identify each element ${\downarrow^{[\leq k]}} X$ of $\Ideals(\posetAllleq{k})$ with the union of $\eqminmin$\hyp{}equivalence classes in ${\downarrow^{[\leq k]}} X$, i.e., the set $\bigcup {\downarrow^{[\leq k]}} X$, which is easily seen to be equal to
$({\downarrow} X ) \cap \clAllleq{k}$, and in this last expression ${\downarrow} X$ is taken in $(\clAll, \mathord{\minmin})$.
\end{definition}

As observed above, we may choose functions of arity at most $2^k$ (at most $2^k - 1$ whenever $k \geq 2$) as representatives of the $\eqminmin$\hyp{}equivalence classes.
Consequently, for every $k \in \IN$, the poset $\posetAllleq{k}$ is finite, and so is its ideal lattice $\Ideals(\posetAllleq{k})$.

\begin{lemma}
\label{lem:kllk-inclusion-2}
For all $X, Y \subseteq \clAllleq{k}$, $\clKlik{k}{X} \subseteq \clKlik{k}{Y}$ if and only if ${\downarrow^{[\leq k]}} X \subseteq {\downarrow^{[\leq k]}} Y$.
\end{lemma}

\begin{proof}
Assume first that ${\downarrow^{[\leq k]}} X \subseteq {\downarrow^{[\leq k]}} Y$.
Then ${\downarrow} X \subseteq {\downarrow} Y$.
(For, suppose, to the contrary, that ${\downarrow} X \nsubseteq {\downarrow} Y$.
Then there is a $z \in {\downarrow} X \setminus {\downarrow} Y$, and hence there is an $x \in X$ such that $z \minmin x$.
Since $x \in X \subseteq {\downarrow^{[\leq k]}} X \subseteq {\downarrow^{[\leq k]}} Y$. there is a $y \in Y$ such that $x \minmin y$.
Thus, $z \minmin x \minmin y$, so $z \in {\downarrow} Y$, a contradiction.)
Consequently, $\clKlik{k}{X} \subseteq \clKlik{k}{Y}$ by Lemma~\ref{lem:kllk-inclusion}.

Assume now that ${\downarrow^{[\leq k]}} X \nsubseteq {\downarrow^{[\leq k]}} Y$.
Then there is a $\varphi \in {\downarrow^{[\leq k]}} X$ such that $\varphi \notin {\downarrow^{[\leq k]}} Y$.
Hence, there is an $x \in X$ such that $\varphi \minmin x$, and we must clearly have $x \notin {\downarrow^{[\leq k]}} Y$.
Therefore, we may assume, without loss of generality, that $\varphi \in X$.
We have $\varphi \in X \subseteq \clKlik{k}{X}$.
Suppose, to the contrary, that $\varphi \in \clKlik{k}{Y}$.
Since $\card{\varphi^{-1}(1)} \leq k$, there exists a $\psi \in Y$ and $\tau \colon \nset{h} \to \nset{n}$ such that $\psi(\vect{a} \tau) = 1$ for all $\vect{a} \in \varphi^{-1}(1)$.
This implies that $\varphi \minmin \psi$ and hence $\varphi \in {\downarrow^{[\leq k]}} \{ \psi \} \subseteq {\downarrow^{[\leq k]}} Y$, a contradiction.
\end{proof}

We have thus obtained the following result.

\begin{theorem}
\label{thm:Uktheta-lattice}
The sets of the form $\clKlik{k}{\Theta}$ constitute a lattice that is isomorphic to $\Ideals(\posetAllleq{k})$.
\end{theorem}

Note that for all $\Theta, \Phi \subseteq \clAllleq{k}$, $\clKlik{k}{\Theta} \cap \clKlik{k}{\Phi} = \clKlik{k}{\Psi}$, where
$\Psi = {\downarrow^{[\leq k]}} \Theta \cap {\downarrow^{[\leq k]}} \Phi$.

\begin{example}
The Hasse diagrams of $\posetAllleq{k}$ for $k \leq 3$ and $\Ideals(\posetAllleq{k})$ for $k \leq 2$ are presented in
Figures \ref{fig:minmin1}, \ref{fig:minmin2}, and \ref{fig:MinMin3}.
We have determined with SageMath~\cite{SageMath} that the ideal lattice $\Ideals(\posetAllleq{3})$ has 2854 elements; it is too big to present explicitly here.
In order to simplify notation in the figures, we specify each ideal by listing its maximal elements without any punctuation.
\end{example}

\begingroup
\newcommand{\DownIdNeg}{$\mathord{\id} \mathord{\neg}$}
\newcommand{\DownId}{$\mathord{\id}$}
\newcommand{\DownNeg}{$\mathord{\neg}$}
\newcommand{\DownNimp}{$\mathord{\nrightarrow}$}
\newcommand{\DownCo}{$0$}
\newcommand{\DownEmpty}{$\emptyset$}
\begin{figure}
\begin{center}
\begin{tabular}{c@{\qquad\qquad}c}
\tikzset{every node/.style={circle,draw,fill=black,inner sep=2pt}, every line/.style={thick}, font=\small} 
\scalebox{0.85}{%
\begin{tikzpicture}[baseline=(c0)]
\node[label=below:$0$] (c0) at (0,0) {};
\node[label=left:$\nrightarrow$] (imp) at (0,1) {};
\node[label=above:$\id$] (id) at (-1,2) {};
\node[label=above:$\neg$] (neg) at (1,2) {};
\draw (c0) -- (imp);
\draw (imp) -- (id);
\draw (imp) -- (neg);
\end{tikzpicture}
}
&
\tikzset{every node/.style={circle,draw,fill=black,inner sep=2pt}, every line/.style={thick}, font=\small} 
\scalebox{0.85}{%
\begin{tikzpicture}[baseline=(Empty)]
\node[label=above:\DownIdNeg] (IdNeg) at (0,0) {};
\node[label=left:\DownId] (Id) at ($(IdNeg)+(-1,-1)$) {};
\node[label=right:\DownNeg] (Neg) at ($(IdNeg)+(1,-1)$) {};
\draw (IdNeg) -- (Id);   \draw (IdNeg) -- (Neg);
\node[label=left:\DownNimp] (Nimp) at ($(IdNeg)+(0,-2)$) {};
\draw (Id) -- (Nimp);   \draw (Neg) -- (Nimp);
\node[label=left:\DownCo] (Co) at ($(Nimp)+(0,-1)$) {};
\draw (Nimp) -- (Co);
\node[label=below:\DownEmpty] (Empty) at ($(Co)+(0,-1)$) {};
\draw (Co) -- (Empty);
\end{tikzpicture}
}
\\
\\
$\posetAllleq{1} = (\clAllleq{1}, \mathord{\minmin})$
&
$\Ideals(\posetAllleq{1})$
\end{tabular}
\end{center}
\caption{The minorant\hyp{}minor poset of Boolean functions with at most one true point and its ideal lattice.}
\label{fig:minmin1}
\end{figure}
\endgroup

\begingroup
\newcommand{\DownCi}{$1$}
\newcommand{\DownIdPlusNeg}{$\mathord{\id} \mathord{+} \mathord{\neg}$}
\newcommand{\DownIdPlus}{$\mathord{\id} \mathord{+}$}
\newcommand{\DownIdNeg}{$\mathord{\id} \mathord{\neg}$}
\newcommand{\DownPlusNeg}{$\mathord{+} \mathord{\neg}$}
\newcommand{\DownIdLco}{$\mathord{\id} \lambda_{30}$}
\newcommand{\DownPlus}{$\mathord{+}$}
\newcommand{\DownNegLci}{$\mathord{\neg} \lambda_{31}$}
\newcommand{\DownId}{$\mathord{\id}$}
\newcommand{\DownLcoLci}{$\lambda_{30} \lambda_{31}$}
\newcommand{\DownNeg}{$\mathord{\neg}$}
\newcommand{\DownLco}{$\lambda_{30}$}
\newcommand{\DownLci}{$\lambda_{31}$}
\newcommand{\DownNimp}{$\mathord{\nrightarrow}$}
\newcommand{\DownCo}{$0$}
\newcommand{\DownEmpty}{$\emptyset$}
\begin{figure}
\begin{center}
\begin{tabular}{c@{\qquad\qquad}c}
\tikzset{every node/.style={circle,draw,fill=black,inner sep=2pt}, every line/.style={thick}, font=\small} 
\scalebox{0.85}{%
\begin{tikzpicture}[baseline=(c0)]
\node[label=below:$0$] (c0) at (0,0) {};
\node[label=left:$\nrightarrow$] (Nimp) at (0,1) {};
\node[label=left:$\lambda_{31}$] (L31) at (-1,2) {};
\node[label=right:$\lambda_{30}$] (L30) at (1,2) {};
\node[label=left:$\id$] (id) at (-2,3) {};
\node[label=left:$+$] (plus) at (0,3) {};
\node[label=right:$\neg$] (neg) at (2,3) {};
\node[label=above:$1$] (c1) at (0,5) {};
\draw (c0) -- (Nimp);
\draw (Nimp) -- (L31);
\draw (Nimp) -- (L30);
\draw (L31) -- (id);
\draw (L31) -- (plus);
\draw (L30) -- (plus);
\draw (L30) -- (neg);
\draw (id) -- (c1);
\draw (plus) -- (c1);
\draw (neg) -- (c1);
\end{tikzpicture}}
&
\tikzset{every node/.style={circle,draw,fill=black,inner sep=2pt}, every line/.style={thick}, font=\small} 
\scalebox{0.85}{%
\begin{tikzpicture}[baseline=(Empty)]
\node[label=above:\DownCi] (C1) at (0,0) {};
\node[label=left:\DownIdPlusNeg] (IdPlusNeg) at ($(C1)+(0,-1)$) {};
\draw (C1) -- (IdPlusNeg);
\node[label=left:\DownIdPlus] (IdPlus) at ($(IdPlusNeg)+(-1,-1)$) {};
\node[label=left:\DownIdNeg] (IdNeg) at ($(IdPlusNeg)+(0,-1)$) {};
\node[label=right:\DownPlusNeg] (PlusNeg) at ($(IdPlusNeg)+(1,-1)$) {};
\draw (IdPlusNeg) -- (IdPlus);   \draw (IdPlusNeg) -- (IdNeg);   \draw (IdPlusNeg) -- (PlusNeg);
\node[label=left:\DownIdLco] (IdL30) at ($(IdPlus)+(0,-1)$) {};
\node[label=left:\DownPlus] (Plus) at ($(IdNeg)+(0,-1)$) {};
\node[label=right:\DownNegLci] (NegL31) at ($(PlusNeg)+(0,-1)$) {};
\draw (IdPlus) -- (IdL30);   \draw (IdPlus) -- (Plus);
\draw (IdNeg) -- (IdL30);   \draw (IdNeg) -- (NegL31);
\draw (PlusNeg) -- (Plus);   \draw (PlusNeg) -- (NegL31);
\node[label=left:\DownId] (Id) at ($(IdL30)+(-1,-1)$) {};
\node[label=left:\DownLcoLci] (L30L31) at ($(Plus)+(0,-1)$) {};
\node[label=right:\DownNeg] (Neg) at ($(NegL31)+(1,-1)$) {};
\draw (IdL30) -- (Id);   \draw (IdL30) -- (L30L31);   \draw (Plus) -- (L30L31);   \draw (NegL31) -- (L30L31);   \draw (NegL31) -- (Neg);
\node[label=left:\DownLci] (L31) at ($(Id)+(1,-1)$) {};
\node[label=right:\DownLco] (L30) at ($(Neg)+(-1,-1)$) {};
\draw (Id) -- (L31);   \draw (L30L31) -- (L31);  \draw (L30L31) -- (L30);   \draw (Neg) -- (L30);
\node[label=left:\DownNimp] (Nimp) at ($(L30)!0.5!(L31)+(0,-1)$) {};
\draw (L30) -- (Nimp);   \draw (L31) -- (Nimp);
\node[label=left:\DownCo] (C0) at ($(Nimp)+(0,-1)$) {};
\draw (Nimp) -- (C0);
\node[label=below:\DownEmpty] (Empty) at ($(C0)+(0,-1)$) {};
\draw (C0) -- (Empty);
\end{tikzpicture}
}
\\
\\
$\posetAllleq{2} = (\clAllleq{2}, \mathord{\minmin})$
&
$\Ideals(\posetAllleq{2})$
\end{tabular}
\end{center}
\caption{The minorant\hyp{}minor poset of Boolean functions with at most two true points and its ideal lattice.}
\label{fig:minmin2}
\end{figure}
\endgroup

\begin{figure}
\begin{center}
\tikzset{every node/.style={circle,draw,fill=black,inner sep=2pt}, every line/.style={thick}, font=\small} 
\scalebox{0.85}{%
\begin{tikzpicture}
\node[label=below:$0$] (c0) at (0,3) {};
\node[label=left:$\nrightarrow$] (imp) at (0,4) {};
\node[label=right:$\lambda_2^0$] (L20) at (0.75,4.75) {};
\node[label=left:$\lambda_2^1$] (L21) at (-0.75,4.75) {};
\node[label=left:$\lambda_2$] (L2) at (0,5.5) {};
\node[label=right:$\lambda_{10}^0$] (L100) at (1.5,5.5) {};
\node[label=left:$\lambda_{11}^1$] (L111) at (-1.5,5.5) {};
\node[label=right:$\delta_0^{0{+}}$] (D00p) at (3,7) {};
\node[label=left:$\delta_1^{1{+}}$] (D11p) at (-3,7) {};
\node[label=right:$\lambda_{30}$] (L30) at (4.5,8.5) {};
\node[label=left:$\lambda_{31}$] (L31) at (-4.5,8.5) {};
\node[label=right:$\delta_0^0$] (D00) at (3,8.5) {};
\node[label=left:$\delta_1^1$] (D11) at (-3,8.5) {};
\node[label=right:$A_0^{+}$] (A0p) at (4.5,10) {};
\node[label=left:$A_1^{+}$] (A1p) at (-4.5,10) {};
\node[label=right:$A_0$] (A0) at (4.5,11.5) {};
\node[label=left:$A_1$] (A1) at (-4.5,11.5) {};
\node[label=right:$\delta_1^0$] (D10) at (6,10) {};
\node[label=left:$\delta_0^1$] (D01) at (-6,10) {};
\node[label=right:$\lambda_{10}$] (L10) at (-3,10) {};
\node[label=left:$\lambda_{11}$] (L11) at (3,10) {};
\node[label=right:$\neg$] (neg) at (4.5,13) {};
\node[label=left:$\id$] (id) at (-4.5,13) {};
\node[label=below:$\Gamma_{01}$] (G01) at (0,12.75) {};
\node[label=left:$\delta_0^{+}$] (D0p) at (-3,13.25) {};
\node[label=right:$\delta_1^{+}$] (D1p) at (3,13.25) {};
\node[label=above:$+$] (plus) at (0,15.75) {};
\node[label=left:$\delta_0$] (D0) at (-4.5,15.75) {};
\node[label=right:$\delta_1$] (D1) at (4.5,15.75) {};
\node[label=right:$\Gamma_0$] (G0) at (3,15.75) {};
\node[label=left:$\Gamma_1$] (G1) at (-3,15.75) {};
\node[label=left:$\vee$] (or) at (-3,17) {};
\node[label=right:$\uparrow$] (nand) at (3,17) {};
\node[label=above:$1$] (c1) at (0,18.5) {};
\draw (c0) -- (imp);
\draw (imp) -- (L20);   \draw (imp) -- (L21);
\draw (L20) -- (L2);    \draw (L21) -- (L2);
\draw (L20) -- (L100);  \draw (L21) -- (L111);
\draw (L2) -- (L10);    \draw (L2) -- (L11);
\draw (L100) -- (L10);  \draw (L111) -- (L11);
\draw (L100) -- (D00p); \draw (L111) -- (D11p);
\draw (D00p) -- (L30);  \draw (D11p) -- (L31);
\draw (L30) -- (L11);   \draw (L31) -- (L10);
\draw (G1) -- (D0p);    \draw (G0) -- (D1p);
\draw (D00p) -- (D0p);  \draw (D11p) -- (D1p);
\draw (L10) -- (D0p);   \draw (L11) -- (D1p);
\draw (D00p) -- (D00);  \draw (D11p) -- (D11);
\draw (L10) -- (G01);   \draw (L11) -- (G01);
\draw (L30) -- (D10);   \draw (L31) -- (D01);
\draw (L30) -- (A0p);   \draw (L31) -- (A1p);
\draw (D00) -- (D0);    \draw (D11) -- (D1);
\draw (D01) -- (D0);    \draw (D10) -- (D1);
\draw (D0p) -- (D0);    \draw (D1p) -- (D1);
\draw (A0p) -- (A0);    \draw (A1p) -- (A1);
\draw (D00) -- (A0);    \draw (D11) -- (A1);
\draw (A0) -- (neg);    \draw (A1) -- (id);
\draw (D10) -- (neg);   \draw (D01) -- (id);
\draw (D0p) -- (plus);  \draw (D1p) -- (plus);
\draw (G01) -- (plus);
\draw (neg) -- (G0);    \draw (id) -- (G1);
\draw (G01) -- (G0);    \draw (G01) -- (G1);
\draw (A1p) -- (G0);    \draw (A0p) -- (G1);
\draw (A0) -- (or);     \draw (A1) -- (nand);
\draw (D0) -- (or);     \draw (D1) -- (nand);
\draw (plus) -- (or);   \draw (plus) -- (nand);
\draw (G1) -- (or);     \draw (G0) -- (nand);
\draw (or) -- (c1);     \draw (nand) -- (c1);
\end{tikzpicture}
}

\medskip
$\posetAllleq{3} = (\clAllleq{3}, \mathord{\minmin})$
\end{center}

\caption{The minorant\hyp{}minor poset of Boolean functions with at most three true points.}
\label{fig:MinMin3}
\end{figure}


\subsection{Closures of functions and extensions of the minorant\hyp{}minor order}

The goal of this subsection is to describe the intersections of classes of the form $\clKlik{k}{\Theta}$ with other types of $(\clIc,\clMcUk{k})$\hyp{}clonoids, such as $\clXI$, $\clIX$, $\clM$, $\clMneg$, and $\clRefl$ (for the notation, see Definitions~\ref{def:all}, \ref{def:neg}, \ref{def:monotone}).
To this end, we are going to define a few extensions of the minorant\hyp{}minor order.

\begin{definition}
\label{def:C-closure}
Let $f \in \clAll^{(n)}$.
For $C \in \{\clXI, \clIX, \clM, \clMneg, \clRefl\}$,
we define the \emph{$C$\hyp{}closure} of $f$, denoted by $f^C$, as the least majorant $g$ of $f$ such that $g \in C$.
(Such a least majorant exists; it is the conjunction of all majorants of $f$ that are in $C$, i.e., $f^C := \bigwedge \{ \, g \in C \mid f \minorant g \, \}$.
In order to see this, note first that there exist majorants of $f$ in $C$, namely $f \minorant 1 \in C$ for each class $C$ considered here.
Secondly, the conjunction of majorants of $f$ is again a majorant of $f$.
Thirdly, for each class $C$ considered here, it holds that $\{ \mathord{\wedge} \} C \subseteq C$.)
It is easy to see that
\begin{itemize}
\item $f^\clXI(\vect{a}) = f(\vect{a})$ if $\vect{a} \neq \vect{1}$ and $f^\clXI(\vect{1}) = 1$,
\item $f^\clIX(\vect{a}) = f(\vect{a})$ if $\vect{a} \neq \vect{0}$ and $f^\clIX(\vect{0}) = 1$,
\item $f^\clM(\vect{a}) = 1$ if and only if there is a $\vect{b} \in \{0,1\}^n$ such that $\vect{b} \leq \vect{a}$ and $f(\vect{b}) = 1$,
\item $f^\clMneg(\vect{a}) = 1$ if and only if there is a $\vect{b} \in \{0,1\}^n$ such that $\vect{a} \leq \vect{b}$ and $f(\vect{b}) = 1$,
\item $f^\clRefl(\vect{a}) = f(\vect{a}) \vee f(\overline{\vect{a}})$.
\end{itemize}
For any set $K \subseteq \clAll$, we let $K^C := \{ \, f^C \mid f \in K \, \}$.
Furthermore, we define the \emph{$C$\hyp{}closure relation} $\RelCl{C}$ on $\clAll$ by the rule $f \RelCl{C} g$ if and only if $f = g^C$.
We extend the minorant\hyp{}minor relation $\minmin$ into the \emph{$C$\hyp{}closed minorant\hyp{}minor relation} $\minmin_C$, defined as the transitive closure of the union of the minorant, minor, and the $C$\hyp{}closure relations, i.e., $\minmin_C := (\mathord{\minorant} \cup \mathord{\minor} \cup \mathord{\RelCl{C}})^\mathrm{tr}$.
Thus, $\minmin_C$ is a quasiorder, and it induces an equivalence relation $\eqminmin_C$ on $\clAll$ and a partial order, also denoted by $\minmin_C$, on $\clAll / \mathord{\eqminmin_C}$.

Similarly to Definition~\ref{def:Id-minmin}, we are interested in the restriction of the $C$\hyp{}closed minorant\hyp{}minor relation $\minmin_C$ to Boolean functions with at most $k$ true points, and we consider the poset
$\posetAllleq{k}_C = (\clAllleq{k} / \mathord{\eqminmin_C}, \mathord{\minmin_C}|_{\clAllleq{k}_C})$
of $\eqminmin_C$\hyp{}equivalence classes of $\clAllleq{k}$ ordered by $\minmin_C$.
Let $\Ideals(\posetAllleq{k}_C)$ be the lattice of ideals of $\posetAllleq{k}_C$,
\[
\Ideals(\posetAllleq{k}_C) :=
\{ \, {\downarrow^{[\leq k]}_C} X \mid X \subseteq \clAllleq{k} / \mathord{\eqminmin_C} \, \},
\]
where, for the sake of clarity, ${\downarrow^{[\leq k]}_C} X$ denotes the downset of $X$ in $\posetAllleq{k}_C$.
We identify each element ${\downarrow^{[\leq k]}_C} X$ of $\Ideals(\posetAllleq{k}_C)$ with the union of $\eqminmin_C$\hyp{}equivalence classes in ${\downarrow^{[\leq k]}_C} X$, i.e., the set $\bigcup {\downarrow^{[\leq k]}_C} X$, which is easily seen to be equal to
$({\downarrow_C} X ) \cap \clAllleq{k}$, and in this last expression, the downset ${\downarrow_C} X$ is taken in $(\clAll, \mathord{\minmin_C})$.
\end{definition}

\begin{remark}
\label{rem:1MR}
By the definition of $C$\hyp{}closure, we have $f \minorant f^C \RelCl{C} f$ for any function $f \in \clAll$ and $C \in \{\clXI, \clIX, \clM, \clMneg, \clRefl\}$.
Therefore $\mathord{\minorant} \circ \mathord{\RelCl{C}}$ is a reflexive relation.
Furthermore, $\mathord{\RelCl{C}}$ is a transitive relation, which follows immediately from the fact that $(f^C)^C = f^C$.
\end{remark}

We now find analogues of Lemma~\ref{lem:minmin} for $\mathord{\minmin_C}$.

\begin{lemma}
\label{lem:1MR-apu}
\leavevmode
\begin{enumerate}[label={\upshape(\roman*)}]
\item\label{lem:1MR-apu:1}
$\mathord{\RelCl{\clXI}} \circ \mathord{\minorant} \subseteq \mathord{\minorant} \circ \mathord{\RelCl{\clXI}}$
and
$\mathord{\RelCl{\clXI}} \circ \mathord{\minor} \subseteq \mathord{\minorant} \circ \mathord{\minor} \circ \mathord{\RelCl{\clXI}}$.

\item\label{lem:1MR-apu:0}
$\mathord{\RelCl{\clIX}} \circ \mathord{\minorant} \subseteq \mathord{\minorant} \circ \mathord{\RelCl{\clIX}}$
and
$\mathord{\RelCl{\clIX}} \circ \mathord{\minor} \subseteq \mathord{\minorant} \circ \mathord{\minor} \circ \mathord{\RelCl{\clIX}}$.

\item\label{lem:1MR-apu:M}
$\mathord{\RelCl{\clM}} \circ \mathord{\minorant} \subseteq \mathord{\minorant} \circ \mathord{\RelCl{\clM}}$
and
$\mathord{\RelCl{\clM}} \circ \mathord{\minor} \subseteq \mathord{\minorant} \circ \mathord{\minor} \circ \mathord{\RelCl{\clM}}$.

\item\label{lem:1MR-apu:Mneg}
$\mathord{\RelCl{\clMneg}} \circ \mathord{\minorant} \subseteq \mathord{\minorant} \circ \mathord{\RelCl{\clMneg}}$
and
$\mathord{\RelCl{\clMneg}} \circ \mathord{\minor} \subseteq \mathord{\minorant} \circ \mathord{\minor} \circ \mathord{\RelCl{\clMneg}}$.

\item\label{lem:1MR-apu:R}
$\mathord{\RelCl{\clRefl}} \circ \mathord{\minorant} \subseteq \mathord{\minorant} \circ \mathord{\RelCl{\clRefl}}$
and
$\mathord{\RelCl{\clRefl}} \circ \mathord{\minor} \subseteq \mathord{\minor} \circ \mathord{\RelCl{\clRefl}}$.
\end{enumerate}
\end{lemma}

\begin{proof}
\ref{lem:1MR-apu:1}
Assume that $f \mathrel{(\mathord{\RelCl{\clXI}} \circ \mathord{\minorant})} g$.
Then there exists a function $h$ such that $f \RelCl{\clXI} h \minorant g$, that is, $f = h^\clXI$ and $h$ is a minorant of $g$.
We want to show that $f \mathrel{(\mathord{\minorant} \circ \mathord{\RelCl{\clXI}})} g$, i.e., $f \minorant g^\clXI$.
For any $\vect{a} \neq \vect{1}$, we have $f(\vect{a}) = h^\clXI(\vect{a}) \leq g(\vect{a}) = g^\clXI(\vect{a})$ and $f(\vect{1}) = h^\clXI(\vect{1}) = 1 = g^\clXI(\vect{1})$,
which shows that $f \minorant g^\clXI$.

Assume now that $f \mathrel{(\mathord{\RelCl{\clXI}} \circ \mathord{\minor})} g$.
Then there exists a minor formation map $\sigma$ such that $f = (g_\sigma)^\clXI$.
In order to show that $f \mathrel{(\mathord{\minorant} \circ \mathord{\minor} \circ \mathord{\RelCl{\clXI}})} g$, it suffices to show that $(g_\sigma)^\clXI \minorant (g^\clXI)_\sigma$.
If $\vect{a} \neq \vect{1}$, then $(g_\sigma)^\clXI(\vect{a}) = g_\sigma(\vect{a}) = g(\vect{a} \sigma) \leq g^\clXI(\vect{a} \sigma) = (g^\clXI)_\sigma(\vect{a})$.
Furthermore, $(g_\sigma)^\clXI(\vect{1}) = 1 = (g^\clXI)_\sigma(\vect{1})$.
This shows that $f \mathrel{(\mathord{\minorant} \circ \mathord{\minor} \circ \mathord{\RelCl{\clXI}})} g$.

\ref{lem:1MR-apu:0}
The proof is similar to that of statement \ref{lem:1MR-apu:1}.

\ref{lem:1MR-apu:M}
Assume that $f \mathrel{(\mathord{\RelCl{\clM}} \circ \mathord{\minorant})} g$.
Then there exists a function $h$ such that $f \RelCl{\clM} h \minorant g$, that is, $f = h^\clM$ and $h$ is a minorant of $g$.
We want to show that $f \mathrel{(\mathord{\minorant} \circ \mathord{\RelCl{\clM}})} g$, i.e., $f \minorant g^\clM$.
For this, it is sufficient to show that for all $\vect{a} \in \{0,1\}^n$, $f(\vect{a}) = 1$ implies $g^\clM(\vect{a}) = 1$.
Assume $f(\vect{a}) = 1$.
Since $f = h^\clM$, there exists a $\vect{b}$ with $\vect{b} \leq \vect{a}$ such that $h(\vect{b}) = 1$.
Since $h \minorant g$, we have $g(\vect{b}) = 1$.
Since $\vect{b} \leq \vect{a}$, this implies that $g^\clM(\vect{a}) = 1$.

Assume that $f \mathrel{(\mathord{\RelCl{\clM}} \circ \mathord{\minor})} g$.
Then there exists a minor formation map $\sigma$ such that $f = (g_\sigma)^\clM$.
In order to show that $f \mathrel{(\mathord{\minorant} \circ \mathord{\minor} \circ \mathord{\RelCl{\clM}})} g$, it suffices to show that $f \minorant (g^\clM)_\sigma$.
For this, it suffices to show that for all $\vect{a}$, $(g^\clM)_\sigma(\vect{a}) = 0$ implies $f(\vect{a}) = 0$.
Assume $(g^\clM)_\sigma(\vect{a}) = 0$.
Then $g^\clM(\vect{a} \sigma) = 0$, and hence, by the definition of the monotone closure, $g(\vect{b}) = 0$ whenever $\vect{b} \leq \vect{a} \sigma$. Consequently, $g_\sigma(\vect{c}) = 0$ whenever $\vect{c} \leq \vect{a}$, which in turn implies $f(\vect{a}) = (g_\sigma)^\clM(\vect{a}) = 0$.

\ref{lem:1MR-apu:Mneg}
The proof is similar to that of statement \ref{lem:1MR-apu:M}.

\ref{lem:1MR-apu:R}
Assume that $f \mathrel{(\mathord{\RelCl{\clRefl}} \circ \mathord{\minorant})} g$.
Then there exists a function $h$ such that $f \RelCl{\clRefl} h \minorant g$, that is, $f = h^\clRefl$ and $h$ is a minorant of $g$.
We want to show that $f \mathrel{(\mathord{\minorant} \circ \mathord{\RelCl{\clRefl}})} g$, i.e., $f \minorant g^\clRefl$.
For this, it is sufficient to show that for all $\vect{a} \in \{0,1\}^n$, $f(\vect{a}) = 1$ implies $g^\clRefl(\vect{a}) = 1$.
Assume $f(\vect{a}) = 1$.
Since $f = h^\clRefl$, we have $h(\vect{a}) \vee h(\overline{\vect{a}}) = 1$; hence $h(\vect{a}) = 1$ or $h(\overline{\vect{a}}) = 1$.
Since $h \minorant g$, it follows that $g(\vect{a}) = 1$ or $g(\overline{\vect{a}}) = 1$; hence $g^\clRefl(\vect{a}) = 1$.
This shows that $f \mathrel{(\mathord{\minorant} \circ \mathord{\RelCl{\clRefl}})} g$.

Assume that $f \mathrel{(\mathord{\RelCl{\clRefl}} \circ \mathord{\minor})} g$.
Then there exists a minor formation map $\sigma$ such that $f = (g_\sigma)^\clRefl$.
We have
\begin{align*}
f(\vect{a}) = 1
& \iff (g_\sigma)^\clRefl(\vect{a}) = 1 \\
& \iff g_\sigma(\vect{a}) \vee g_\sigma(\overline{\vect{a}}) = 1 \\
& \iff g_\sigma(\vect{a}) = 1 \text{ or } g_\sigma(\overline{\vect{a}}) = 1 \\
& \iff g(\vect{a} \sigma) = 1 \text{ or } g(\overline{\vect{a}} \sigma) = 1 \\
& \iff g^\clRefl(\vect{a} \sigma) = 1 \\
& \iff (g^\clRefl)_\sigma(\vect{a}) = 1;
\end{align*}
hence $f = (g^\clRefl)_\sigma$, which shows that $f \mathrel{(\mathord{\minor} \circ \mathord{\RelCl{\clRefl}})} g$.
\end{proof}

\begin{lemma}
\label{lem:1MR-prod}
For $C \in \{\clXI, \clIX, \clM, \clMneg, \clRefl\}$,
$\mathord{\minmin_C} = \mathord{\minorant} \circ \mathord{\minor} \circ \mathord{\RelCl{C}}$.
\end{lemma}

\begin{proof}
We need to show that $(\mathord{\minorant} \cup \mathord{\minor} \cup \mathord{\RelCl{C}})^\mathrm{tr} = \mathord{\minorant} \circ \mathord{\minor} \circ \mathord{\RelCl{C}}$.
The inclusion $\mathord{\minorant} \circ \mathord{\minor} \circ \mathord{\RelCl{C}} \subseteq (\mathord{\minorant} \cup \mathord{\minor} \cup \mathord{\RelCl{C}})^\mathrm{tr}$ holds by the definition of transitive closure.

In order to prove the converse inclusion,
assume $f \mathrel{(\mathord{\minorant} \cup \mathord{\minor} \cup \mathord{\RelCl{C}})^\mathrm{tr}} g$.
Then there exists a $p \in \IN$ and a sequence of relations $\alpha_1, \dots, \alpha_p \in \{\mathord{\minorant}, \mathord{\minor}, \mathord{\RelCl{C}}\}$ such that $f \mathrel{(\alpha_1 \circ \alpha_2 \circ \dots \circ \alpha_p)} g$.
Since $\minorant$, $\minor$, and $\mathord{\minorant} \circ \mathord{\RelCl{C}}$ are reflexive relations (see Remark~\ref{rem:1MR}), we may assume that each of the relations $\minorant$, $\minor$, and $\mathord{\RelCl{C}}$ appears in the sequence at least once.
By using the inclusions between the relational products
established in Lemmata~\ref{lem:minmin}\ref{lem:minmin:move} and \ref{lem:1MR-apu},
we can bring all occurrences of $\minorant$ to the left, keep all occurrences of $\minor$ in the middle, and bring all occurrences of $\mathord{\RelCl{C}}$ to the right.
Since $\minorant$, $\minor$, and $\mathord{\RelCl{C}}$ are transitive (see Remark~\ref{rem:1MR}), we can then omit all repetitions, and we are left with $f \mathrel{(\mathord{\minorant} \circ \mathord{\minor} \circ \mathord{\RelCl{C}})} g$.
This shows that $(\mathord{\minorant} \cup \mathord{\minor} \cup \mathord{\RelCl{C}})^\mathrm{tr} \subseteq \mathord{\minorant} \circ \mathord{\minor} \circ \mathord{\RelCl{C}}$.
\end{proof}

\begin{lemma}
For $C \in \{\clXI, \clIX, \clM, \clMneg, \clRefl\}$,
$\mathord{\minmin} \subseteq \mathord{\minmin_C}$.
Consequently,
$\eqminmin$ is a refinement of $\eqminmin_C$.
\end{lemma}

\begin{proof}
Clear from the definition of the relations $\minmin$ and $\minmin_C$.
\end{proof}

\begin{definition}
\label{def:kC-closed}
Let $k \in \IN \cup \{\infty\}$, $\Phi \subseteq \clAll$, and $C \in \{\clXI, \clIX, \clM, \clMneg, \clRefl\}$.
The set $\Phi$ is \emph{$(k,C)$\hyp{}closed} if for every $\varphi \in \Phi$ and for every $T \subseteq (\varphi^C)^{-1}(1)$ with $\card{T} \leq k$, we have $\chi_T \in \Phi$.
We also say that a $(\clIc,\clMcUk{k})$\hyp{}clonoid $K$ is \emph{$(k,C)$\hyp{}closed} if $K = \clKlik{k}{\Phi}$ for some $(k,C)$\hyp{}closed $\Phi$.
\end{definition}

\begin{lemma}
\label{lem:k-1MR-char}
Let $k \in \IN \cup \{\infty\}$, $\Phi \in \Ideals(\posetAllleq{k})$, and $C \in \{\clXI, \clIX, \clM, \clMneg, \clRefl\}$.
Then $\Phi$ is $(k,C)$\hyp{}closed if and only if $\Phi \in \Ideals(\posetAllleq{k}_C)$.
\end{lemma}

\begin{proof}
Assume first that $\Phi$ is $(k,C)$\hyp{}closed.
We need to show that for any $\varphi \in \Phi$ and every $\gamma \in \clAllleq{k}$ with $\gamma \minmin_C \varphi$, we have $\gamma \in \Phi$.

Thus, let $\varphi \in \Phi$ and $\gamma \in \clAllleq{k}$ with $\gamma \minmin_C \varphi$.
Then $\gamma \mathrel{(\mathord{\minorant} \circ \mathord{\minor} \circ \RelCl{C})} \varphi$ by Lemma~\ref{lem:1MR-prod}, so $\gamma \mathrel{(\mathord{\minorant} \circ \mathord{\minor})} \varphi^C$.
Therefore, there exists a minor formation map $\sigma$ such that $\gamma \minorant (\varphi^C)_\sigma$.
Let $S := \gamma^{-1}(1)$, and let $S \sigma := \{ \, \vect{a} \sigma \mid \vect{a} \in S \, \}$.
For every $\vect{b} \in S \sigma$, we have $\vect{b} = \vect{a} \sigma$ for some $\vect{a} \in S$, and hence $1 = \chi_{S \sigma}(\vect{b}) = \chi_{S \sigma}(\vect{a} \sigma) = \chi_S(\vect{a}) =  \gamma(\vect{a}) = (\varphi^C)_\sigma(\vect{a}) = \varphi^C(\vect{a} \sigma) = \varphi^C(\vect{b})$; therefore $\chi_{S \sigma} \minorant \varphi^C$.
Since $\gamma \in \clAllleq{k}$, we have $\card{S} \leq k$; hence $\card{S \sigma} \leq k$ as well.
Therefore $\chi_{S \sigma} \in \Phi$ by Definition~\ref{def:kC-closed}.
Since $\gamma = \chi_S \minmin (\chi_{S \sigma})_\sigma$ by Lemma~\ref{lem:chiT-chiTsigma}, we have $\gamma \in \Phi$, because $\Phi$ is closed under minorants and minors.

Assume now that $\Phi \in \Ideals(\posetAllleq{k}_C)$.
Let $\varphi \in \Phi$, $S \subseteq (\varphi^C)^{-1}(1)$ with $\card{S} \leq k$.
We need to show that $\chi_S \in \Phi$.
We have $\chi_S \in \clAllleq{k}$ by definition.
Moreover, $\chi_S \minorant \varphi^C$, so $\chi_S \minorant \varphi^C \minor \varphi^C \mathrel{\RelCl{C}} \varphi$, so $\chi_S \minmin_C \varphi$.
It now follows from our assumption that $\chi_S \in \Phi$.
\end{proof}

\begin{lemma}
\label{lem:k-1MR-closed}
Let $k \in \IN \cup \{\infty\}$, $\Phi \subseteq \clAll$ and $C \in \{ \clXI, \clIX, \clM, \clMneg, \clRefl \}$.
\begin{enumerate}[label=\textup{(\roman*)}]
\item\label{lem:k-1MR-closed:a}
If $\Phi$ is $(k,C)$\hyp{}closed, then $(\clKlik{k}{\Phi})^C = \clKlik{k}{\Phi} \cap C$.

\item\label{lem:k-1MR-closed:b}
If $\Phi$ is a minorant minion and $(\clKlik{k}{\Phi})^C = \clKlik{k}{\Phi} \cap C$, then $\Phi$ is $(k,C)$\hyp{}closed.
\end{enumerate}
\end{lemma}

\begin{proof}
\ref{lem:k-1MR-closed:a}
The inclusion $\clKlik{k}{\Phi} \cap C \subseteq (\clKlik{k}{\Phi})^C$ holds by definition.

In order to prove the converse inclusion $(\clKlik{k}{\Phi})^C \subseteq \clKlik{k}{\Phi} \cap C$, let $f \in (\clKlik{k}{\Phi})^C$.
We have $f \in C$ by definition, and it only remains to show that $f \in \clKlik{k}{\Phi}$.
There exists a $f' \in \clKlik{k}{\Phi}$ such that $f = (f')^C$.
Let $T \subseteq f^{-1}(1)$ with $\card{T} \leq k$; say $T = \{ \vect{a}_1, \vect{a}_2, \dots, \vect{a}_p \}$ for some $p \leq k$.
We need to show that $\chi_T \in {\downarrow} \Phi$.
We consider different cases according to the class $C$.

Case 1: $C = \clXI$.
If $\vect{1} \notin T$, then $T \subseteq (f')^{-1}(1)$, so $\chi_T \in {\downarrow} \Phi$ because $f' \in \clKlik{k}{\Phi}$.
If $\vect{1} \in T$, then $T \setminus \{\vect{1}\} \subseteq (f')^{-1}(1)$, so $\chi_{T \setminus \{\vect{1}\}} \in {\downarrow} \Phi$.
In other words, there exists a $\varphi \in \Phi$ such that $\chi_{T \setminus \{\vect{1}\}} \minmin \varphi$, i.e., there is a $\tau \colon \nset{\arity \varphi} \to \nset{n}$ such that
$\chi_{T \setminus \{\vect{1}\}} \minorant \varphi_\tau \minor \varphi$.
Let $T \tau := \{ \vect{a} \tau \mid \vect{a} \in T \}$; note that $\vect{1} \tau = \vect{1}$ and $\chi_{T \tau} \minorant \varphi^\clXI$.
Since $\Phi$ is $(k,\clXI)$\hyp{}closed and $\card{T \tau} \leq \card{T} \leq k$, we have that $\chi_{T \tau} \in \Phi$.
Consequently, $\chi_T \in {\downarrow} \Phi$ by Lemma~\ref{lem:chiT-chiTsigma}.

Case 2: $C = \clIX$.
The proof is similar to the previous case.

Case 3: $C = \clM$.
For each $\vect{a} \in T$, there is a $\vect{b}_\vect{a} \in (f')^{-1}(1)$ with $\vect{b}_\vect{a} \leq \vect{a}$; let $B := \{ \vect{b}_\vect{a} \mid \vect{a} \in T \}$.
Since $f' \in \clKlik{k}{\Phi}$ and $\card{B} \leq \card{T} \leq k$, there exists a $\varphi \in \Phi$ and $\tau \colon \nset{\arity \varphi} \to \nset{n}$ such that $\varphi(\vect{b}_\vect{a} \tau) = 1$ for all $\vect{a} \in T$.
Note that $\vect{b}_\vect{a} \tau \leq \vect{a} \tau$; hence $\varphi^\clM(\vect{a} \tau) = 1$ for all $\vect{a} \in T$.
Let $T \tau := \{\vect{a} \tau \mid \vect{a} \in T \}$.
Since $\Phi$ is $(k,\clM)$\hyp{}closed and $\card{T \tau} \leq \card{T} \leq k$, we have that $\chi_{T \tau} \in \Phi$.
Consequently, $\chi_T \in {\downarrow} \Phi$ by Lemma~\ref{lem:chiT-chiTsigma}.

Case 4: $C = \clMneg$.
The proof is similar to the previous case.

Case 5: $C = \clRefl$.
For each $\vect{a} \in T$, there is a $\vect{b}_\vect{a} \in (f')^{-1}(1)$ such that $\vect{b}_\vect{a} \in \{\vect{a}, \overline{\vect{a}}\}$.
Since $f' \in \clKlik{k}{\Phi}$, there exist a $\varphi \in \Phi$ and $\tau \colon \nset{\arity \varphi} \to \nset{n}$ such that $\varphi(\vect{b}_\vect{a} \tau) = 1$ for all $\vect{a} \in T$.
Note that $\vect{a} \tau \in \{ \vect{b}_\vect{a} \tau, \overline{\vect{b}_\vect{a} \tau} \}$.
Let $T \tau := \{\vect{a} \tau \mid \vect{a} \in T\}$.
Since $\Phi$ is $(k,\clRefl)$\hyp{}closed, $\chi_{T \tau} \minorant \varphi^\clRefl$, and $\card{T \tau} \leq \card{T} \leq k$,
it follows that $\chi_{T \tau} \in \Phi$.
Consequently, $\chi_T \in {\downarrow} \Phi$ by Lemma~\ref{lem:chiT-chiTsigma}.

\ref{lem:k-1MR-closed:b}
Let $\varphi \in \Phi$ and $T \subseteq (\varphi^C)^{-1}(1)$ with $\card{T} \leq k$.
We need to show that $\chi_T \in \Phi$.
Since $\varphi \in \Phi \subseteq \clKlik{k}{\Phi}$, we have $\varphi^C \in (\clKlik{k}{\Phi})^C = \clKlik{k}{\Phi} \cap C$.
Because $\varphi^C \in \clKlik{k}{\Phi}$, $T \subseteq (\varphi^C)^{-1}(1)$, and $\card{T} \leq k$, we have $T \tulee \Phi$, so $\chi_T \in {\downarrow} \Phi = \Phi$, where the equality holds because $\Phi$ is a minorant minion.
\end{proof}

\begin{lemma}
\label{lem:unionkC}
For $k \in \IN \cup \{\infty\}$ and $C \in \{ \clXI, \clIX, \clM, \clMneg, \clRefl \}$,
the union of $(k,C)$\hyp{}closed subsets of $\clAllleq{k}$ is $(k,C)$\hyp{}closed.
\end{lemma}

\begin{proof}
Straightforward verification.
\end{proof}

By Lemma~\ref{lem:unionkC} it is meaningful to speak of the largest $(k,C)$\hyp{}closed subset of a set $\Theta$ of Boolean functions.

\begin{proposition}
\label{prop:1MR-part}
Let $k \in \IN \cup \{\infty\}$, let $C \in \{ \clXI, \clIX, \clM, \clMneg, \clRefl \}$, and let $\Theta \in \Ideals(\posetAllleq{k})$.
Then $\clKlik{k}{\Theta} \cap C = (\clKlik{k}{\Psi})^C = \clKlik{k}{\Psi} \cap C$,
where $\Psi$ is the largest $(k,C)$\hyp{}closed subset of $\Theta$.
\end{proposition}

\begin{proof}
The equality
$(\clKlik{k}{\Psi})^C = \clKlik{k}{\Psi} \cap C$
holds by Lemma~\ref{lem:k-1MR-closed}.
Since $\Psi \subseteq \Theta$,
we have $\clKlik{k}{\Psi} \subseteq \clKlik{k}{\Theta}$ by Lemma~\ref{lem:kllk-inclusion-2} and hence
$\clKlik{k}{\Psi} \cap C \subseteq \clKlik{k}{\Theta} \cap C$.

Suppose now, to the contrary, that
the inclusion $\clKlik{k}{\Psi} \cap C \subseteq \clKlik{k}{\Theta} \cap C$ is proper.
Then there exists an $f \in (\clKlik{k}{\Theta} \cap C) \setminus (\clKlik{k}{\Psi} \cap C) = (\clKlik{k}{\Theta} \cap C) \setminus \clKlik{k}{\Psi}$.
Since $f \notin \clKlik{k}{\Psi}$, there exists $T \subseteq f^{-1}(1)$ with $\card{T} \leq k$, say $T = \{\vect{u}_1, \dots, \vect{u}_p\} \in f^{-1}(1)$ with $p \leq k$, such that for all $\psi \in \Psi$ and for all $\pi \colon \nset{\arity \psi} \to \nset{n}$, there is an $i \in \nset{p}$ such that $\psi(\vect{u}_i \pi) = 0$.
Since $f \in \clKlik{k}{\Theta}$, there is a $\theta \in \Theta$ and $\tau \in \nset{\arity \theta} \to \nset{n}$ such that $\theta(\vect{u}_i \tau) = 1$ for all $i \in \nset{p}$.
We may assume that $\tau$ is surjective.
(If necessary, we just need to add fictitious arguments to $\theta$ and extend $\tau$ in such a way that every element of $\nset{n}$ gets a preimage.
Since $\Theta$ is minor\hyp{}closed, the new modified $\theta$ will also belong to $\Theta$.)
In fact, for $T \tau := \{ \, \vect{t} \tau \mid \vect{t} \in T \,\}$, we have $\chi_{T \tau}(\vect{t} \tau) = \theta(\vect{t} \tau) = 1$ for all $\vect{t} \in T$, and therefore we can take $\chi_{T \tau}$ in place of $\theta$; since $\chi_{T \tau} \minorant \theta$ and $\Theta$ is minorant\hyp{}closed, we have $\chi_{T \tau} \in \Theta$ as well.
Since $\tau$ is surjective, it has a right inverse, i.e., a map $\tau' \colon \nset{n} \to \nset{\arity \theta}$ such that $\tau \tau'$ is the identity map on $\nset{n}$, and, consequently, $\vect{t} = \vect{t} \tau \tau'$ for all $\vect{t} \in T$.

Note that $\chi_{T \tau}(\vect{t} \tau) = 1$ for all $\vect{t} \in T$; therefore $\chi_{T \tau} \notin \Psi$ because of the way we chose the set $T$.
By the definition of $\Psi$, there exists a set $S \subseteq (\chi_{T \tau}^C)^{-1}(1)$ with $\card{S} \leq k$ such that $\chi_S \notin \Theta$.
Our goal is to show that for $S \tau' := \{ \, \vect{a} \tau' \mid \vect{a} \in S \,\}$, we have $S \tau' \subseteq f^{-1}(1)$.
Since $f \in \clKlik{k}{\Theta}$ and $\card{S \tau'} \leq \card{S} \leq k$, this will imply that $\chi_{S \tau'} \in {\downarrow} \Theta$.
By Lemma~\ref{lem:chiT-chiTsigma}, we then get $\chi_S \in {\downarrow} \Theta$.
Moreover, $\chi_S \in {\downarrow^{[\leq k]}} \Theta = \Theta$, and we will have reached the desired contradiction.

It remains to show that $S \tau' \subseteq f^{-1}(1)$.
The argument proceeds in slightly different ways for the different classes $C$.

\begin{itemize}[wide]
\item
Case 1: $C = \clXI$.
Let $\vect{u} \in S \tau'$.
Then $\vect{u} = \vect{v} \tau'$ for some $\vect{v} \in S$.
We have $\vect{v} \in T \tau \cup \{\vect{1}\}$.
If $\vect{v} \in T \tau$, then $\vect{v} = \vect{t} \tau$ for some $\vect{t} \in T$; hence $\vect{u} = \vect{v} \tau' = \vect{t} \tau \tau' = \vect{t} \in T \subseteq f^{-1}(1)$.
If $\vect{v} = \vect{1}$, then $\vect{u} = \vect{1} \tau' = \vect{1} \in f^{-1}(1)$ because $f \in \clXI$.

\item
Case 2: $C = \clIX$.
The proof is similar to the previous case.

\item
Case 3: $C = \clM$.
Let $\vect{u} \in S \tau'$.
Then $\vect{u} = \vect{v} \tau'$ for some $\vect{v} \in S$.
By the definition of monotone closure, there exists a $\vect{w} \in T \tau$ such that $\vect{w} \leq \vect{v}$; moreover, $\vect{w} = \vect{t} \tau$ for some $\vect{t} \in T$.
Now, $\vect{u} = \vect{v} \tau' \geq \vect{w} \tau' = \vect{t} \tau \tau' = \vect{t} \in T \subseteq f^{-1}(1)$.
Since $f \in \clM$, it follows that $\vect{u} \in f^{-1}(1)$.

\item
Case 4: $C = \clMneg$.
The proof is similar to the previous case.

\item
Case 5: $C = \clRefl$.
Let $\vect{u} \in S \tau'$.
Then $\vect{u} = \vect{v} \tau'$ for some $\vect{v} \in S$.
By the definition of reflexive closure, there exists a $\vect{w} \in T \tau$ such that $\vect{v} \in \{ \vect{w}, \overline{\vect{w}} \}$.
Now, $\vect{w} = \vect{t} \tau$ for some $\vect{t} \in T$, and $\overline{\vect{w}} = \overline{\vect{t} \tau} = \overline{\vect{t}} \tau$.
Consequently, $\vect{u} = \vect{v} \tau' \in \{ \vect{w} \tau', \overline{\vect{w}} \tau' \} = \{ \vect{t} \tau \tau', \overline{\vect{t}} \tau \tau' \} = \{ \vect{t}, \overline{\vect{t}} \} \subseteq f^{-1}(1)$ because $\vect{t} \in T \subseteq f^{-1}(1)$ and $f$ is reflexive.
\qedhere
\end{itemize}
\end{proof}

We conclude this subsection with a description of the least and greatest elements and the atoms and coatoms of the lattice of $(k,C)$\hyp{}closed $(\clIc,\clMcUk{k})$\hyp{}clonoids.

\begin{proposition}
\label{prop:some-kC-closed}
Let $k \geq 2$.
\begin{enumerate}[label={\upshape(\roman*)}]
\item\label{prop:some-kC-closed:AllEmpty}
For each $C \in \{\clXI, \clIX, \clM, \clMneg, \clRefl\}$, $\clEmpty$ and $\clAll$ are the least and the greatest $(k,C)$\hyp{}closed $(\clIc,\clMcUk{k})$\hyp{}clonoids.
\item\label{prop:some-kC-closed:C0}
For $C \in \{\clM, \clMneg, \clRefl\}$, $\clVako$ is the least nonempty $(k,C)$\hyp{}closed $(\clIc,\clMcUk{k})$\hyp{}clonoid.
However, $\clVako$ is not $(k,C')$\hyp{}closed for $C' \in \{\clXI, \clIX\}$.
\item\label{prop:some-kC-closed:Uk}
For $C \in \{\clXI, \clM\}$, the least $(k,C)$\hyp{}closed $(\clIc,\clMcUk{k})$\hyp{}clonoid distinct from $\clEmpty$ and $\clVako$ is $\clUk{k}$.
\item\label{prop:some-kC-closed:Wkneg}
For $C \in \{\clIX, \clMneg\}$, the least $(k,C)$\hyp{}closed $(\clIc,\clMcUk{k})$\hyp{}clonoid distinct from $\clEmpty$ and $\clVako$ is $\clWkneg{k}$.
\item\label{prop:some-kC-closed:Ukplus}
The least $(k,\clRefl)$\hyp{}closed $(\clIc,\clMcUk{k})$\hyp{}clonoid distinct from $\clEmpty$ and $\clVako$ is $\clKlik{k}{\{\mathord{+}\}}$.
\item\label{prop:some-kC-closed:OX}
For $C \in \{\clXI, \clM\}$, the greatest $(k,C)$\hyp{}closed $(\clIc,\clMcUk{k})$\hyp{}clonoid distinct from $\clAll$ is $\clOX$.
\item\label{prop:some-kC-closed:XO}
For $C \in \{\clIX, \clMneg\}$, the greatest $(k,C)$\hyp{}closed $(\clIc,\clMcUk{k})$\hyp{}clonoid distinct from $\clAll$ is $\clXO$.
\item\label{prop:some-kC-closed:OO}
The greatest $(k,\clRefl)$\hyp{}closed $(\clIc,\clMcUk{k})$\hyp{}clonoid distinct from $\clAll$ is $\clOO$.
\end{enumerate}
\end{proposition}

\begin{proof}
Note first that the classes $\clAll$, $\clEmpty$, $\clVako$, $\clUk{k}$, $\clKlik{k}{\{\mathord{+}\}}$, $\clOX$, $\clXO$, and $\clOO$ are of the form $\clKlik{k}{\Phi}$ for some $\Phi$; see Example~\ref{ex:classes}.

\ref{prop:some-kC-closed:AllEmpty}
Clear.

\ref{prop:some-kC-closed:C0}
Clear, because every nonempty minorant minion includes $\clVako$ and $0^C = 0$ for each $C \in \{\clM, \clMneg, \clRefl\}$ and $0^\clXI = \id \notin \clVako$ and $0^\clIX = \neg \notin \clVako$.

\ref{prop:some-kC-closed:Uk}
For any $f \in \clAll$, $\id \minmin f^\clXI$.
Similarly, for any $g \in \clAll \setminus \clVako$, $\id \minmin g^\clM$.
Therefore, ${\downarrow^{[\leq k]}} \{\id\}$ must be the least nonempty element of $\Ideals(\posetAllleq{k}_\clXI)$ and the least element distinct from $\clEmpty$ and $\clVako$ of $\Ideals(\posetAllleq{k}_\clM)$.
Therefore $\clUk{k} = \clKlik{k}{\{\id\}}$ has the claimed property.

\ref{prop:some-kC-closed:Wkneg}
The proof is similar to \ref{prop:some-kC-closed:Uk}.

\ref{prop:some-kC-closed:Ukplus}
For any $f \in \clAll \setminus \clVako$, $\mathord{+} \minmin f^\clRefl$.
Therefore ${\downarrow^{[\leq k]}} \{\mathord{+}\}$ is the least element distinct from $\clEmpty$ and $\clVako$ of $\Ideals(\posetAllleq{k}_\clRefl)$.
Therefore $\clKlik{k}{\{\mathord{+}\}}$ has the claimed property.

\ref{prop:some-kC-closed:OX}
For $C \in \{\clXI, \clM\}$, $\clOX$ is $(k,C)$\hyp{}closed by Lemma~\ref{lem:k-1MR-closed} because $(\clOX)^\clXI = \clOI = \clOX \cap \clXI$ and $(\clOX)^\clM = \clMo = \clOX \cap \clM$.
Moreover, for any $f \in \clAll \setminus \clOX = \clIX$, we have $1 \minmin f^C$.
Because $\clKlik{k}{\{1\}} = \clAll$, it follows that $\clOX$ is the greatest $(k,C)$\hyp{}closed $(\clIc,\clMcUk{k})$\hyp{}clonoid distinct from $\clAll$.

\ref{prop:some-kC-closed:XO}
The proof is similar to \ref{prop:some-kC-closed:OX}.

\ref{prop:some-kC-closed:OO}
By Lemma~\ref{lem:k-1MR-closed}, $\clOO$ is $(k,\clRefl)$\hyp{}closed because $(\clOO)^\clRefl = \clReflOO = \clOO \cap \clRefl$.
Moreover, for any $f \in \clAll \setminus \clOO = \clEioo$, we have $1 \minmin f^\clRefl$.
Because $\clKlik{k}{\{1\}} = \clAll$, it follows that $\clOO$ is the greatest $(k,\clRefl)$\hyp{}closed $(\clIc,\clMcUk{k})$\hyp{}clonoid distinct from $\clAll$.
\end{proof}


\section{$(G,k)$-semibisectable functions}
\label{sec:semibisectable}
\renewcommand{\gendefault}{(\clIc,\clMcUk{k})}

\numberwithin{theorem}{section}

We now introduce one further tool that makes it relatively easy to determine whether a function belongs to the $(\clIc,\clMcUk{k})$\hyp{}clonoid generated by a given set of functions.

\begin{lemma}
\label{lem:extend-McUk}
Let $k \geq 2$.
Let $T$ and $F$ be nonempty subsets of $\{0,1\}^n$ such that
for all $\vect{u}_1, \dots, \vect{u}_k \in T$ it holds that $\vect{u}_1 \wedge \dots \wedge \vect{u}_k \neq \vect{0}$,
and for all $\vect{u} \in T$ and $\vect{v} \in F$ it holds that $\vect{u} \nleq \vect{v}$.
Then there exists an $n$\hyp{}ary function $f \in \clMcUk{k}$ such that $f(\vect{a}) = 1$ for all $\vect{a} \in T$ and $f(\vect{b}) = 0$ for all $\vect{b} \in F$.
\end{lemma}

\begin{proof}
Define $f \colon \{0,1\}^n \to \{0,1\}$ by the rule $f(\vect{a}) = 1$ if and only if $\vect{a} \in {\uparrow} T$.
We need to verify that $f \in \clMcUk{k}$ and $f(\vect{a}) = 1$ for all $\vect{a} \in T$ and $f(\vect{b}) = 0$ for all $\vect{b} \in F$.
It is immediate from the definition that $f \in \clM$ and $f(\vect{a}) = 1$ for all $\vect{a} \in T$.

Observe that ${\uparrow} T \cap {\downarrow} F = \emptyset$.
(Suppose to the contrary that there is an $\vect{a} \in {\uparrow} T \cap {\downarrow} F$.
Then there exist $\vect{u} \in T$ and $\vect{v} \in F$ such that $\vect{u} \leq \vect{a} \leq \vect{v}$, which contradicts the assumption that $\vect{u} \nleq \vect{v}$.)
Therefore $f(\vect{b}) = 0$ for all $\vect{b} \in F$.

Since $T \neq \emptyset$, we have $\vect{1} \in {\uparrow} T$, so $f(\vect{1}) = 1$.
Since $F \neq \emptyset$, we have $\vect{0} \in {\downarrow} F$, so $f(\vect{0}) = 0$.
Therefore $f \in \clOI$.

It remains to verify that $f \in \clUk{k}$.
Let $\vect{a}_1, \dots, \vect{a}_k \in f^{-1}(1) = {\uparrow} T$.
Then there exist $\vect{u}_1, \dots, \vect{u}_k \in T$ such that $\vect{u}_i \leq \vect{a}_i$ for all $i \in \nset{k}$.
It follows from our assumption that $\vect{0} \neq \vect{u}_1 \wedge \dots \wedge \vect{u}_k \leq \vect{a}_1 \wedge \dots \wedge \vect{a}_k$.
Therefore $f \in \clUk{k}$.
\end{proof}

\begin{definition}
\label{def:helpful}
Let $G$ be a set of Boolean functions, and for each $n \in \IN$, denote by $G_n$ the set of all $n$\hyp{}ary minors of functions in $G$.
Let $f$ be a nonconstant $n$\hyp{}ary Boolean function, and $k \in \IN_{+}$.
We say that $f$ is \emph{$(G,k)$\hyp{}semibisectable} if the following two conditions hold:
\begin{enumerate}[label={\upshape(\Alph*)}]
\item\label{helpful:k-true} For all $\vect{a}_1, \vect{a}_2, \dots, \vect{a}_k \in f^{-1}(1)$ there exists a $\tau \in G_n$ such that $\tau(\vect{a}_1) = \tau(\vect{a}_2) = \dots = \tau(\vect{a}_k) = 1$.
\item\label{helpful:both} For all $\vect{a} \in f^{-1}(1)$ and for all $\vect{b} \in f^{-1}(0)$ there exists a $\tau \in G_n$ such that $\tau(\vect{a}) = 1$ and $\tau(\vect{b}) = 0$.
\end{enumerate}
We say that a class $C \subseteq \clAll$ is \emph{$(G,k)$\hyp{}semibisectable} if every function in $C$ is $(G,k)$\hyp{}semibisectable.
\end{definition}

\begin{lemma}
\label{lem:helpful}
Let $G \subseteq \clAll$, $f \in \clAll \setminus \clVak$, and $k \geq 2$.
If $f$ is $(G,k)$\hyp{}semibisectable, then $f \in \gen{G}$.
\end{lemma}

\begin{proof}
Let $f \in \clAll \setminus \clVak$ with $\arity{f} = n$, and assume that $f$ is $(G,k)$\hyp{}semibisectable.
Let $N$ be the cardinality of $G_n$ (there are only a finite number of $n$\hyp{}ary Boolean functions, so the set $G_n$ is certainly finite),
and assume that $\varphi_1, \dots, \varphi_N$ is an enumeration of the functions in $G_n$ in some fixed order.
Let $\varphi \colon \{0,1\}^n \to \{0,1\}^N$, $\varphi(\vect{a}) := (\varphi_1(\vect{a}), \dots, \varphi_N(\vect{a}))$.
Let $T := \varphi(f^{-1}(1))$ and $F := \varphi(f^{-1}(0))$.
Let $\vect{u}_1, \dots, \vect{u}_k \in T$; then there exist $\vect{a}_1, \dots, \vect{a}_k \in f^{-1}(1)$ such that $\varphi(\vect{a}_i) = \vect{u}_i$ for each $i \in \nset{k}$.
By condition \ref{helpful:k-true}, there exists an index $t \in \nset{N}$ such that $\varphi_t(\vect{a}_1) = \dots = \varphi_t(\vect{a}_k) = 1$; hence $\vect{u}_1 \wedge \dots \wedge \vect{u}_k \neq \vect{0}$.
Let now $\vect{u} \in T$ and $\vect{v} \in F$.
Then there exist $\vect{a} \in f^{-1}(1)$ and $\vect{b} \in f^{-1}(0)$ such that $\varphi(\vect{a}) = \vect{u}$ and $\varphi(\vect{b}) = \vect{v}$.
By condition \ref{helpful:both}, there exists an index $s \in \nset{N}$ such that $\varphi_s(\vect{a}) = 1$ and $\varphi_s(\vect{b}) = 0$; hence $\vect{u} = \varphi(\vect{a}) \nleq \varphi(\vect{b}) = \vect{v}$.
Therefore the sets $T$ and $F$ satisfy the hypotheses of Lemma~\ref{lem:extend-McUk}, and it follows that there exists an $N$\hyp{}ary function $h \in \clMcUk{k}$ such that $T \subseteq h^{-1}(1)$ and $F \subseteq h^{-1}(0)$.
Then it holds that $f = h \circ \varphi = h(\varphi_1, \dots, \varphi_N)$.
Since $h \in \clMcUk{k}$ and $\varphi_1, \dots, \varphi_N \in G \, \clIc$, we conclude with the help of Lemma~\ref{lem:F-closure} that $f \in \clMcUk{k} (G \, \clIc) = \gen{G}$.
\end{proof}


\section{Main result: $(\clIc,\mathsf{MU}^k_{01})$\hyp{}clonoids}
\label{sec:IcMcUk}
\numberwithin{theorem}{subsection}
\renewcommand{\gendefault}{(\clIc,\clMcUk{k})}

We are now ready to present and prove our main result: a complete description of the $(\clIc,\clMcUk{k})$\hyp{}clonoids for $k \geq 2$.

\subsection{The $(\clIc,\mathsf{MU}^k_{01})$\hyp{}clonoids}

\begin{theorem}
\label{thm:IcMcUk-clonoids}
For $k \in \IN$ with $k \geq 2$, the $(\clIc,\clMcUk{k})$\hyp{}clonoids are the following:
\begin{enumerate}[label={\upshape(\alph*)}]
\item\label{MAIN:klik}
$\clKlik{k}{\Theta}$ for each $\Theta \in \Ideals(\posetAllleq{k})$,
\item\label{MAIN:klikXI}
$\clKlik{k}{\Theta} \cap (\clOICO)$ and $\clKlik{k}{\Theta} \cap \clOI$ for each nonempty $\Theta \in \Ideals(\posetAllleq{k}_{\clXI})$,
\item\label{MAIN:klikIX}
$\clKlik{k}{\Theta} \cap (\clIOCO)$ and $\clKlik{k}{\Theta} \cap \clIO$ for each nonempty $\Theta \in \Ideals(\posetAllleq{k}_{\clIX})$,
\item\label{MAIN:klikM}
$\clKlik{k}{\Theta} \cap \clMo$ and $\clKlik{k}{\Theta} \cap \clMc$ for each nonempty $\Theta \in \Ideals(\posetAllleq{k}_{\clM})$,
\item\label{MAIN:klikMneg}
$\clKlik{k}{\Theta} \cap \clMineg$ and $\clKlik{k}{\Theta} \cap \clMcneg$ for each nonempty $\Theta \in \Ideals(\posetAllleq{k}_{\clMneg})$,
\item\label{MAIN:klikR}
$\clKlik{k}{\Theta} \cap \clReflOO$ for each nonempty $\Theta \in \Ideals(\posetAllleq{k}_{\clRefl})$,
\item\label{MAIN:rest}
$\clEiio$, $\clEioi$,
$\clEq$,
$\clOXC$, $\clXOC$,
$\clIXC$, $\clXIC$,
$\clOOC$, $\clOIC$, $\clIOC$,
$\clIIC$,
$\clOICI$, $\clIOCI$,
$\clIX$, $\clXI$,
$\clII$,
$\clM$, $\clMi$, $\clMneg$, $\clMoneg$,
$\clRefl$, $\clReflOOC$, $\clReflIIC$, $\clReflII$,
$\clVak$,
$\clVaki$.
\end{enumerate}
\end{theorem}

\begin{figure}
\begin{center}
\scalebox{0.25}{
\tikzstyle{every node}=[circle, draw, fill=black, scale=1, font=\huge]
\tikzstyle{minclosed}=[fill=black]
\tikzstyle{forallk}=[fill=black]
\tikzstyle{Icl}=[fill=black]
\tikzstyle{Ocl}=[fill=black]
\tikzstyle{Mcl}=[fill=black]
\tikzstyle{Mnegcl}=[fill=black]
\tikzstyle{Rcl}=[fill=black]
\pgfdeclarelayer{poset}
\pgfdeclarelayer{blobs}
\pgfsetlayers{blobs,poset}
\begin{tikzpicture}[baseline, scale=1]
\begin{pgfonlayer}{poset}
\coordinate (bottom) at (0,0);
\coordinate (sw) at ($(bottom)+(-20,9)$);
\coordinate (se) at ($(bottom)+(20,9)$);
\coordinate (top) at ($(bottom)+(0,50)$);
\coordinate (nw) at ($(top)+(0,-9)$);
\coordinate (ne) at ($(top)+(20,-9)$);
\node[minclosed] (empty) at ($(bottom)+(-9,0)+(-6,-5)$) {};
\draw ($(empty)+(270:0.7)$) node[draw=none,fill=none]{$\clEmpty$};
\node[minclosed] (D0C0) at ($(bottom)+(-13,4)+(-6,-5)$) {};
\draw ($(D0C0)+(270:0.7)$) node[draw=none,fill=none]{$\clVako$};
\node[forallk] (D0C1) at ($(bottom)+(16,4)+(0,-5)$) {};
\draw ($(D0C1)+(270:0.7)$) node[draw=none,fill=none]{$\clVaki$};
\node[forallk] (D0) at ($(bottom)+(12,8)+(0,-5)$) {};
\draw ($(D0)+(270:0.7)$) node[draw=none,fill=none]{$\clVak$};
\node[minclosed] (All) at ($(top)+(4,0)$) {};
\draw ($(All)+(90:0.7)$) node[draw=none,fill=none]{$\clAll$};
\node[forallk] (Eiio) at ($(top)+(4,-4)$) {};
\draw ($(Eiio)+(135:0.8)$) node[draw=none,fill=none]{$\clEiio$};
\node[minclosed] (Eiii) at ($(top)+(-6,-8)$) {};
\draw ($(Eiii)+(135:1)$) node[draw=none,fill=none]{$\clEiii$};
\node[forallk] (Eioi) at ($(top)+(12,-4)$) {};
\draw ($(Eioi)+(45:0.8)$) node[draw=none,fill=none]{$\clEioi$};
\node[forallk] (C0D0) at ($(top)+(0,-8)$) {};
\draw ($(C0D0)+(160:1.2)$) node[draw=none,fill=none]{$\clOXC$};
\node[forallk] (E1D0) at ($(top)+(4,-8)$) {};
\draw ($(E1D0)+(160:1.2)$) node[draw=none,fill=none]{$\clXIC$};
\node[forallk] (P0) at ($(top)+(8,-8)$) {};
\draw ($(P0)+(180:0.7)$) node[draw=none,fill=none]{$\clEq$};
\node[forallk] (E0D0) at ($(top)+(12,-8)$) {};
\draw ($(E0D0)+(0:1.3)$) node[draw=none,fill=none]{$\clXOC$};
\node[forallk] (C1D0) at ($(top)+(16,-8)$) {};
\draw ($(C1D0)+(20:1.2)$) node[draw=none,fill=none]{$\clIXC$};
\node[forallk] (C1) at ($(C1D0)+(2,-2)$) {};
\draw ($(C1)+(0:0.7)$) node[draw=none,fill=none]{$\clIX$};
\node[forallk] (E1) at ($(E1D0)+(0,-2)$) {};
\draw ($(E1)+(10:0.8)$) node[draw=none,fill=none]{$\clXI$};
\node[forallk] (C1E1D0) at ($(top)+(18,-16)$) {};
\draw ($(C1E1D0)+(80:1.2)$) node[draw=none,fill=none,rotate=284]{$\clIIC$};
\node[forallk] (C1E1) at ($(C1E1D0)+(2,-2)$) {};
\draw ($(C1E1)+(0:0.8)$) node[draw=none,fill=none]{$\clII$};
\node[forallk] (C0E0D0) at ($(E1)+1.5*(0,-2)$) {};
\draw ($(C0E0D0)+(180:1.25)$) node[draw=none,fill=none]{$\clOOC$};
\node[Icl] (C0E1) at ($(bottom)+(-16,18)$) {};
\draw ($(C0E1)+(90:0.8)$) node[draw=none,fill=none]{$\clOI$};
\node[Icl] (C0E1D000) at ($(C0E1)+(-1.2,1.2)$) {};
\draw ($(C0E1D000)+(102:1.1)$) node[draw=none,fill=none,rotate=66.8]{$\clOICO$};
\node[forallk] (C0E1D011) at ($(C0E1)+(1.2+0.5,1.2)$) {};
\draw ($(C0E1D011)+(350:1.4)$) node[draw=none,fill=none]{$\clOICI$};
\node[forallk] (C0E1D0) at ($(C0E1)+(0+0.5,2*1.2)$) {};
\draw ($(C0E1D0)+(102:1)$) node[draw=none,fill=none,rotate=66.8]{$\clOIC$};
\node[minclosed] (C0) at ($(C0E1D000)+5*(1.2,2.8)$) {};
\draw ($(C0)+(90:1.0)$) node[draw=none,fill=none]{$\clOX$};
\node[Icl] (XIclosed1OI) at ($(C0E1)+(-1.2-0.5,-1.2)$) {};
\node[Icl] (XIclosed1OIC0) at ($(XIclosed1OI)+(-1.2,1.2)$) {};
\node[minclosed] (XIclosed1) at ($(XIclosed1OIC0)+5*(1.2,2.8)$) {};
\node[Icl] (XIclosed2OI) at ($(C0E1)+2*(-1.2-0.5,-1.2)$) {};
\draw ($(XIclosed2OI)+(0:1.5)$) node[draw=none,fill=none]{$\clKlik{k}{\Theta} \cap \clOI$};
\node[Icl] (XIclosed2OIC0) at ($(XIclosed2OI)+(-1.2,1.2)$) {};
\draw ($(XIclosed2OIC0)+(80:1.9)$) node[draw=none,fill=none,rotate=66.8]{$\clKlik{k}{\Theta} \cap (\clOICO)$};
\node[minclosed] (XIclosed2) at ($(XIclosed2OIC0)+5*(1.2,2.8)$) {};
\draw ($(XIclosed2)+(75:1.0)$) node[draw=none,fill=none]{$\clKlik{k}{\Theta}$};
\node[Mcl] (XIclosed2Mc) at ($(XIclosed2OI)+1.5*(0.6,-2.4)$) {};
\draw ($(XIclosed2Mc)+(0:1.5)$) node[draw=none,fill=none]{$\clKlik{k}{\Theta} \cap \clMc$};
\node[Mcl] (XIclosed2M0) at ($(XIclosed2OIC0)+1.5*(0.6,-2.4)$) {};
\draw ($(XIclosed2M0)+(-.7,-.7)$) node[draw=none,fill=none]{$\clKlik{k}{\Theta} \cap \clMo$};
\node[Icl] (XIclosed3OI) at ($(C0E1)+3*(-1.2-0.5,-1.2)$) {};
\node[Icl] (XIclosed3OIC0) at ($(XIclosed3OI)+(-1.2,1.2)$) {};
\node[minclosed] (XIclosed3) at ($(XIclosed3OIC0)+5*(1.2,2.8)$) {};
\node[Icl] (TcU2) at ($(C0E1)+4*(-1.2-0.5,-1.2)$) {};
\draw ($(TcU2)+(90:0.8)$) node[draw=none,fill=none]{$\clTcUk{k}$};
\node[Icl] (TcU2D0) at ($(TcU2)+(-1.2,1.2)$) {};
\draw ($(TcU2D0)+(180:1.3)$) node[draw=none,fill=none]{$\clTcUkCO{k}$};
\node[minclosed] (U2) at ($(TcU2D0)+5*(1.2,2.8)$) {};
\draw ($(U2)+(135:0.9)$) node[draw=none,fill=none]{$\clUk{k}$};
\node[Mcl] (Mc) at ($(C0E1)+1.5*(0.6,-2.4)$) {};
\draw ($(Mc)+(300:0.8)$) node[draw=none,fill=none]{$\clMc$};
\node[Mcl] (M0) at ($(C0E1D000)+1.5*(0.6,-2.4)$) {};
\draw ($(M0)+(270:0.8)$) node[draw=none,fill=none]{$\clMo$};
\node[forallk] (M1) at ($(C0E1D011)+1.5*(0.6,-2.4)$) {};
\draw ($(M1)+(290:0.8)$) node[draw=none,fill=none]{$\clMi$};
\node[forallk] (M) at ($(C0E1D0)+1.5*(0.6,-2.4)$) {};
\draw ($(M)+(260:0.7)$) node[draw=none,fill=none]{$\clM$};
\node[Mcl] (MU2) at ($(TcU2D0)+1.5*(0.6,-2.4)$) {};
\draw ($(MU2)+(180:0.9)$) node[draw=none,fill=none]{$\clMUk{k}$};
\node[Mcl] (McU2) at ($(TcU2)+1.5*(0.6,-2.4)$) {};
\draw ($(McU2)+(-15:1.1)$) node[draw=none,fill=none]{$\clMcUk{k}$};
\node[Ocl] (C1E0) at ($(bottom)+(8,18)$) {};
\draw ($(C1E0)+(75:0.8)$) node[draw=none,fill=none]{$\clIO$};
\node[forallk] (C1E0D011) at ($(C1E0)+(1.2+0.5,1.2)$) {};
\draw ($(C1E0D011)+(0:1.4)$) node[draw=none,fill=none]{$\clIOCI$};
\node[Ocl] (C1E0D000) at ($(C1E0)+(-1.2,1.2)$) {};
\draw ($(C1E0D000)+(140:1.4)$) node[draw=none,fill=none,rotate=315]{$\clIOCO$};
\node[forallk] (C1E0D0) at ($(C1E0)+(0+0.5,2*1.2)$) {};
\draw ($(C1E0D0)+(140:1.4)$) node[draw=none,fill=none,rotate=315]{$\clIOC$};
\node[minclosed] (E0) at ($(C1E0D000)+4*(-0.6,2.4)$) {};
\draw ($(E0)+(110:1.0)$) node[draw=none,fill=none]{$\clXO$};
\node[Icl] (IXclosed1IO) at ($(C1E0)+(-1.2-0.5,-1.2)$) {};
\node[Icl] (IXclosed1IOC0) at ($(IXclosed1IO)+(-1.2,1.2)$) {};
\node[minclosed] (IXclosed1) at ($(IXclosed1IOC0)+4*(-0.6,2.4)$) {};
\node[Icl] (IXclosed2IO) at ($(C1E0)+2*(-1.2-0.5,-1.2)$) {};
\draw ($(IXclosed2IO)+(1.4,0)$) node[draw=none,fill=none]{$\clKlik{k}{\overline{\Theta}} \cap \clIO$};
\node[Icl] (IXclosed2IOC0) at ($(IXclosed2IO)+(-1.2,1.2)$) {};
\draw ($(IXclosed2IOC0)+(-.6,0)+(104:1.4)$) node[draw=none,fill=none,rotate=284]{$\clKlik{k}{\overline{\Theta}} \cap (\clIOCO)$};
\node[minclosed] (IXclosed2) at ($(IXclosed2IOC0)+4*(-0.6,2.4)$) {};
\draw ($(IXclosed2)+(75:1.0)$) node[draw=none,fill=none]{$\clKlik{k}{\overline{\Theta}}$};
\node[Mcl] (IXclosed2Mcneg) at ($(IXclosed2IO)+1.5*(0.6,-2.4)$) {};
\draw ($(IXclosed2Mcneg)+(1.5,0)$) node[draw=none,fill=none]{$\clKlik{k}{\overline{\Theta}} \cap \clMcneg$};
\node[Mcl] (IXclosed2M1neg) at ($(IXclosed2IOC0)+1.5*(0.6,-2.4)$) {};
\draw ($(IXclosed2M1neg)+(-.7,-.7)$) node[draw=none,fill=none]{$\clKlik{k}{\overline{\Theta}} \cap \clMineg$};
\node[Icl] (IXclosed3IO) at ($(C1E0)+3*(-1.2-0.5,-1.2)$) {};
\node[Icl] (IXclosed3IOC0) at ($(IXclosed3IO)+(-1.2,1.2)$) {};
\node[minclosed] (IXclosed3) at ($(IXclosed3IOC0)+4*(-0.6,2.4)$) {};
\node[Ocl] (TcW2neg) at ($(C1E0)+4*(-1.2-0.5,-1.2)$) {};
\draw ($(TcW2neg)+(70:0.8)$) node[draw=none,fill=none]{$\clTcWkneg{k}$};
\node[Ocl] (TcW2negD0) at ($(TcW2neg)+(-1.2,1.2)$) {};
\draw ($(TcW2negD0)+(195:1.4)$) node[draw=none,fill=none]{$\clTcWknegCO{k}$};
\node[minclosed] (W2neg) at ($(TcW2negD0)+4*(-0.6,2.4)$) {};
\draw ($(W2neg)+(120:1.0)$) node[draw=none,fill=none]{$\clWkneg{k}$};
\node[Mnegcl] (Mcneg) at ($(C1E0)+1.5*(0.6,-2.4)$) {};
\draw ($(Mcneg)+(280:0.8)$) node[draw=none,fill=none]{$\clMcneg$};
\node[forallk] (M0neg) at ($(C1E0D011)+1.5*(0.6,-2.4)$) {};
\draw ($(M0neg)+(10:0.8)$) node[draw=none,fill=none]{$\clMoneg$};
\node[Mnegcl] (M1neg) at ($(C1E0D000)+1.5*(0.6,-2.4)$) {};
\draw ($(M1neg)+(270:0.8)$) node[draw=none,fill=none]{$\clMineg$};
\node[forallk] (Mneg) at ($(C1E0D0)+1.5*(0.6,-2.4)$) {};
\draw ($(Mneg)+(250:0.8)$) node[draw=none,fill=none]{$\clMneg$};
\node[Mnegcl] (MW2neg) at ($(TcW2negD0)+1.5*(0.6,-2.4)$) {};
\draw ($(MW2neg)+(165:1.1)$) node[draw=none,fill=none]{$\clMWkneg{k}$};
\node[Mnegcl] (McW2neg) at ($(TcW2neg)+1.5*(0.6,-2.4)$) {};
\draw ($(McW2neg)+(320:1.0)$) node[draw=none,fill=none]{$\clMcWkneg{k}$};
\node[minclosed] (UkWkneg) at ($(bottom)+(-9,17)$) {};
\draw ($(UkWkneg)+(190:1.4)$) node[draw=none,fill=none]{$\clUk{k} \cap \clWkneg{k}$};
\node[forallk] (X1P0) at ($(bottom)+(17,31)$) {};
\draw ($(X1P0)+(45:0.6)$) node[draw=none,fill=none]{$\clRefl$};
\node[forallk] (X1C0E0D0) at ($(X1P0)+2*(-1,-12/17)$) {};
\draw ($(X1C0E0D0)+(345:1.2)$) node[draw=none,fill=none]{$\clReflOOC$};
\node[Rcl] (X1C0E0) at ($(X1P0)+20*(-1,-12/17)$) {};
\draw ($(X1C0E0)+(0.6,-0.6)$) node[draw=none,fill=none]{$\clReflOO$};
\node[forallk] (X1C1E1D0) at ($(X1P0)+2*(1,-12/17)$) {};
\draw ($(X1C1E1D0)+(80:1.2)$) node[draw=none,fill=none,rotate=284]{$\clReflIIC$};
\node[forallk] (X1C1E1) at ($(X1P0)+4*(1,-12/17)$) {};
\draw ($(X1C1E1)+(0:0.8)$) node[draw=none,fill=none]{$\clReflII$};
\node[minclosed] (C0E0) at ($(X1C0E0)+3.4*(-0.6,2.4)$) {};
\draw ($(C0E0)+(110:0.9)$) node[draw=none,fill=none]{$\clOO$};
\node[Rcl] (smallestRefl) at ($(X1C0E0)+2*(-1.2-0.5,-1.2)$) {};
\draw ($(smallestRefl)+(168:1.55)$) node[draw=none,fill=none]{$\clKlik{k}{\{+\}} \cap \clReflOO$};
\node[Rcl] (smallestReflclosed) at ($(smallestRefl)+3.4*(-0.6,2.4)$) {};
\draw ($(smallestReflclosed)+(205:1.0)$) node[draw=none,fill=none]{$\clKlik{k}{\{+\}}$};
\node[Rcl] (arbRefl) at ($(smallestRefl)!0.5!(X1C0E0)$) {};
\node[Rcl] (arbReflclosed) at ($(smallestReflclosed)!0.5!(C0E0)$) {};

\foreach \u/\v in {
   empty/D0C0, empty/D0C1, D0C0/D0, D0C1/D0,
   D0/X1C0E0D0, D0/X1C1E1D0,
   D0C0/MU2,
   X1C0E0/X1C0E0D0, X1C0E0D0/X1P0, X1C0E0/C0E0, X1C1E1/X1C1E1D0, X1C1E1D0/X1P0, X1C1E1/C1E1, 
   D0C1/X1C1E1,
   McU2/MU2, McU2/TcU2,
   MU2/TcU2D0, TcU2/TcU2D0, TcU2D0/U2,
   Mc/M0, Mc/M1, M0/M, M1/M,
   McW2neg/MW2neg, McW2neg/TcW2neg, 
   MW2neg/TcW2negD0, TcW2neg/TcW2negD0, TcW2negD0/W2neg, Mcneg/M0neg, Mcneg/M1neg, M0neg/Mneg, M1neg/Mneg,
   Mc/C0E1,
   Mcneg/C1E0,
   M0/C0E1D000, M1/C0E1D011, M/C0E1D0, M0neg/C1E0D011, M1neg/C1E0D000, Mneg/C1E0D0,
   X1C0E0D0/C0E0D0, X1C1E1D0/C1E1D0,
   X1P0/P0,
   C0E1D000/C0, C0E1D011/E1,
   C1E0D011/C1, C1E0D000/E0,
   C1E1/C1, C1E1/E1,
   C0E0/C0E0D0, C1E1/C1E1D0, C0E1/C0E1D000, C0E1/C0E1D011, C0E1D000/C0E1D0, C0E1D011/C0E1D0, C1E0/C1E0D000, C1E0/C1E0D011, C1E0D000/C1E0D0, C1E0D011/C1E0D0,
   C0/C0D0, C1/C1D0, E0/E0D0, E1/E1D0,
   C0E0D0/C0D0, C0E0D0/E0D0, C1E1D0/C1D0, C1E1D0/E1D0, C0E1D0/C0D0, C0E1D0/E1D0, C1E0D0/C1D0, C1E0D0/E0D0,
   C0E0D0/P0, C1E1D0/P0,
   C0D0/Eiio, C1D0/Eioi, E0D0/Eioi, E1D0/Eiio,
   P0/Eioi,
   Eiio/All, Eiii/All,
   Eioi/All,
XIclosed1OI/XIclosed1OIC0, XIclosed1OIC0/XIclosed1,
XIclosed2OI/XIclosed2OIC0, XIclosed2OIC0/XIclosed2, XIclosed2Mc/XIclosed2M0, XIclosed2Mc/XIclosed2OI, XIclosed2M0/XIclosed2OIC0,
XIclosed3OI/XIclosed3OIC0, XIclosed3OIC0/XIclosed3,
IXclosed1IO/IXclosed1IOC0, IXclosed1IOC0/IXclosed1,
IXclosed2IO/IXclosed2IOC0, IXclosed2IOC0/IXclosed2, IXclosed2Mcneg/IXclosed2M1neg, IXclosed2Mcneg/IXclosed2IO, IXclosed2M1neg/IXclosed2IOC0,
IXclosed3IO/IXclosed3IOC0, IXclosed3IOC0/IXclosed3,
smallestRefl/smallestReflclosed,
D0C0/smallestRefl,
arbRefl/arbReflclosed,
   D0C0/UkWkneg,
   D0/M, D0/Mneg,
   D0C0/MW2neg,
   empty/McU2, empty/McW2neg,
   D0C1/M1, D0C1/M0neg, D0/Mneg, P0/Eiio%
}
{
   \draw [thick] (\u) -- (\v);
}
\draw [thin] (UkWkneg) to[out=29,in=323,looseness=2.2] (Eiii.center) to[out=217,in=142,looseness=2] (UkWkneg) -- cycle;
\foreach \u/\v in {
C0E1/TcU2.center, C0E1D000/TcU2D0.center, Mc/McU2.center, M0/MU2.center,
C1E0/TcW2neg.center, C1E0D000/TcW2negD0.center, Mcneg/McW2neg.center, M1neg/MW2neg.center}
{
\draw [thin] (\u) to[bend left=12] (\v) to[bend left=12] (\u) -- cycle;
}
\foreach \u/\v in {
U2/C0.center, W2neg/E0.center}
{
\draw [thin] (\u) to[bend left=27] (\v) to[bend left=27] (\u) -- cycle;
}
\foreach \u/\v in {
smallestRefl/X1C0E0.center}
{
\draw [thin] (\u) to[bend left=17] (\v) to[bend left=17] (\u) -- cycle;
}
\foreach \u/\v in {
smallestReflclosed/C0E0.center}
{
\draw [thin] (\u) to[bend left=30] (\v) to[bend left=30] (\u) -- cycle;
}
\end{pgfonlayer}{poset}
\begin{pgfonlayer}{blobs}
    \draw[thin,draw=yellow!60!black,fill=yellow!25] plot[smooth cycle] coordinates{($(All)+(-1,1)$) ($(Eiii)+(-1.5,1)$) ($(U2)+(-2,0.5)$)
   ($(UkWkneg)+(-2.6,-1.2)$) ($(D0C0)+(-1.5,-0.25)$) ($(empty)+(1,-1)$) ($(UkWkneg)+(1,-2)$) ($(E0)+(1,-3)$) ($(E0)+(1.7,2.7)$) ($(C0D0)+(-2.75,-0.5)$) ($(All)+(1,-1)$)};
    \draw[thin,draw=green!60!black,fill=green!25] ($(smallestRefl)+(315:0.75)$) -- ($(X1C0E0)+(315:0.75)$) arc (315:495:0.75) -- ($(smallestRefl)+(135:0.75)$) arc (135:315:0.75) -- cycle;
    \draw[thin,draw=green!60!black,fill=green!25] (smallestReflclosed) circle [radius=0.5];
    \draw[thin,draw=green!60!black,fill=green!25] (arbReflclosed) circle [radius=0.5];
    \draw[thin,draw=green!60!black,fill=green!25] (C0E0) circle [radius=0.5];
    \draw[thin,draw=blue!60!black,fill=blue!25] ($(C0E1)+(35:0.75)$) -- ($(C0E1D000)+(35:0.75)$) arc (35:135:0.75) -- ($(TcU2D0)+(135:0.75)$) arc (135:215:0.75) -- ($(TcU2)+(215:0.75)$) arc (215:315:0.75) -- ($(C0E1)+(315:0.75)$) arc (315:395:0.75);
    \draw[thin,draw=magenta!60!black,fill=magenta!22] ($(Mc)+(35:0.75)$) -- ($(M0)+(35:0.75)$) arc (35:135:0.75) -- ($(MU2)+(135:0.75)$) arc (135:215:0.75) -- ($(McU2)+(215:0.75)$) arc (215:315:0.75) -- ($(Mc)+(315:0.75)$) arc (315:395:0.75);
    \draw[thin,draw=blue!60!black,fill=blue!25] ($(C1E0)+(35:0.75)$) -- ($(C1E0D000)+(35:0.75)$) arc (35:135:0.75) -- ($(TcW2negD0)+(135:0.75)$) arc (135:215:0.75) -- ($(TcW2neg)+(215:0.75)$) arc (215:315:0.75) -- ($(C1E0)+(315:0.75)$) arc (315:395:0.75);
    \draw[thin,draw=magenta!60!black,fill=magenta!22] ($(Mcneg)+(35:0.75)$) -- ($(M1neg)+(35:0.75)$) arc (35:135:0.75) -- ($(MW2neg)+(135:0.75)$) arc (135:215:0.75) -- ($(McW2neg)+(215:0.75)$) arc (215:315:0.75) -- ($(Mcneg)+(315:0.75)$) arc (315:395:0.75);
    \draw[thin,draw=blue!60!black,fill=blue!22] (C0) circle [radius=0.6];
    \draw[thin,draw=magenta!60!black,fill=magenta!22] (C0) circle [radius=0.4];
    \draw[thin,draw=blue!60!black,fill=blue!22] (XIclosed1) circle [radius=0.6];
    \draw[thin,draw=blue!60!black,fill=blue!22] (XIclosed2) circle [radius=0.6];
    \draw[thin,draw=blue!60!black,fill=blue!22] (XIclosed3) circle [radius=0.6];
    \draw[thin,draw=magenta!60!black,fill=magenta!22] (XIclosed2) circle [radius=0.4];
    \draw[thin,draw=blue!60!black,fill=blue!22] (U2) circle [radius=0.6];
    \draw[thin,draw=magenta!60!black,fill=magenta!22] (U2) circle [radius=0.4];
    \draw[thin,draw=blue!60!black,fill=blue!22] (E0) circle [radius=0.6];
    \draw[thin,draw=magenta!60!black,fill=magenta!22] (E0) circle [radius=0.4];
    \draw[thin,draw=blue!60!black,fill=blue!22] (IXclosed1) circle [radius=0.6];
    \draw[thin,draw=blue!60!black,fill=blue!22] (IXclosed2) circle [radius=0.6];
    \draw[thin,draw=blue!60!black,fill=blue!22] (IXclosed3) circle [radius=0.6];
    \draw[thin,draw=magenta!60!black,fill=magenta!22] (IXclosed2) circle [radius=0.4];
    \draw[thin,draw=blue!60!black,fill=blue!22] (W2neg) circle [radius=0.6];
    \draw[thin,draw=magenta!60!black,fill=magenta!22] (W2neg) circle [radius=0.4];
\end{pgfonlayer}{blobs}
\end{tikzpicture}
}
\end{center}
\caption{A schematic Hasse diagram of the poset of $(\clIc,\protect\clMcUk{k})$\hyp{}clonoids.}
\label{fig:McUk-stable}
\end{figure}

\begin{remark}
Note that $\clAll$, $\clVako$, $\clEmpty$ are included in item \ref{MAIN:klik} of Theorem~\ref{thm:IcMcUk-clonoids}.
Note also that we get the classes $\clOICO$, $\clOI$, $\clIOCO$, $\clIO$, $\clMo$, $\clMc$, $\clMineg$, $\clMcneg$, $\clReflOO$ from items \ref{MAIN:klikXI}--\ref{MAIN:klikR}, because $\clAllleq{k} \in \Ideals(\posetAllleq{k}_C)$ and $\clKlik{k}{\clAllleq{k}} = \clAll$.
\end{remark}

\begin{remark}
\label{rem:IcMcWk-clonoids}
From Theorem~\ref{thm:IcMcUk-clonoids}, we obtain additionally a description of the $(\clIc,\clMcWk{k})$\hyp{}clonoids.
Because $\clIc^\mathrm{d} = \clIc$ and $(\clMcUk{k})^\mathrm{d} = \clMcWk{k}$, 
Lemma~\ref{lem:duality} tells us that $(\clIc,\clMcWk{k})$\hyp{}clonoids are precisely the duals of $(\clIc,\clMcUk{k})$\hyp{}clonoids.
\end{remark}

A schematic Hasse diagram of $\closys{(\clIc,\clMcUk{k})}$ is presented in Figure~\ref{fig:McUk-stable}.
Parts of the diagram are shaded with different colours; these correspond to the different items of Theorem~\ref{thm:IcMcUk-clonoids}.
The big yellow part comprises the $(\clIc,\clMcUk{k})$\hyp{}clonoids of the form $\clKlik{k}{\Theta}$ (item \ref{MAIN:klik}); by Theorem~\ref{thm:Uktheta-lattice}, this sublattice is isomorphic to $\Ideals(\posetAllleq{k})$ (not all elements are shown in the diagram).
The blue vertices within the yellow part represent the $(k,\clXI)$- and $(k,\clIX)$\hyp{}closed clonoids distinct from $\clEmpty$ and $\clAll$; their intersections with $\clOICO$, $\clOI$, $\clIOCO$, and $\clIO$ (items \ref{MAIN:klikXI} and \ref{MAIN:klikIX}) form the blue intervals outside the yellow part.
The magenta vertices\footnote{Note that every $(k,\clM)$\hyp{}closed ($(k,\clMneg)$\hyp{}closed, resp.)\ clonoid distinct from $\clVako$ is also $(k,\clXI)$\hyp{}closed ($(k,\clIX)$\hyp{}closed, resp.);\ hence each magenta vertex is also blue.} within the yellow part represent the $(k,\clM)$- and $(k,\clMneg)$\hyp{}closed clonoids distinct from $\clEmpty$, $\clVako$, and $\clAll$; their intersections with $\clMo$, $\clMc$, $\clMineg$, and $\clMcneg$ (items \ref{MAIN:klikM} and \ref{MAIN:klikMneg}) form the magenta intervals outside the yellow part.
The green vertices within the yellow part represent the $(k,\clRefl)$\hyp{}closed clonoids distinct from $\clEmpty$, $\clVako$, and $\clAll$; their intersections with $\clReflOO$ (item \ref{MAIN:klikR}) form the green interval outside the yellow part.
(The least and the greatest $(k,C)$\hyp{}closed clonoids distinct from $\clEmpty$, $\clVako$, and $\clAll$ are described in Proposition~\ref{prop:some-kC-closed}.)
The unshaded parts are the remaining clonoids listed in item \ref{MAIN:rest}.

The remainder of this section is devoted to the proof of Theorem~\ref{thm:IcMcUk-clonoids}.
The proof has two parts.
(1) We show that the classes listed in the statement of the theorem are $(\clIc,\clMcUk{k})$\hyp{}clonoids.
This is the easier part of the proof.
(2) We show that there are no other $(\clIc,\clMcUk{k})$\hyp{}clonoids than the ones listed.
This is the more complicated and longer part of the proof.
Our proof strategy is the following.
For each class $K$ listed in the theorem, we identify the lower covers of $K$ in $\closys{(\clIc,\clMcUk{k})}$ among the listed classes.
Then we show that $K$ is generated by any subset of $K$ that is not contained in any of the lower covers of $K$ listed.
It then follows that every subset of $\clAll$ generates one of the $(\clIc,\clMcUk{k})$\hyp{}clonoids listed in Theorem~\ref{thm:IcMcUk-clonoids}; therefore the list is complete, and there are no further $(\clIc,\clMcUk{k})$\hyp{}clonoids.

\subsection{The classes listed in Theorem~\ref{thm:IcMcUk-clonoids} are $(\clIc,\mathsf{MU}^k_{01})$\hyp{}clonoids}

The first part of the proof of Theorem~\ref{thm:IcMcUk-clonoids} is to show that the classes listed is the statement of the theorem are $(\clIc,\clMcUk{k})$\hyp{}clonoids.
Because intersections of $(\clIc,\clMcUk{k})$\hyp{}clonoids are $(\clIc,\clMcUk{k})$\hyp{}clonoids, it is sufficient to verify this for the meet\hyp{}irreducible classes; the remaining ones are intersections of the meet\hyp{}irreducible ones.
We know from Theorem~\ref{prop:UTheta-stability}\ref{prop:UTheta-stability:stable} that the classes of the form $\clKlik{k}{\Theta}$ are $(\clIc,\clMcUk{k})$\hyp{}clonoids.
As for the classes that are not of this form, we make use of the monotonicity of function class composition (see Lemma~\ref{lem:stable-monotonicity}) and earlier results on 
$(\clIc,C)$\hyp{}clonoids for clones $C$ with $\clMcUk{k} \subseteq C$.

\begin{proposition}
\label{prop:stable-1}
For any $k \geq 2$, the following classes are $(\clIc,\clMcUk{k})$\hyp{}clonoids:
\begin{enumerate}[label={\upshape(\roman*)}]
\item\label{prop:stable-1:clones} $\clAll$, $\clOX$, $\clXI$, $\clXO$, $\clIX$, $\clM$, $\clMneg$, $\clUk{\ell}$, $\clWkneg{\ell}$, for $\ell \in \{2, \dots, k\}$,
\item\label{prop:stable-1:refl} $\clRefl$,
$\clEiio$, $\clEioi$, $\clEiii$, $\clOXC$, $\clIXC$, $\clXOC$, $\clXIC$.
\end{enumerate}
\end{proposition}

\begin{proof}
\ref{prop:stable-1:clones}
The classes $\clAll$, $\clOX$, $\clXI$, $\clM$, and $\clUk{\ell}$ ($\ell \in \{2, \dots, k\}$) are $(\clIc,\clMcUk{k})$\hyp{}clonoids by Lemma~\ref{lem:stable-clones}, because they are clones containing both $\clIc$ and $\clMcUk{k}$.
Because $\clIc = \clIc^\mathrm{d}$, it follows from Lemma~\ref{lem:stability-nd} that also $(\clOX)^\mathrm{n} = \clXO$, $(\clXI)^\mathrm{n} = \clIX$, $\clM^\mathrm{n} = \clMneg$, and $(\clUk{\ell})^\mathrm{n} = \clWkneg{\ell}$ are $(\clIc,\clMcUk{k})$\hyp{}clonoids.

\ref{prop:stable-1:refl}
The class $\clRefl$ is an $(\clS,\clAll)$\hyp{}clonoid by \cite[Lemma~7.9]{CouLeh-Lcstability}.
The classes $\clEiio$ and $\clEioi$ are $(\clOI,\clM)$\hyp{}clonoids, and the class $\clEiii$ is an $(\clOI,\clUk{k})$\hyp{}clonoid by \cite[Proposition~5.7]{Lehtonen-SM}.
The classes $\clOXC$ and $\clIXC$ are $(\clOX,\clM)$\hyp{}clonoids, and the classes $\clXOC$ and $\clXIC$ are $(\clXI,\clM)$\hyp{}clonoids by \cite[Proposition~5.9]{Lehtonen-SM}.
Therefore, these classes are $(\clIc,\clMcUk{k})$\hyp{}clonoids by Lemma~\ref{lem:stable-monotonicity}.
\end{proof}

\subsection{Covering relations for clonoids of the form $\clKlik{k}{\Theta}$ and their intersections with other clonoids}

\begin{lemma}
\label{lem:minmin-1}
Let $\theta, \theta' \in \clAll$, and assume that $\theta \minmin \theta'$.
If $\theta(\vect{1}) = 1$, then $\theta'(\vect{1}) = 1$.
If $\theta(\vect{0}) = 1$, then $\theta'(\vect{0}) = 1$.
\end{lemma}

\begin{proof}
Since $\theta \minmin \theta'$, there exists a minor formation map $\sigma$ such that $\theta \minorant \theta'_\sigma$.
Therefore $1 = \theta(\vect{1}) \leq \theta'_\sigma(\vect{1}) = \theta'(\vect{1} \sigma) = \theta'(\vect{1})$, so $\theta'(\vect{1}) = 1$.
The proof of the second claim is similar.
\end{proof}

\begin{lemma}
\label{lem:Klik-diff-1}
Let $\vect{c} \in \{\vect{1}, \vect{0}\}$.
Let $\Theta, \Psi \in \Ideals(\posetAllleq{k})$.
Assume that $\Psi \subsetneq \Theta$ and for every $\theta \in \Theta \setminus \Psi$ we have $\theta(\vect{c}) = 1$.
Let $g \in \clKlik{k}{\Theta} \setminus \clKlik{k}{\Psi}$.
Then $g(\vect{c}) = 1$.
\end{lemma}

\begin{proof}
We prove the claim for $\vect{c} = \vect{1}$.
The case when $\vect{c} = \vect{0}$ is proved similarly.

Suppose, to the contrary, that $g(\vect{1}) = 0$.
Let $T \subseteq g^{-1}(1)$ with $\card{T} \leq k$.
Since $g \in \clKlik{k}{\Theta}$, there exists a $\varphi \in \Theta$ and $\sigma \colon \nset{\arity{\varphi}} \to \nset{n}$ such that $\varphi(\vect{a} \sigma) = 1$ for all $\vect{a} \in T$.

It follows that $\chi_T \minmin \varphi$, because we can obtain $\chi_T$ from $\varphi$ by introducting fictitious arguments (which is a special case of taking minors) and then taking a minorant.
Since $\vect{1} \notin T = g^{-1}(1)$, we have $\chi_T(\vect{1}) = 0$.
Therefore $\chi_T \notin \Theta \setminus \Psi$, so $\chi_T \in \Psi$.
Moreover, $\chi_T(\vect{a}) = 1$ for all $\vect{a} \in T$.
We conclude that $g \in \clKlik{k}{\Psi}$, a contradiction.
\end{proof}

\begin{lemma}
\label{lem:Klik-all}
Let $\Theta \subseteq \clAll$ and $k \in \IN \cup \{\infty\}$ with $k \geq 2$.
The following conditions are equivalent.
\begin{enumerate}[label=\textup{(\roman*)}]
\item\label{lem:Klik-all:1}
$\clKlik{k}{\Theta} \cap \clII \neq \emptyset$.
\item\label{lem:Klik-all:2}
$1 \in {\downarrow} \Theta$.
\item\label{lem:Klik-all:3}
$\clKlik{k}{\Theta} = \clAll$.
\end{enumerate}
\end{lemma}

\begin{proof}
\ref{lem:Klik-all:1} $\implies$ \ref{lem:Klik-all:2}
Since $\clKlik{k}{\Theta} \cap \clII \neq \emptyset$, there exists an $f \in \clKlik{k}{\Theta} \cap \clII$.
Since $f \in \clII$, $\{\vect{0}, \vect{1}\} \subseteq f^{-1}(1)$.
Since $f \in \clKlik{k}{\Theta}$, there exists a $\theta \in \Theta$ and $\sigma \colon \nset{\arity \theta} \to \nset{n}$ such that
$\theta(\vect{0} \sigma) = \theta(\vect{1} \sigma) = 1$, i.e., $\theta(\vect{0}) = \theta(\vect{1}) = 1$.
Then the unary constant $1$ function is a minor of $\theta$, so $1 \in {\downarrow} \Theta$.

\ref{lem:Klik-all:2} $\implies$ \ref{lem:Klik-all:3}
Let $f \in \clAll$, and let $T \subseteq f^{-1}(1)$ with $\card{T} \leq k$.
Because every Boolean function is a minorant of a constant $1$ function and $1 \in {\downarrow} \Theta$, it follows that $f \minmin 1 \minmin \theta$ for some $\theta \in \Theta$, so $f \in {\downarrow} \Theta$.
Therefore, $f \in \clKlik{k}{\Theta}$.
This shows that $\clAll \subseteq \clKlik{k}{\Theta}$.
The converse inclusion is trivial.

\ref{lem:Klik-all:3} $\implies$ \ref{lem:Klik-all:1}
Trivial.
\end{proof}

\begin{lemma}
\label{lem:Klik-diff-2}
Let $k \geq 2$, $\Theta \in \Ideals(\posetAllleq{k})$, and assume that $\Theta \neq \clAllleq{k}$.
\begin{enumerate}[label=\textup{(\alph*)}]
\item\label{lem:Klik-diff-2:a}
If $\Theta \cap \clOI \neq \emptyset$, then there is a lower cover $\Psi$ of $\Theta$ in $\Ideals(\posetAllleq{k})$ such that $\clKlik{k}{\Theta} \setminus \clKlik{k}{\Psi} \subseteq \clOI$.
\item\label{lem:Klik-diff-2:b}
If $\Theta \cap \clIO \neq \emptyset$, then there is a lower cover $\Psi$ of $\Theta$ in $\Ideals(\posetAllleq{k})$ such that $\clKlik{k}{\Theta} \setminus \clKlik{k}{\Psi} \subseteq \clIO$.
\end{enumerate}
\end{lemma}

\begin{proof}
\ref{lem:Klik-diff-2:a}
Observe first that $1 \notin {\downarrow} \Theta$ because $\Theta \neq \clAllleq{k} = {\downarrow^{[\leq k]}} \{1\}$.
Therefore, $\clKlik{k}{\Theta} \cap \clII = \clEmpty$ by Lemma~\ref{lem:Klik-all}.
Assume now that $\Theta \cap \clOI \neq \clEmpty$, so there is an $\alpha \in \Theta \cap \clOI$, and let $\alpha'$ be a $\minmin$\hyp{}maximal element of $\Theta$ such that $\alpha \minmin \alpha'$.
Let $\Psi$ be the lower cover of $\Theta$ in $\Ideals(\posetAllleq{k})$ that does not contain $\alpha'$, i.e., $\Psi = \Theta \setminus (\alpha' / \mathord{\eqminmin})$. 
It follows from Lemma~\ref{lem:minmin-1} that $\alpha' / \mathord{\eqminmin} \subseteq \clXI$.
By Lemma~\ref{lem:Klik-diff-1}, we have $\clKlik{k}{\Theta} \setminus \clKlik{k}{\Psi} \subseteq \clXI$.
Because $\clKlik{k}{\Theta} \cap \clII = \clEmpty$, we conclude that $\clKlik{k}{\Theta} \setminus \clKlik{k}{\Psi} \subseteq \clOI$, as claimed.

\ref{lem:Klik-diff-2:b}
The proof is similar to \ref{lem:Klik-diff-2:a}.
\end{proof}

\begin{lemma}
\label{lem:covers}
Let $k \in \IN \cup \{\infty\}$ and $\Theta, \Psi \in \Ideals(\posetAllleq{k})$.

\begin{enumerate}[label=\textup{(\roman*)}]
\item\label{lem:covers:none}
Assume that $\Psi$ is a lower cover of $\Theta$ in $\Ideals(\posetAllleq{k})$.
Let $\theta \in \Theta \setminus \Psi$ and $g \in \clKlik{k}{\Theta} \setminus \clKlik{k}{\Psi}$.
Then $\theta \minmin g$.
Moreover, for $\vect{c} \in \{\vect{0}, \vect{1}\}$, if $\theta(\vect{c}) = 1$ then $g(\vect{c}) = 1$.

\item\label{lem:covers:C}
Let $C \in \{ \clXI, \clIX, \clM, \clMneg, \clRefl \}$.
Assume that $\Psi$ is a lower cover of $\Theta$ in $\Ideals(\posetAllleq{k}_C)$.
Let $\theta \in \Theta \setminus \Phi$ and $g \in (\clKlik{k}{\Theta} \cap C) \setminus (\clKlik{k}{\Psi} \cap C)$.
Then $\theta \minmin g$.
\end{enumerate}
\end{lemma}

\begin{proof}
For the proof of statements \ref{lem:covers:none} and \ref{lem:covers:C}, we denote by $\minmin'$ the relation $\minmin$ and $\minmin_C$, respectively, and by $\eqminmin'$ the relation $\eqminmin$ and $\eqminmin_C$, respectively.

Since $g \notin \clKlik{k}{\Psi}$, there exists $T \subseteq g^{-1}(1)$ with $\card{T} \leq k$ such that for all $\psi \in \Psi$ and for all $\sigma \colon \nset{\arity \psi} \to \nset{n}$ there is an $\vect{a} \in T$ with $\psi(\vect{a} \sigma) = 0$.
On the other hand, since $g \in \clKlik{k}{\Theta}$, there is a $\theta' \in \Theta$ and $\tau \colon \nset{\arity{\theta'}} \to \nset{n}$ such that $\theta'(\vect{t} \tau) = 1$ for all $\vect{t} \in T$; clearly $\theta' \in \Theta \setminus \Psi$.
Let $T \tau := \{ \, \vect{t} \tau \mid \vect{t} \in T \,\}$.
We thus have $\chi_T \minorant g$, $\chi_{T \tau} \minorant g_\tau \leq g$, and $\chi_{T \tau} \minmin \theta'$, and so $\chi_{T \tau} \minmin' \theta'$.

We claim that $\theta' \minmin' \chi_{T \tau}$.
For, suppose, to the contrary, that $\theta' \not\minmin' \chi_{T \tau}$.
Since $\chi_{T \tau} \minmin' \theta'$, we have that $\chi_{T \tau}$ is strictly below $\theta'$ in the $\minmin'$ order.
Because $\theta' \in \Theta \setminus \Psi$ and $\Psi$ is a lower cover of $\Theta$, it follows that $\theta' / \mathord{\eqminmin'} = \Theta \setminus \Psi$.
Therefore, $\chi_{T \tau} \not\eqminmin' \theta'$, so $\chi_{T \tau} \in \Theta \setminus (\theta' / \mathord{\eqminmin'}) = \Psi$.
By the definition of $\chi_{T \tau}$, we have $\chi_{T \tau}(\vect{t} \tau) = 1$ for all $\vect{t} \in T$, but
since $\chi_{T \tau} \in \Psi$, there must be some $\vect{u} \in T$ such that $\chi_{T \tau}(\vect{u} \tau) = 0$, a contradiction.

Since both $\theta$ and $\theta'$ belong to $\Theta \setminus \Psi$, which is a $\eqminmin'$\hyp{}class, it follows from the above that $\theta \eqminmin' \theta' \eqminmin' \chi_{T \tau} \minmin' g$.

For statement \ref{lem:covers:C},
note that $\theta \minmin_C g$ means $\theta \minmin g^C \RelCl{C} g$ because $\mathord{\minmin_C} = \mathord{(\mathord{\minorant} \circ \mathord{\minor} \circ \mathord{\RelCl{C}})}$.
Since $g = g^C$, we obtain $\theta \minmin g$, as claimed, also in this case.

Finally, the last claim in statement \ref{lem:covers:none} follows from Lemma~\ref{lem:Klik-diff-1}.
\end{proof}

We are now ready to describe the lower covers of the $(\clIc,\clMcUk{k})$\hyp{}clonoids of the form $\clKlik{k}{\Theta}$ and $\clKlik{k}{\Theta} \cap C$, where
$\clKlik{k}{\Theta} \neq \clAll$ and
\[
C \in \{\clOICO, \clIOCO, \clOI, \clIO, \clMo, \clMc, \clMineg, \clMcneg, \clReflOO\}.
\]

\begin{theorem}
\label{thm:UkTC-covers}
Let $k \geq 2$, and let $\Theta \in \Ideals(\posetAllleq{k})$
with $\Theta \notin \{\clAllleq{k}, \clVako, \clEmpty\}$.
\begin{enumerate}[label=\textup{(\roman*)}]
\item The lower covers of $\clKlik{k}{\Theta}$ are the classes $\clKlik{k}{\Psi}$ for each lower cover $\Psi$ of $\Theta$ in $\Ideals(\posetAllleq{k})$ and, additionally, the following:
\begin{itemize}
\item if $\Theta \in \Ideals(\posetAllleq{k}_\clXI)$, the class $\clKlik{k}{\Theta} \cap (\clOICO)$;
\item if $\Theta \in \Ideals(\posetAllleq{k}_\clIX)$, the class $\clKlik{k}{\Theta} \cap (\clIOCO)$;
\item if $\Theta \in \Ideals(\posetAllleq{k}_\clRefl)$, the class $\clKlik{k}{\Theta} \cap \clReflOO$.
\end{itemize}

\item Assume $\Theta \in \Ideals(\posetAllleq{k}_\clXI)$.
\begin{enumerate}[label=\textup{(\alph*)}]
\item
The lower covers of $\clKlik{k}{\Theta} \cap (\clOICO)$ are the classes $\clKlik{k}{\Psi} \cap (\clOICO)$ for each lower cover $\Psi$ of $\Theta$ in $\Ideals(\posetAllleq{k}_\clXI)$, $\clKlik{k}{\Theta} \cap \clOI$, and, additionally,
if $\Theta \in \Ideals(\posetAllleq{k}_\clM)$, the class $\clKlik{k}{\Theta} \cap \clMo$.
\item
The lower covers of $\clKlik{k}{\Theta} \cap \clOI$ are the classes $\clKlik{k}{\Psi} \cap \clOI$ for each lower cover $\Psi$ of $\Theta$ in $\Ideals(\posetAllleq{k}_\clXI)$ and, additionally,
if $\Theta \in \Ideals(\posetAllleq{k}_\clM)$, the class $\clKlik{k}{\Theta} \cap \clMc$.
\end{enumerate}

\item Assume $\Theta \in \Ideals(\posetAllleq{k}_\clIX)$.
\begin{enumerate}[label=\textup{(\alph*)}]
\item
The lower covers of $\clKlik{k}{\Theta} \cap (\clIOCO)$ are the classes $\clKlik{k}{\Psi} \cap (\clIOCO)$ for each lower cover $\Psi$ of $\Theta$ in $\Ideals(\posetAllleq{k}_\clIX)$, $\clKlik{k}{\Theta} \cap \clIO$, and, additionally,
if $\Theta \in \Ideals(\posetAllleq{k}_\clMneg)$, the class $\clKlik{k}{\Theta} \cap \clMoneg$.
\item
The lower covers of $\clKlik{k}{\Theta} \cap \clIO$ are the classes $\clKlik{k}{\Psi} \cap \clIO$ for each lower cover $\Psi$ of $\Theta$ in $\Ideals(\posetAllleq{k}_\clIX)$ and, additionally,
if $\Theta \in \Ideals(\posetAllleq{k}_\clMneg)$, the class $\clKlik{k}{\Theta} \cap \clMcneg$.
\end{enumerate}

\item Assume $\Theta \in \Ideals(\posetAllleq{k}_\clM)$.
\begin{enumerate}[label=\textup{(\alph*)}]
\item
The lower covers of $\clKlik{k}{\Theta} \cap \clMo$ are the classes $\clKlik{k}{\Psi} \cap \clMo$ for each lower cover $\Psi$ of $\Theta$ in $\Ideals(\posetAllleq{k}_\clM)$ and $\clKlik{k}{\Theta} \cap \clMc$.
\item
The lower covers of $\clKlik{k}{\Theta} \cap \clMc$ are the classes $\clKlik{k}{\Psi} \cap \clMc$ for each lower cover $\Psi$ of $\Theta$ in $\Ideals(\posetAllleq{k}_\clM)$.
\end{enumerate}

\item Assume $\Theta \in \Ideals(\posetAllleq{k}_\clMneg)$.
\begin{enumerate}[label=\textup{(\alph*)}]
\item
The lower covers of $\clKlik{k}{\Theta} \cap \clMineg$ are the classes $\clKlik{k}{\Psi} \cap \clMineg$ for each lower cover $\Psi$ of $\Theta$ in $\Ideals(\posetAllleq{k}_\clMneg)$ and $\clKlik{k}{\Theta} \cap \clMcneg$.
\item
The lower covers of $\clKlik{k}{\Theta} \cap \clMcneg$ are the classes $\clKlik{k}{\Psi} \cap \clMcneg$ for each lower cover $\Psi$ of $\Theta$ in $\Ideals(\posetAllleq{k}_\clMneg)$.
\end{enumerate}

\item Assume $\Theta \in \Ideals(\posetAllleq{k}_\clRefl)$.
\newline
The lower covers of $\clKlik{k}{\Theta} \cap \clReflOO$ are the classes $\clKlik{k}{\Psi} \cap \clReflOO$ for each lower cover $\Psi$ of $\Theta$ in $\Ideals(\posetAllleq{k}_\clRefl)$.

\end{enumerate}
\end{theorem}

Note that if $\Theta = \{0\}$, then $\Theta$ belongs to $\Ideals(\posetAllleq{k}_\clM)$, $\Ideals(\posetAllleq{k}_\clMneg)$. and $\Ideals(\posetAllleq{k}_\clRefl)$, and $\clVako = \clKlik{k}{\Theta} \cap \clMo = \clKlik{k}{\Theta} \cap \clMineg = \clKlik{k}{\Theta} \cap \clReflOO$ and $\clEmpty = \clKlik{k}{\Theta} \cap \clMc = \clKlik{k}{\Theta} \cap \clMcneg$.

\begin{proof}
For the statement about the lower covers of each $(\clIc,\clMcUk{k})$\hyp{}clonoid $K$, we need to prove two things.
First, we need to show that the claimed lower covers are proper subsets of $K$.
This is straightforward verification and is left as an exercise to the reader.
Second, we need to show that $K$ is generated by any subset $G$ of $K$ that is not contained in any of the claimed lower covers of $K$.
This is the content of Propositions~\ref{prop:UkT-gen:plain}--\ref{prop:UkT-gen:R} that follow.
\end{proof}

\begin{proposition}
\label{prop:UkT-gen:plain}
Let $k \geq 2$, let $\Theta, \Psi_1, \dots, \Psi_p \in \Ideals(\posetAllleq{k})$, and assume that $\Theta \neq \emptyset$ and $\Psi_1, \dots, \Psi_p$ are the lower covers of $\Theta$ in $\Ideals(\posetAllleq{k})$.
For $i \in \nset{p}$, let $g_i \in \clKlik{k}{\Theta} \setminus \clKlik{k}{\Psi_i}$.

\begin{enumerate}[label=\textup{(\alph*)}]
\item\label{prop:UkT-gen:plain:plain}
If $\Theta \subseteq \clEiii$ and
$\Theta \notin \Ideals(\posetAllleq{k}_C)$ for all $C \in \{\clXI, \clIX, \clRefl\}$,
then
we have
$\gen{g_1, \dots, g_p} = \clKlik{k}{\Theta}$.

\item\label{prop:UkT-gen:plain:1}
If $\Theta \subseteq \clOX$ and
$\Theta \in \Ideals(\posetAllleq{k}_\clXI)$,
then
for any $h \in \clKlik{k}{\Theta} \setminus (\clKlik{k}{\Theta} \cap (\clOICO))$, we have
$\gen{g_1, \dots, g_p, h} = \clKlik{k}{\Theta}$.

\item\label{prop:UkT-gen:plain:0}
If $\Theta \subseteq \clXO$ and
$\Theta \in \Ideals(\posetAllleq{k}_\clIX)$,
then
for any $h \in \clKlik{k}{\Theta} \setminus (\clKlik{k}{\Theta} \cap (\clIOCO))$, we have
$\gen{g_1, \dots, g_p, h} = \clKlik{k}{\Theta}$.

\item\label{prop:UkT-gen:plain:R}
If $\Theta \subseteq \clOO$ and
$\Theta \in \Ideals(\posetAllleq{k}_\clRefl)$,
then
for any $h \in \clKlik{k}{\Theta} \setminus (\clKlik{k}{\Theta} \cap \clReflOO)$, we have
$\gen{g_1, \dots, g_p, h} = \clKlik{k}{\Theta}$.
\end{enumerate}
\end{proposition}

\begin{proof}
We are going to show that $\clKlik{k}{\Theta} \setminus \clVako$ is $(G,k)$\hyp{}semibisectable with $G$ being the set of functions described in each statement (either $\{g_1, \dots, g_p\}$ or $\{g_1, \dots, g_p, h\}$) and that $0 \in \gen{G}$.
It then follows from Lemma~\ref{lem:helpful} and Proposition~\ref{prop:constants}\ref{prop:constants:0} that
$\clKlik{k}{\Theta} 
= (\clKlik{k}{\Theta} \setminus \clVako) \cup \clVako
\subseteq \gen{G} \cup \gen{0}
\subseteq \gen{G}
\subseteq \clKlik{k}{\Theta}$.

Let $f \in \clKlik{k}{\Theta} \setminus \clVako$.
Let $T \subseteq f^{-1}(1)$ with $\card{T} \leq k$.
Then there is a $\phi \in \Theta$ and $\tau \colon \nset{\arity{\phi}} \to \nset{n}$ such that $\phi(\vect{a} \tau) = 1$ for all $\vect{a} \in T$.
There is a maximal element $\theta$ of $\Theta$ such that $\phi \minmin \theta$, and there is a lower cover $\Psi_j$ of $\Theta$ such that $\theta \in \Theta \setminus \Psi_j$.
By Lemma~\ref{lem:covers}\ref{lem:covers:none}, $\theta \minmin g_j$, and so $\phi \minmin g_j$.
It follows that there is a $\pi \colon \nset{\arity{g_j}} \to \nset{\arity{\psi}}$ such that $\phi \minorant (g_j)_\pi$, so $(g_j)_\pi(\vect{a} \tau) = g_j(\vect{a} \tau \pi) = 1$ for all $\vect{a} \in T$.
We conclude that condition \ref{helpful:k-true} of Definition~\ref{def:helpful} is satisfied by any $G$ with $\{g_1, \dots, g_p\} \subseteq G$.

It remains to verify that also condition \ref{helpful:both} is satisfied and that $0 \in \gen{G}$ in each of the cases described in statements \ref{prop:UkT-gen:plain:plain}, \ref{prop:UkT-gen:plain:1}, \ref{prop:UkT-gen:plain:0}, \ref{prop:UkT-gen:plain:R}.
Let $\vect{a} \in f^{-1}(1)$ and $\vect{b} \in f^{-1}(0)$.
We have $\vect{a} \neq \vect{b}$, so there is an $i \in \nset{n}$ such that $a_i \neq b_i$.

\ref{prop:UkT-gen:plain:plain}
We consider several cases according to whether $\Theta$ contains a function in $\clOI$ or $\clIO$.

Case 1:
Assume that $\Theta \cap \clOI \neq \emptyset$ and $\Theta \cap \clIO \neq \emptyset$.
By Lemma~\ref{lem:Klik-diff-2}, there are $r, s \in \nset{p}$ such that $g_r \in \clOI$ and $g_s \in \clIO$.
We have $\id \leq g_r$ and $\neg \leq g_s$, and $\tau_i(\vect{a}) = 1$ and $\tau_i(\vect{b}) = 0$ for some $\tau \in \{\id, \neg\}$, so condition \ref{helpful:both} is satisfied.
Moreover, for $k \geq 3$, $\threshold{k+1}{k}(\id, \id, \underbrace{\neg, \dots, \neg}_{k-1}) = 0$ and $(x \wedge (y \vee z))(\id, \neg, \neg) = 0$, so $0 \in \gen{g_r, g_s} \subseteq \gen{G}$.

Case 2:
Assume that $\Theta \cap \clOI \neq \emptyset$ but $\Theta \cap \clIO = \emptyset$.
Then $\Theta \subseteq \clOX$ and hence $\clKlik{k}{\Theta} \subseteq \clOX$, and we have that $\vect{a} \neq \vect{0}$, so there is a $j \in \nset{n}$ such that $a_j = 1$.
By Lemma~\ref{lem:Klik-diff-2}, there is an $s \in \nset{p}$ such that $g_s \in \clOI$, and hence $\id \leq g_s$.

Let $\Gamma$ be the largest $(k,\clXI)$\hyp{}closed subset of $\Theta$.
Since $\Theta$ is not $(k,\clXI)$\hyp{}closed,
$\Gamma$ is a proper subset of $\Theta$, so there exists a lower cover $\Psi_t$ of $\Theta$ such that $\Gamma \subseteq \Psi_t$.
By Proposition~\ref{prop:1MR-part},
\[
\clKlik{k}{\Theta} \cap \clOI = \clKlik{k}{\Gamma} \cap \clOI = \clKlik{k}{\Psi_t} \cap \clOI.
\]
Therefore,
$(\clKlik{k}{\Theta} \setminus \clKlik{k}{\Psi_t}) \cap \clOI = \emptyset$,
so $g_t \in \clOX \setminus \clOI = \clOO$.
Note that $g_t$ is not a constant $0$ function, because $\clVako \subseteq \clKlik{k}{\Psi_t}$.
Therefore $g_t$ has a binary minor $\gamma$ satisfying $\gamma(0,0) = \gamma(1,1) = 0$ and $\gamma(0,1) = 1$.
Moreover, $0 \minor \gamma$, so $0 \in \gen{g_t} \subseteq \gen{G}$.

As we have either
$\left( \begin{smallmatrix} a_i \\ b_i \end{smallmatrix} \right) = \left( \begin{smallmatrix} 1 \\ 0 \end{smallmatrix} \right)$
or $\left( \begin{smallmatrix} a_i & a_j \\ b_i & b_j \end{smallmatrix} \right) = \left( \begin{smallmatrix} 0 & 1 \\ 1 & 1 \end{smallmatrix} \right)$,
we have either $\id_i(\vect{a}) = 1$ and $\id_i(\vect{b}) = 0$ or $\gamma_{ij}(\vect{a}) = 1$ and $\gamma_{ij}(\vect{b}) = 0$, so condition \ref{helpful:both} is satisfied.

Case 3:
Assume that $\Theta \cap \clIO \neq \emptyset$ but $\Theta \cap \clOI = \emptyset$.
The proof is similar to the previous case.

Case 4:
Assume that $\Theta \cap \clOI = \emptyset$ and $\Theta \cap \clIO = \emptyset$.
Then $\Theta \subseteq \clOO$ and hence $\clKlik{k}{\Theta} \subseteq \clOO$.
Therefore $\vect{a} \notin \{\vect{0}, \vect{1}\}$, so there exists a $j \in \nset{n}$ such that $a_i \neq a_j$.

Let now $\Gamma$ be the largest $(k,\clRefl)$\hyp{}closed subset of $\Theta$.
Since $\Theta$ is not $(k,\clRefl)$\hyp{}closed, there exists a lower cover $\Psi_t$ of $\Theta$ such that $\Gamma \subseteq \Psi_t$.
By Proposition~\ref{prop:1MR-part},
\[
\clKlik{k}{\Theta} \cap \clReflOO = \clKlik{k}{\Gamma} \cap \clReflOO = \clKlik{k}{\Psi_t} \cap \clReflOO.
\]
Therefore,
$(\clKlik{k}{\Theta} \setminus \clKlik{k}{\Psi_t}) \cap \clReflOO = \emptyset$,
so $g_t \in \clOO \setminus \clReflOO$.
Note that $g_t$ is not a constant $0$ function, because $\clVako \subseteq \clKlik{k}{\Psi_t}$.
Therefore $g_t$ has a binary minor $\gamma$ satisfying $\gamma(0,0) = \gamma(1,1) = \gamma(0,1) = 0$ and $\gamma(1,0) = 1$, i.e., $\gamma = \mathord{\nrightarrow}$.
Moreover, $0 \minor \gamma$, so $0 \in \gen{g_t} \subseteq \gen{G}$.

Recall that $a_i \neq b_i$ and $a_i \neq a_j$.
Therefore
$\left( \begin{smallmatrix} a_i & a_j \\ b_i & b_j \end{smallmatrix} \right)
\in
\left\{
\left( \begin{smallmatrix} 0 & 1 \\ 1 & 0 \end{smallmatrix} \right),
\left( \begin{smallmatrix} 1 & 0 \\ 0 & 1 \end{smallmatrix} \right),
\left( \begin{smallmatrix} 0 & 1 \\ 1 & 1 \end{smallmatrix} \right),
\left( \begin{smallmatrix} 1 & 0 \\ 0 & 0 \end{smallmatrix} \right)
\right\}$.
It follows that
$\mathord{\nrightarrow}_{ij}(\vect{a}) = 1$ and $\mathord{\nrightarrow}_{ij}(\vect{b}) = 0$
or
$\mathord{\nrightarrow}_{ji}(\vect{a}) = 1$ and $\mathord{\nrightarrow}_{ji}(\vect{b}) = 0$, so condition \ref{helpful:both} is satisfied.

\ref{prop:UkT-gen:plain:1}
Since $\Theta \subseteq \clOX$, we have also $\clKlik{k}{\Theta} \subseteq \clOX$.
Therefore $\vect{a} \neq \vect{0}$, so there is a $j \in \nset{n}$ such that $a_j = 1$.

Let $h \in \clKlik{k}{\Theta} \setminus (\clKlik{k}{\Theta} \cap (\clOI \cup \clVako))$.
Then $h \in \clOO \setminus \clVako$, so $h$ has a binary minor $\gamma$ satisfying $\gamma(0,0) = \gamma(1,1) = 0$ and $\gamma(0,1) = 1$.
Moreover, $0 \minor \gamma$, so $0 \in \gen{h} \subseteq \gen{G}$.

As we have either
$\left( \begin{smallmatrix} a_i \\ b_i \end{smallmatrix} \right) = \left( \begin{smallmatrix} 1 \\ 0 \end{smallmatrix} \right)$
or $\left( \begin{smallmatrix} a_i & a_j \\ b_i & b_j \end{smallmatrix} \right) = \left( \begin{smallmatrix} 0 & 1 \\ 1 & 1 \end{smallmatrix} \right)$,
it holds that either $\id_i(\vect{a}) = 1$ and $\id_i(\vect{b}) = 0$ or $\gamma_{ij}(\vect{a}) = 1$ and $\gamma_{ij}(\vect{b}) = 0$, so condition \ref{helpful:both} is satisfied.

\ref{prop:UkT-gen:plain:0}
The proof is similar to that of statement \ref{prop:UkT-gen:plain:1}.

\ref{prop:UkT-gen:plain:R}
Since $\Theta \subseteq \clOO$, we also have $\clKlik{k}{\Theta} \subseteq \clOO$.
Therefore $\vect{a} \notin \{\vect{0}, \vect{1}\}$, so there exists a $j \in \nset{n}$ such that $a_i \neq a_j$.

Let $h \in \clKlik{k}{\Theta} \setminus (\clKlik{k}{\Theta} \cap \clReflOO)$.
Then $h \in \clOO \setminus \clReflOO$, so $h$ has a binary minor $\gamma$ satisfying $\gamma(0,0) = \gamma(1,1) = \gamma(0,1) = 0$ and $\gamma(1,0) = 1$, i.e., $\gamma = \mathord{\nrightarrow}$.
Moreover, $0 \minor \gamma$, so $0 \in \gen{g_t} \subseteq \gen{G}$.

Recall that $a_i \neq b_i$ and $a_i \neq a_j$.
Therefore
$\left( \begin{smallmatrix} a_i & a_j \\ b_i & b_j \end{smallmatrix} \right)
\in
\left\{
\left( \begin{smallmatrix} 0 & 1 \\ 1 & 0 \end{smallmatrix} \right),
\left( \begin{smallmatrix} 1 & 0 \\ 0 & 1 \end{smallmatrix} \right),
\left( \begin{smallmatrix} 0 & 1 \\ 1 & 1 \end{smallmatrix} \right),
\left( \begin{smallmatrix} 1 & 0 \\ 0 & 0 \end{smallmatrix} \right)
\right\}$.
It follows that $\mathord{\nrightarrow}_{ij}(\vect{a}) = 1$ and $\mathord{\nrightarrow}_{ij}(\vect{b}) = 0$ or $\mathord{\nrightarrow}_{ji}(\vect{a}) = 1$ and $\mathord{\nrightarrow}_{ji}(\vect{b}) = 0$, so condition \ref{helpful:both} is satisfied.
\end{proof}

\begin{proposition}
\label{prop:UkT-gen:1}
Assume that
$\Theta, \Psi_1, \dots, \Psi_p \in \Ideals(\posetAllleq{k}_\clXI)$ and
$\Psi_1, \dots, \Psi_p$ are the lower covers of $\Theta$ in $\Ideals(\posetAllleq{k}_\clXI)$.
For $i \in \nset{p}$, let $g_i \in (\clKlik{k}{\Theta} \cap \clOI) \setminus (\clKlik{k}{\Psi_i} \cap \clOI) = (\clKlik{k}{\Theta} \cap (\clOI \cup \clVako)) \setminus (\clKlik{k}{\Psi_i} \cap (\clOI \cup \clVako))$.
\begin{enumerate}[label=\textup{(\alph*)}]
\item\label{prop:UkT-gen:1:plain}
If $\Theta \notin \Ideals(\posetAllleq{k}_\clM)$,
then $\gen{g_1, \dots, g_p} = \clKlik{k}{\Theta} \cap \clOI$
and for any $c \in (\clKlik{k}{\Theta} \cap (\clOI \cup \clVako)) \setminus (\clKlik{k}{\Theta} \cap \clOI)$, $\gen{g_1, \dots, g_p, c} = \clKlik{k}{\Theta} \cap (\clOI \cup \clVako)$.

\item\label{prop:UkT-gen:1:M}
If $\Theta \in \Ideals(\posetAllleq{k}_\clM)$,
then for any $h \in (\clKlik{k}{\Theta} \cap \clOI) \setminus (\clKlik{k}{\Theta} \cap \clM) = (\clKlik{k}{\Theta} \cap (\clOI \cup \clVako)) \setminus (\clKlik{k}{\Theta} \cap \clM)$, we have $\gen{g_1, \dots, g_p, h} = \clKlik{k}{\Theta} \cap \clOI$, and for any $c \in (\clKlik{k}{\Theta} \cap (\clOI \cup \clVako)) \setminus (\clKlik{k}{\Theta} \cap \clOI)$, $\gen{g_1, \dots, g_p, h, c} = \clKlik{k}{\Theta} \cap (\clOI \cup \clVako)$.
\end{enumerate}
\end{proposition}

\begin{proof}
We are going to show that $\clKlik{k}{\Theta} \cap \clOI$ is $(G,k)$\hyp{}semibisectable with $G$ being either $\{g_1, \dots, g_p\}$ or $\{g_1, \dots, g_p, h\}$, as required in each statement.
It then follows from Lemma~\ref{lem:helpful} that
$\clKlik{k}{\Theta} \cap \clOI \subseteq \gen{G} \subseteq \clKlik{k}{\Theta} \cap \clOI$.
Furthermore, because $c \in \clVako \subseteq \clKlik{k}{\Theta}$, it follows from Proposition~\ref{prop:constants}\ref{prop:constants:0} that
$\clKlik{k}{\Theta} \cap (\clOICO) = (\clKlik{k}{\Theta} \cap \clOI) \cup \clVako
\subseteq \gen{G} \cup \gen{0}
\subseteq \gen{G \cup \{c\}}
\subseteq \clKlik{k}{\Theta} \cap (\clOICO)$.

Let $f \in \clKlik{k}{\Theta} \cap \clOI$.
Let $T \subseteq f^{-1}(1)$ with $\card{T} \leq k$.
Then there is a $\phi \in \Theta$ and $\tau \colon \nset{\arity{\phi}} \to \nset{n}$ such that $\phi(\vect{a} \tau) = 1$ for all $\vect{a} \in T$.
There is a maximal element $\theta$ of $\Theta$ such that $\phi \minmin \theta$, and there is a lower cover $\Psi_j$ of $\Theta$ in $\Ideals(\clAllleq{k} / \mathord{\eqminmin_\clXI}, \mathord{\minmin_\clXI})$ such that $\theta \in \Theta \setminus \Psi_j$.
By Lemma~\ref{lem:covers}\ref{lem:covers:C}, $\theta \minmin g_j$, and so $\phi \minmin g_j$.
It follows that there is a $\pi \colon \nset{\arity{g_j}} \to \nset{\arity{\psi}}$ such that $\phi \minorant (g_j)_\pi$, so $(g_j)_\pi(\vect{a} \tau) = g_j(\vect{a} \tau \pi) = 1$ for all $\vect{a} \in T$.
We conclude that condition \ref{helpful:k-true} of Definition~\ref{def:helpful} is satisfied by any $G$ with $\{g_1, \dots, g_p\} \subseteq G$.

It remains to verify that also condition \ref{helpful:both} is satisfied in each case.
Let $\vect{a} \in f^{-1}(1)$ and $\vect{b} \in f^{-1}(0)$.
We have $\vect{a} \neq \vect{b}$, so there is an $i \in \nset{n}$ such that $a_i \neq b_i$.
Since $f \in \clOI$, we have $\vect{a} \neq \vect{0}$ and $\vect{b} \neq \vect{1}$, so there are $j, k \in \nset{n}$ such that $a_j = 1$ and $b_k = 0$.
Consequently, either
$\left( \begin{smallmatrix} a_i \\ b_i \end{smallmatrix} \right) = \left( \begin{smallmatrix} 1 \\ 0 \end{smallmatrix} \right)$
or
$\left( \begin{smallmatrix} a_i & a_j & a_k \\ b_i & b_j & b_k \end{smallmatrix} \right) = \left( \begin{smallmatrix} 0 & 1 & 0 \\ 1 & 1 & 0 \end{smallmatrix} \right)$.

Since $g_i \in \clOI$ for any $i \in \nset{p}$, we have $\id \leq g_i$.
Therefore we are done in the case when
$\left( \begin{smallmatrix} a_i \\ b_i \end{smallmatrix} \right) = \left( \begin{smallmatrix} 1 \\ 0 \end{smallmatrix} \right)$, because then we have $\id_i(\vect{a}) = 1$ and $\id_i(\vect{b}) = 0$.

In order to deal with the case when
$\left( \begin{smallmatrix} a_i & a_j & a_k \\ b_i & b_j & b_k \end{smallmatrix} \right) = \left( \begin{smallmatrix} 0 & 1 & 0 \\ 1 & 1 & 0 \end{smallmatrix} \right)$,
we need to provide a function $\gamma$ with $\gamma(0,0,1) = 1$ and $\gamma(0,1,1) = 0$; for such a function $\gamma$ we have $\gamma_{kij}(\vect{a}) = 1$ and $\gamma_{kij}(\vect{b}) = 0$.

\ref{prop:UkT-gen:1:plain}
Since $\Theta$ is not $(k,\clM)$\hyp{}closed by Lemma~\ref{lem:k-1MR-char}, there is a lower cover $\Psi_j$ such that the largest $(k,\clM)$\hyp{}closed subset $\Gamma$ of $\Theta$ is included in $\Psi_j$.
By Proposition~\ref{prop:1MR-part}, $\clKlik{k}{\Theta} \cap \clM = \clKlik{k}{\Gamma} \cap \clM = \clKlik{k}{\Psi_j} \cap \clM$, and from this it follows that $(\clKlik{k}{\Theta} \setminus \clKlik{k}{\Psi_j}) \cap \clM = \emptyset$; hence $g_j$ is not monotone.
Therefore $g_j$ has a ternary minor $\gamma$ with $\gamma(0,0,0) = 0$, $\gamma(0,0,1) = 1$, $\gamma(0,1,1) = 0$, and $\gamma(1,1,1) = 1$, as desired.

\ref{prop:UkT-gen:1:M}
Let $h \in (\clKlik{k}{\Theta} \cap (\clOI \cup \clVako)) \setminus (\clKlik{k}{\Theta} \cap \clM)$.
Then $h$ is not monotone, and $h$ has a ternary minor $\gamma$ with $\gamma(0,0,0) = 0$, $\gamma(0,0,1) = 1$, $\gamma(0,1,1) = 0$, and $\gamma(1,1,1) = 1$, as desired.
\end{proof}

\begin{proposition}
\label{prop:UkT-gen:0}
Assume that $\Theta, \Psi_1, \dots, \Psi_p \in \Ideals(\posetAllleq{k}_\clIX)$ and $\Psi_1, \dots, \Psi_p$ are the lower covers of $\Theta$ in $\Ideals(\posetAllleq{k}_\clIX)$.
For $i \in \nset{p}$, let $g_i \in (\clKlik{k}{\Theta} \cap \clIO) \setminus (\clKlik{k}{\Psi_i} \cap \clIO)$.
\begin{enumerate}[label=\textup{(\alph*)}]
\item\label{prop:UkT-gen:0:plain}
If $\Theta \notin \Ideals(\posetAllleq{k}_\clMneg)$,
then $\gen{g_1, \dots, g_p} = \clKlik{k}{\Theta} \cap \clIO$
and for any $c \in (\clKlik{k}{\Theta} \cap (\clIO \cup \clVako)) \setminus (\clKlik{k}{\Theta} \cap \clIO)$, $\gen{g_1, \dots, g_p, c} = \clKlik{k}{\Theta} \cap (\clIO \cup \clVako)$.

\item\label{prop:UkT-gen:0:Mneg}
If $\Theta \in \Ideals(\posetAllleq{k}_\clMneg)$,
then for any $h \in (\clKlik{k}{\Theta} \cap \clIO) \setminus (\clKlik{k}{\Theta} \cap \clMneg) = (\clKlik{k}{\Theta} \cap (\clIO \cup \clVako)) \setminus (\clKlik{k}{\Theta} \cap \clMneg)$, we have $\gen{g_1, \dots, g_p, h} = \clKlik{k}{\Theta} \cap \clIO$,
and for any $c \in (\clKlik{k}{\Theta} \cap (\clIO \cup \clVako)) \setminus (\clKlik{k}{\Theta} \cap \clIO)$, $\gen{g_1, \dots, g_p, h, c} = \clKlik{k}{\Theta} \cap (\clIO \cup \clVako)$.
\end{enumerate}
\end{proposition}

\begin{proof}
This follows from Proposition~\ref{prop:UkT-gen:1} by taking inner negations and applying Lemma~\ref{lem:stability-nd}.
\end{proof}

\begin{proposition}
\label{prop:UkT-gen:M}
Assume that
$\Theta \in \Ideals(\posetAllleq{k}_\clM)$
and $\Psi_1, \dots, \Psi_p$ are the lower covers of $\Theta$ in $\Ideals(\posetAllleq{k}_\clM)$.
For $i \in \nset{p}$, let $g_i \in (\clKlik{k}{\Theta} \cap \clMc) \setminus (\clKlik{k}{\Psi_i} \cap \clMc) = (\clKlik{k}{\Theta} \cap \clMo) \setminus (\clKlik{k}{\Psi_i} \cap \clMo)$.
Then $\gen{g_1, \dots, g_p} = \clKlik{k}{\Theta} \cap \clMc$
and for any $c \in (\clKlik{k}{\Theta} \cap \clMo) \setminus (\clKlik{k}{\Theta} \cap \clMc)$, $\gen{g_1, \dots, g_p, c} = \clKlik{k}{\Theta} \cap \clMo$.
\end{proposition}

\begin{proof}
We are going to show that $\clKlik{k}{\Theta} \cap \clMc$ is $(G,k)$\hyp{}semibisectable with $G = \{g_1, \dots, g_p\}$.
Hence $\clKlik{k}{\Theta} \cap \clMc \subseteq \gen{g_1, \dots, g_p} \subseteq \clKlik{k}{\Theta} \cap \clMc$.
Because $c \in \clVako$, it also follows that
$\clKlik{k}{\Theta} \cap \clMo = (\clKlik{k}{\Theta} \cap \clMc) \cup \clVako = \gen{g_1, \dots, g_p} \cup \gen{0} \subseteq \gen{g_1, \dots, g_p, c} \subseteq \clKlik{k}{\Theta} \cap \clMo$.

Let $f \in \clKlik{k}{\Theta} \cap \clMc$.
Let $T \subseteq f^{-1}(1)$ with $\card{T} \leq k$.
Then there is a $\phi \in \Theta$ and $\tau \colon \nset{\arity{\phi}} \to \nset{n}$ such that $\phi(\vect{a} \tau) = 1$ for all $\vect{a} \in T$.
There is a maximal element $\theta$ of $\Theta$ such that $\phi \minmin \theta$, and there is a lower cover $\Psi_j$ of $\Theta$ in $\Ideals(\clAllleq{k} / \mathord{\eqminmin_\clM}, \mathord{\minmin_\clM})$ such that $\theta \in \Theta \setminus \Psi_j$.
By Lemma~\ref{lem:covers}\ref{lem:covers:C}, $\theta \minmin g_j$, and so $\phi \minmin g_j$.
It follows that there is a $\pi \colon \nset{\arity{g_j}} \to \nset{\arity{\psi}}$ such that $\phi \minorant (g_j)_\pi$, so $(g_j)_\pi(\vect{a} \tau) = g_j(\vect{a} \tau \pi) = 1$ for all $\vect{a} \in T$.
We conclude that condition \ref{helpful:k-true} of Definition~\ref{def:helpful} is satisfied by any $G$ with $\{g_1, \dots, g_p\} \subseteq G$.

It remains to verify that also condition \ref{helpful:both} is satisfied.
Let $\vect{a} \in f^{-1}(1)$ and $\vect{b} \in f^{-1}(0)$.
Since $f$ is monotone, we have $\vect{a} \nleq \vect{b}$, so there exists an $i \in \nset{n}$ such that $a_i = 1$ and $b_i = 0$.
Since each $g_i$ is monotone and nonconstant, we have $\id \leq g_i$.
Since $\id_i(\vect{a}) = 1$ and $\id_i(\vect{b}) = 0$, we are done.
\end{proof}

\begin{proposition}
\label{prop:UkT-gen:Mneg}
Assume that
$\Theta \in \Ideals(\posetAllleq{k}_\clMneg)$
and $\Psi_1, \dots, \Psi_p$ are the lower covers of $\Theta$ in $\Ideals(\posetAllleq{k}_\clMneg)$.
For $i \in \nset{p}$, let $g_i \in (\clKlik{k}{\Theta} \cap \clMcneg) \setminus (\clKlik{k}{\Psi_i} \cap \clMcneg) = (\clKlik{k}{\Theta} \cap \clMineg) \setminus (\clKlik{k}{\Psi_i} \cap \clMineg)$.
Then $\gen{g_1, \dots, g_p} = \clKlik{k}{\Theta} \cap \clMcneg$
and for any $c \in (\clKlik{k}{\Theta} \cap \clMineg) \setminus (\clKlik{k}{\Theta} \cap \clMcneg)$, $\gen{g_1, \dots, g_p, c} = \clKlik{k}{\Theta} \cap \clMineg$.
\end{proposition}

\begin{proof}
This follows from Proposition~\ref{prop:UkT-gen:M} by taking inner negations and applying Lemma~\ref{lem:stability-nd}.
\end{proof}

\begin{proposition}
\label{prop:UkT-gen:R}
Assume that
$\Theta \in \Ideals(\posetAllleq{k}_\clRefl)$
and $\Psi_1, \dots, \Psi_p$ are the lower covers of $\Theta$ in $\Ideals(\posetAllleq{k}_\clRefl)$.
For $i \in \nset{p}$, let $g_i \in (\clKlik{k}{\Theta} \cap \clReflOO) \setminus (\clKlik{k}{\Psi_i} \cap \clReflOO)$.
Then $\gen{g_1, \dots, g_p} = \clKlik{k}{\Theta} \cap \clReflOO$.
\end{proposition}

\begin{proof}
We are going to show that $(\clKlik{k}{\Theta} \cap \clReflOO) \setminus \clVako$ is $(G,k)$\hyp{}semibisectable with $G = \{g_1, \dots, g_p\}$.
Since each $g_i$ is in $\clOO$, we have $0 \minor g_i$ and hence $0 \in \gen{G}$.
Therefore,
$\clKlik{k}{\Theta} \cap \clReflOO = ((\clKlik{k}{\Theta} \cap \clReflOO) \setminus \clVako) \cup \clVako
\subseteq \gen{g_1, \dots, g_p} \cup \gen{0}
\subseteq \gen{g_1, \dots, g_p} = \clKlik{k}{\Theta} \cap \clReflOO$.

Let $f \in (\clKlik{k}{\Theta} \cap \clReflOO) \setminus \clVako$.
Then there is a $\phi \in \Theta$ and $\tau \colon \nset{\arity{\phi}} \to \nset{n}$ such that $\phi(\vect{a} \tau) = 1$ for all $\vect{a} \in T$.
There is a maximal element $\theta$ of $\Theta$ such that $\phi \minmin \theta$, and there is a lower cover $\Psi_j$ of $\Theta$ in $\Ideals(\clAllleq{k} / \mathord{\eqminmin_\clRefl}, \mathord{\minmin_\clRefl})$ such that $\theta \in \Theta \setminus \Psi_j$.
By Lemma~\ref{lem:covers}\ref{lem:covers:C}, $\theta \minmin g_j$, and so $\phi \minmin g_j$.
It follows that there is a $\pi \colon \nset{\arity{g_j}} \to \nset{\arity{\psi}}$ such that $\phi \minorant (g_j)_\pi$, so $(g_j)_\pi(\vect{a} \tau) = g_j(\vect{a} \tau \pi) = 1$ for all $\vect{a} \in T$.
We conclude that condition \ref{helpful:k-true} of Definition~\ref{def:helpful} is satisfied by any $G$ with $\{g_1, \dots, g_p\} \subseteq G$.

It remains to verify that also condition \ref{helpful:both} is satisfied.
Let $\vect{a} \in f^{-1}(1)$ and $\vect{b} \in f^{-1}(0)$.
Since $f \in \clReflOO$, we have $\vect{a} \neq \vect{b}$, $\vect{a} \neq \overline{\vect{b}}$ , and $\vect{a} \notin \{\vect{0}, \vect{1}\}$, so there exist $i, j \in \nset{n}$ such that $a_i \neq b_i$ and $a_j = b_j$.
It is not difficult to see that these indices can be chosen in such a way that $a_i \neq a_j$ and hence $b_i = b_j$.
Since each $g_i$ is a nonconstant function in $\clReflOO$, we have $\mathord{+} \leq g_i$.
We have $\mathord{+}_{ij}(\vect{a}) = 1$ and $\mathord{+}_{ij}(\vect{b}) = 0$, and we are done.
\end{proof}


\subsection{Lower covers of all remaining $(\clIc,\mathsf{MU}^k_{01})$\hyp{}clonoids}

It remains to consider the remaining $(\clIc,\clMcUk{k})$\hyp{}clonoids that are not of the form $\clKlik{k}{\Theta}$ or $\clKlik{k}{\Theta} \cap C$ for some clonoid $C$, i.e., they lie in the non\hyp{}shaded parts of the Hasse diagram in Figure~\ref{fig:McUk-stable}.
We verify that the lower covers of each such $(\clIc,\clMcUk{k})$\hyp{}clonoid are precisely as suggested by the diagram, for every $k \geq 2$.

\begin{proposition}
\label{prop:empty}
For any $k \geq 2$, $\clEmpty$ is the least $(\clIc,\clMcUk{k})$\hyp{}clonoid and, consequently, $\gen{\clEmpty} = \clEmpty$.
\end{proposition}

\begin{proof}
Trivial.
\end{proof}

\begin{proposition}
\label{prop:constants}
Let $k \geq 2$.
\begin{enumerate}[label=\textup{(\roman*)}]
\item\label{prop:constants:0} For any $f \in \clVako$, we have $\gen{f} = \clVako$.
\item\label{prop:constants:1} For any $f \in \clVaki$, we have $\gen{f} = \clVaki$.
\item\label{prop:constants:both} For any $f, g \in \clVak$ with $f \notin \clVako$, $g \notin \clVaki$, we have $\gen{f,g} = \clVak$.
\end{enumerate}
\end{proposition}

\begin{proof}
\ref{prop:constants:0}
By Lemma~\ref{lem:F-closure}, we have $\gen{f} = \clMcUk{k} (f \clIc) = \clMcUk{k} \, \clVako = \clVako$.

\ref{prop:constants:1}
Similarly, $\gen{f} = \clMcUk{k} (f \clIc) = \clMcUk{k} \, \clVaki = \clVaki$.

\ref{prop:constants:both}
We have $f \in \clVaki$ and $g \in \clVako$, so it follows from parts \ref{prop:constants:0} and \ref{prop:constants:1} that
$\clVak = \clVako \cup \clVaki = \gen{g} \cup \gen{f} \subseteq \gen{f,g} \subseteq \clVak$.
\end{proof}

\begin{proposition}
\label{prop:R}
Let $k \geq 2$.
\begin{enumerate}[label=\textup{(\roman*)}]
\item\label{prop:R:R00C} For any $f, g \in \clReflOOC$ with $f \notin \clReflOO$ and $g \notin \clVak$, we have $\gen{f, g} = \clReflOOC$.
\item\label{prop:R:R11} For any $f \in \clReflII$ with $f \notin \clVaki$, we have $\gen{f} = \clReflII$.
\item\label{prop:R:R11C} For any $f, g \in \clReflIIC$ with $f \notin \clReflII$ and $g \notin \clVak$, we have $\gen{f, g} = \clReflIIC$.
\item\label{prop:R:R} For any $f, g \in \clRefl$ with $f \notin \clReflOOC$ and $g \notin \clReflIIC$, we have $\gen{f,g} = \clRefl$.
\end{enumerate}
\end{proposition}

\begin{proof}
\ref{prop:R:R00C}
We have $f \in \clVaki$ and $g \in \clReflOO \setminus \clVako$, so $1 \leq f$ and $0 \leq g$ and there is a binary minor $g' \leq g$ such that $g'(0,0) = g'(1,1) = 0$ and $g'(0,1) = g'(1,0) = 1$, i.e., $\mathord{+} \leq g$.
It suffices to show that $\clReflOO \setminus \clVako$ is $(G,k)$\hyp{}semibisectable with $G := \{1, \mathord{+}\}$,
because from this it follows by Lemma~\ref{lem:helpful} and Proposition~\ref{prop:constants}\ref{prop:constants:both} that
$\clReflOOC = (\clReflOO \setminus \clVako) \cup \clVak \subseteq \gen{1, \mathord{+}} \cup \gen{0, 1} \subseteq \gen{f, g} \subseteq \clReflOOC$.

Let $\theta \in \clReflOO \setminus \clVako$.
Condition \ref{helpful:k-true} is obviously satisfied because $1 \in G$.
In order to verify condition \ref{helpful:both}, let $\vect{a} \in \theta^{-1}(1)$ and $\vect{b} \in \theta^{-1}(0)$.
Then $\vect{a} \neq \vect{b}$ and $\vect{a} \neq \overline{\vect{b}}$, so there exist $i$ and $j$ such that $a_i \neq b_i$ and $a_j = b_j$; moreover, $\vect{a} \notin \{\vect{0}, \vect{1}\}$, so there exists a $k$ such that $a_k \neq a_i$.
If $a_i \neq a_j$, then $\mathord{+}_{ij}(\vect{a}) = 1$ and $\mathord{+}_{ij}(\vect{b}) = 0$.
Assume now that $a_i = a_j$; thus $b_i \neq b_j$ and either $b_i = b_k$ or $b_j = b_k$.
Then $\tau(\vect{a}) = 1$ and $\tau(\vect{b}) = 0$ for some $\tau \in \{\mathord{+}_{ik}, \mathord{+}_{jk}\}$.

\ref{prop:R:R11}
We have $1 \leq f$ and there exists a binary minor $f' \leq f$ such that $f'(0,0) = f'(1,1) = 1$, $f'(0,1) = f'(1,0) = 0$, i.e., $\mathord{\leftrightarrow} \leq f$.
It suffices to show that $\clReflII \setminus \clVaki$ is $(G,k)$\hyp{}semibisectable with $G := \{1, \mathord{\leftrightarrow}\}$,
because from this it follows by Lemma~\ref{lem:helpful} and Proposition~\ref{prop:constants}\ref{prop:constants:1} that
$\clReflII = (\clReflII \setminus \clVaki) \cup \clVaki \subseteq \gen{1, \mathord{\leftrightarrow}} \cup \gen{1} \subseteq \gen{f} \subseteq \clReflII$.

Let $\theta \in \clReflII \setminus \clVaki$.
Condition \ref{helpful:k-true} is obviously satisfied because $1 \in G$.
In order to verify condition \ref{helpful:both}, let $\vect{a} \in \theta^{-1}(1)$ and $\vect{b} \in \theta^{-1}(0)$.
Then $\vect{a} \neq \vect{b}$ and $\vect{a} \neq \overline{\vect{b}}$, so there exist $i$ and $j$ such that $a_i \neq b_i$ and $a_j = b_j$; moreover, $\vect{b} \notin \{\vect{0}, \vect{1}\}$, so there exists a $k$ such that $b_k \neq b_i$.
If $a_i = a_j$, then $\mathord{\leftrightarrow}_{ij}(\vect{a}) = 1$ and $\mathord{\leftrightarrow}_{ij}(\vect{b}) = 0$.
Assume now that $a_i \neq a_j$; thus $b_i = b_j \neq b_k$ and either $a_i = a_k$ or $a_j = a_k$.
Then $\tau(\vect{a}) = 1$ and $\tau(\vect{b}) = 0$ for some $\tau \in \{\mathord{\leftrightarrow}_{ik}, \mathord{\leftrightarrow}_{jk}\}$.

\ref{prop:R:R11C}
We have $f \in \clVako$ and $g \in \clReflII \setminus \clVaki$, so by part \ref{prop:R:R11} and Proposition~\ref{prop:constants}\ref{prop:constants:0} it holds that
$\clReflIIC = \clReflII \cup \clVako = \gen{g} \cup \gen{f} \subseteq \gen{f, g} \subseteq \clReflIIC$.

\ref{prop:R:R}
We have $f \in \clReflII \setminus \clVaki$ and $g \in \clReflOO \setminus \clVako$.
We have $1 \leq f$, and it follows from parts \ref{prop:R:R00C} and \ref{prop:R:R11} that
$\clRefl = (\clReflOOC) \cup \clReflII = \gen{1, g} \cup \gen{f} \subseteq \gen{f, g} \subseteq \clRefl$.
\end{proof}

\begin{proposition}
\label{prop:M}
Let $k \geq 2$.
\begin{enumerate}[label=\textup{(\roman*)}]
\item\label{prop:M:Mi} For any $f, g \in \clMi$ with $f \notin \clMc$ and $g \notin \clVaki$, we have $\gen{f, g} = \clMi$.
\item\label{prop:M:M} For any $f, g, h \in \clM$ with $f \notin \clMo$, $g \notin \clMi$, and $h \notin \clVak$, we have \linebreak $\gen{f,g,h} = \clM$.
\end{enumerate}
\end{proposition}

\begin{proof}
\ref{prop:M:Mi}
We have $f \in \clVaki$ and $g \in \clMc$, so $1 \leq f$ and $\id \leq g$.
It suffices to show that $\clMi \setminus \clVaki$ is $(G,k)$\hyp{}semibisectable with $G := \{1, \id\}$.
because from this it follows by Lemma~\ref{lem:helpful} and Proposition~\ref{prop:constants}\ref{prop:constants:1} that
$\clMi = (\clMi \setminus \clVaki) \cup \clVaki \subseteq \gen{1, \id} \cup \gen{1} \subseteq \gen{f, g} \subseteq \clMi$.

Let $\theta \in \clMi \setminus \clVaki$.
Condition \ref{helpful:k-true} is obviously satisfied because $1 \in G$.
In order to verify condition \ref{helpful:both}, let $\vect{a} \in \theta^{-1}(1)$ and $\vect{b} \in \theta^{-1}(0)$.
Then $\vect{a} \nleq \vect{b}$, so there exists an $i$ such that $a_i = 1$ and $b_i = 0$; hence $\id_i(\vect{a}) = 1$ and $\id_i(\vect{b}) = 0$.

\ref{prop:M:M}
We have $f \in \clVaki$, $g \in \clVako$, and $h \in \clMc$, so $1 \leq f$, $0 \leq g$, and $\id \leq h$.
By part \ref{prop:M:Mi} and Proposition~\ref{prop:constants}\ref{prop:constants:0} it holds that
$\clM = \clMi \cup \clVako = \gen{1, \id} \cup \gen{0} \subseteq \gen{f, g, h} \subseteq \clM$.
\end{proof}

\begin{proposition}
\label{prop:Mneg}
Let $k \geq 2$.
\begin{enumerate}[label=\textup{(\roman*)}]
\item\label{prop:Mneg:Moneg} For any $f, g \in \clMoneg$ with $f \notin \clMcneg$ and $g \notin \clVaki$, we have $\gen{f, g} = \clMoneg$.
\item\label{prop:Mneg:Mneg} For any $f, g, h \in \clMneg$ with $f \notin \clMoneg$, $g \notin \clMineg$, and $h \notin \clVak$, we have \linebreak $\gen{f,g,h} = \clMneg$.
\end{enumerate}
\end{proposition}

\begin{proof}
This follows from Proposition~\ref{prop:M} by taking inner negations and applying Lemma~\ref{lem:stability-nd}.
\end{proof}

\begin{proposition}
\label{prop:OIC}
Let $k \geq 2$.
\begin{enumerate}[label=\textup{(\roman*)}]
\item\label{prop:OIC:OICI}
For any $f, g \in \clOICI$ with $f \notin \clOI$ and $g \notin \clMi$, we have $\gen{f,g} = \clOICI$.
\item\label{prop:OIC:OIC}
For any $f, g, h \in \clOIC$ with $f \notin \clOICO$, $g \notin \clOICI$, and $h \notin \clM$, we have $\gen{f,g,h} = \clOIC$.
\end{enumerate}
\end{proposition}

\begin{proof}
\ref{prop:OIC:OICI}
We have $f \in \clVaki$ and $g \in \clOI \setminus \clM$, so $1 \leq f$, $\id \leq g$, and there is a ternary minor $g' \leq g$ such that $g'(0,0,0) = 0$, $g'(0,0,1) = 1$, $g'(0,1,1) = 0$, $g'(1,1,1) = 1$.
It suffices to show that $\clOI$ is $(G,k)$\hyp{}semibisectable with $G := \{1, \id, g'\}$.
because from this it follows by Lemma~\ref{lem:helpful} and Proposition~\ref{prop:constants}\ref{prop:constants:1} that
$\clOICI \subseteq \gen{1, \id, g'} \cup \gen{1} \subseteq \gen{f, g} \subseteq \clOICI$.

Let $\theta \in \clOI$.
Condition \ref{helpful:k-true} is obviously satisfied because $1 \in G$.
In order to verify condition \ref{helpful:both}, let $\vect{a} \in \theta^{-1}(1)$ and $\vect{b} \in \theta^{-1}(0)$.
We have $\vect{a} \neq \vect{0}$ and $\vect{b} \neq \vect{1}$, so there are $i$ and $j$ such that $a_i = 1$ and $b_j = 0$; moreover, $\vect{a} \neq \vect{b}$, so there is a $k$ such that $a_k \neq b_k$.
If $a_k = 1$ and $b_k = 0$, then $\id_k(\vect{a}) = 1$ and $\id_k(\vect{b}) = 0$.
Assume now that there is no $\ell$ such that $a_\ell = 1$ and $b_\ell = 0$; consequently $b_i = 1$, $a_j = 0$, $a_k = 0$, and $b_k = 1$, and we have $g'_{jki}(\vect{a}) = 1$ and $g'_{jki}(\vect{b}) = 0$.

\ref{prop:OIC:OIC}
We have $f \in \clVaki$, $g \in \clVako$, and $h \in \clOI \setminus \clM$, so $1 \leq f$ and $0 \leq g$.
By part \ref{prop:Mneg:Moneg} and Proposition~\ref{prop:constants}\ref{prop:constants:0} it holds that
$\clOIC = (\clOICI) \cup \clVako = \gen{1,h} \cup \gen{0} \subseteq \gen{f,g,h} \subseteq \clOIC$.
\end{proof}

\begin{proposition}
\label{prop:IOC}
Let $k \geq 2$.
\begin{enumerate}[label=\textup{(\roman*)}]
\item\label{prop:IOC:IOCI}
For any $f, g \in \clIOCI$ with $f \notin \clIO$ and $g \notin \clMoneg$, we have $\gen{f,g} = \clIOCI$.
\item\label{prop:IOC:IOC}
For any $f, g, h \in \clIOC$ with $f \notin \clIOCO$, $g \notin \clIOCI$, and $h \notin \clMneg$, we have $\gen{f,g,h} = \clIOC$.
\end{enumerate}
\end{proposition}

\begin{proof}
This follows from Proposition~\ref{prop:OIC} by taking inner negations and applying Lemma~\ref{lem:stability-nd}.
\end{proof}

\begin{proposition}
\label{prop:OOC}
Let $k \geq 2$.
For any $f, g \in \clOOC$ with $f \notin \clOO$ and $g \notin \clReflOOC$, we have $\gen{f,g} = \clOOC$.
\end{proposition}

\begin{proof}
We have $f \in \clVaki$ and $g \in \clOO \setminus \clReflOO$, so $1 \leq f$, $0 \leq g$, and there is a binary minor $g' \leq g$ such that $g'(0,0) = 0$, $g'(1,1) = 0$, $g'(0,1) = 0$, $g'(1,0) = 1$, i.e., $g' = \mathord{\nrightarrow}$.
It suffices to show that $\clOO \setminus \clVak$ is $(G,k)$\hyp{}semibisectable with $G := \{1, \mathord{\nrightarrow}\}$, because from this it follows by Lemma~\ref{lem:helpful} and Proposition~\ref{prop:constants}\ref{prop:constants:both} that
$\clOOC = (\clOO \setminus \clVak) \cup \clVak \subseteq \gen{1, \mathord{\nrightarrow}} \cup \gen{0,1} \subseteq \gen{f,g} \subseteq \clOOC$.

Let $\theta \in \clOO \setminus \clVak$.
Condition \ref{helpful:k-true} is obviously satisfied because $1 \in G$.
In order to verify condition \ref{helpful:both}, let $\vect{a} \in \theta^{-1}(1)$ and $\vect{b} \in \theta^{-1}(0)$.
We have $\vect{a} \notin \{\vect{0}, \vect{1}\}$, so there exist $i$ and $j$ such that $a_i = 0$, $a_j = 1$.
We have $\vect{a} \neq \vect{b}$, so there is a $k$ such that $a_k \neq b_k$.
If $a_k = 1$ and $b_k = 0$, then $\mathord{\nrightarrow}_{ki}(\vect{a}) = 1$ and $\mathord{\nrightarrow}_{ki}(\vect{b}) = 0$.
If $a_k = 0$ and $b_k = 1$, then $\mathord{\nrightarrow}_{jk}(\vect{a}) = 1$ and $\mathord{\nrightarrow}_{jk}(\vect{b}) = 0$.
\end{proof}

\begin{proposition}
\label{prop:OXC}
Let $k \geq 2$.
For any $f, g, h \in \clOXC$ with $f \notin \clOX$, $g \notin \clOOC$, $h \notin \clOIC$, we have $\gen{f,g,h} = \clOXC$.
\end{proposition}

\begin{proof}
We have $f \in \clVaki$, $g \in \clOI$, and $h \in \clOO \setminus \clVak$, so $1 \leq f$, $\id \leq g$, $0 \leq h$, and there is a binary minor $h' \leq h$ such that $h'(0,0) = 0$, $h'(1,1) = 0$, $h'(0,1) = 1$.
It suffices to show that $\clOX \setminus \clVak$ is $(G,k)$\hyp{}semibisectable with $G := \{1, \id, h'\}$, because from this it follows by Lemma~\ref{lem:helpful} and Proposition~\ref{prop:constants}\ref{prop:constants:both} that
$\clOXC = (\clOX \setminus \clVak) \cup \clVak \subseteq \gen{1, \id, h'} \cup \gen{0,1} \subseteq \gen{f,g,h} \subseteq \clOXC$.

Let $\theta \in \clOX \setminus \clVak$.
Condition \ref{helpful:k-true} is obviously satisfied because $1 \in G$.
In order to verify condition \ref{helpful:both}, let $\vect{a} \in \theta^{-1}(1)$ and $\vect{b} \in \theta^{-1}(0)$.
We have $\vect{a} \neq \vect{0}$, so there is an $i$ such that $a_i = 1$.
We have $\vect{a} \neq \vect{b}$, so there is a $j$ such that $a_j \neq b_j$.
If $a_j = 1$ and $b_j = 0$, then $\id_j(\vect{a}) = 1$ and $\id_j(\vect{b}) = 0$.
Assume now that there is no $\ell$ such that $a_\ell = 1$ and $b_\ell = 0$.
Then $a_j = 0$, $b_j = 1$, and $b_i = 1$, and we have $h'_{ji}(\vect{a}) = 1$ and $h'_{ji}(\vect{b}) = 0$.
\end{proof}

\begin{proposition}
\label{prop:XOC}
Let $k \geq 2$.
For any $f, g, h \in \clXOC$ with $f \notin \clXO$, $g \notin \clOOC$, $h \notin \clIOC$, we have $\gen{f,g,h} = \clXOC$.
\end{proposition}

\begin{proof}
This follows from Proposition~\ref{prop:OXC} by taking inner negations and applying Lemma~\ref{lem:stability-nd}.
\end{proof}

\begin{proposition}
\label{prop:II}
Let $k \geq 2$.
\begin{enumerate}[label=\textup{(\roman*)}]
\item\label{prop:II:II}
For any $f \in \clII \setminus \clReflII$, we have $\gen{f} = \clII$.
\item\label{prop:II:IIC}
For any $f, g \in \clIIC$ with $f \notin \clReflIIC$ and $g \notin \clII$, we have $\gen{f,g} = \clIIC$.
\end{enumerate}
\end{proposition}

\begin{proof}
\ref{prop:II:II}
We have $1 \leq f$, and there is a binary minor $f' \leq f$ such that $f'(0,0) = f'(1,1) = 1$ and $f'(0,1) = 1$, $f'(1,0) = 0$, i.e., $\mathord{\rightarrow} \leq f$.
It suffices to show that $\clII \setminus \clVaki$ is $(G,k)$\hyp{}semibisectable with $G := \{1, \mathord{\rightarrow}\}$, because from this it follows by Lemma~\ref{lem:helpful} and Proposition~\ref{prop:constants}\ref{prop:constants:1} that
$\clII = (\clII \setminus \clVaki) \cup \clVaki \subseteq \gen{1, \mathord{\rightarrow}} \cup \gen{1} \subseteq \gen{f} \subseteq \clII$.

Let $\theta \in \clII \setminus \clVaki$.
Condition \ref{helpful:k-true} is obviously satisfied because $1 \in G$.
In order to verify condition \ref{helpful:both}, let $\vect{a} \in \theta^{-1}(1)$ and $\vect{b} \in \theta^{-1}(0)$.
Then $\vect{a} \neq \vect{b}$, so there is an $i$ such that $a_i \neq b_i$.
Since $\vect{b} \notin \{\vect{0}, \vect{1}\}$, there is a $j$ such that $b_j \neq b_i$.
We have $\tau(\vect{a}) = 1$ and $\tau(\vect{b}) = 0$ for some $\tau \in \{\mathord{\rightarrow}_{ij}, \mathord{\rightarrow}_{ji}\}$.

\ref{prop:II:IIC}
We have $f \in \clII \setminus \clReflII$ and $g \in \clVako$.
It follows from part \ref{prop:II:II} and Proposition~\ref{prop:constants}\ref{prop:constants:0} that
$\clIIC = \clII \cup \clVako = \gen{f} \cup \gen{g} \subseteq \gen{f,g} \subseteq \clIIC$.
\end{proof}

\begin{proposition}
\label{prop:IX}
Let $k \geq 2$.
\begin{enumerate}[label=\textup{(\roman*)}]
\item\label{prop:IX:IX}
For any $f, g \in \clIX$ with $f \notin \clII$ and $g \notin \clIOCI$, we have $\gen{f,g} = \clIX$.
\item\label{prop:IX:IXC}
For any $f, g, h \in \clIXC$ with $f \notin \clIIC$, $g \notin \clIOC$, $h \notin \clIX$, we have $\gen{f,g,h} = \clIXC$.
\item\label{prop:IX:XI}
For any $f, g \in \clXI$ with $f \notin \clII$ and $g \notin \clOICI$, we have $\gen{f,g} = \clXI$.
\item\label{prop:IX:XIC}
For any $f, g, h \in \clXIC$ with $f \notin \clIIC$, $g \notin \clOIC$, $h \notin \clXI$, we have $\gen{f,g,h} = \clXIC$.
\end{enumerate}
\end{proposition}

\begin{proof}
\ref{prop:IX:IX}
We have $f \in \clIO$ and $g \in \clII \setminus \clVak$, so $\mathord{\neg} \leq f$, $1 \leq g$, and there exists a binary minor $g' \leq g$ such that $g'(0,0) = g'(1,1) = 1$ and $g'(0,1) = 0$.
It suffices to show that $\clIX \setminus \clVaki$ is $(G,k)$\hyp{}semibisectable with $G := \{1, \mathord{\neg}, g'\}$, because from this it follows by Lemma~\ref{lem:helpful} and Proposition~\ref{prop:constants}\ref{prop:constants:1} that
$\clIX = (\clIX \setminus \clVaki) \cup \clVaki \subseteq \gen{1, \mathord{\neg}, g'} \cup \gen{1} \subseteq \gen{f,g} \subseteq \clIX$.

Let $\theta \in \clIX \setminus \clVaki$.
Condition \ref{helpful:k-true} is obviously satisfied because $1 \in G$.
In order to verify condition \ref{helpful:both}, let $\vect{a} \in \theta^{-1}(1)$ and $\vect{b} \in \theta^{-1}(0)$.
Then $\vect{a} \neq \vect{b}$, so there is an $i$ such that $a_i \neq b_i$.
If $a_i = 0$ and $b_i = 1$, then $\mathord{\neg}_i(\vect{a}) = 1$ and $\mathord{\neg}_i(\vect{b}) = 0$.
Assume now that there is no $\ell$ such that $a_\ell = 0$ and $b_\ell = 1$; thus $a_i = 1$ and $b_i = 0$.
We have $\vect{b} \neq \vect{0}$, so there is a $j$ such that $b_j = 1$ and hence $a_j = 1$.
Therefore $g'_{ij}(\vect{a}) = 1$ and $g'_{ij}(\vect{b}) = 0$.

\ref{prop:IX:IXC}
We have $f \in \clIO$, $g \in \clII \setminus \clVak$, $h \in \clVako$, so $0 \leq h$.
It follows from part \ref{prop:IX:IX} and Proposition~\ref{prop:constants}\ref{prop:constants:0} that
$\clIXC = \clIX \cup \clVako = \gen{f,g} \cup \gen{0} \subseteq \gen{f,g,h} \subseteq \clIXC$.

\ref{prop:IX:XI}, \ref{prop:IX:XIC}
These statements follow from \ref{prop:IX:IX} and \ref{prop:IX:IXC} by taking inner negations and applying Lemma~\ref{lem:stability-nd}.
\end{proof}

\begin{proposition}
\label{prop:Eq}
Let $k \geq 2$.
For any $f, g, h \in \clEq$ with $f \notin \clOOC$, $g \notin \clIIC$, $h \notin \clRefl$, we have $\gen{f,g,h} = \clEq$.
\end{proposition}

\begin{proof}
We have $f \in \clII \setminus \clVak$, $g \in \clOO \setminus \clVak$, $h \in \clEq \setminus \clRefl$.
Thus $1 \leq f$ and $0 \leq g$, and there exist minors $f' \leq f$, $g' \leq g$, $h' \leq h$ such that
$f'(0,0) = f(1,1) = 1$, $f'(0,1) = 0$, $g'(0,0) = g'(1,1) = 0$, $g'(0,1) = 1$, $h'(0,0) = h'(1,1) = h'(0,1) \neq h'(1,0)$.
It suffices to show that $\clEq \setminus \clVak$ is $(G,k)$\hyp{}semibisectable with $G := \{1, f', g', h'\}$, because from this it follows by Lemma~\ref{lem:helpful} and Proposition~\ref{prop:constants}\ref{prop:constants:both} that
$\clEq = (\clEq \setminus \clVak) \cup \clVak \subseteq \gen{1, f', g', h'} \cup \gen{0, 1} \subseteq \gen{f, g, h} \subseteq \clEq$.

Let $\theta \in \clEq \setminus \clVak$.
Condition \ref{helpful:k-true} is obviously satisfied because $1 \in G$.
In order to verify condition \ref{helpful:both}, let $\vect{a} \in \theta^{-1}(1)$ and $\vect{b} \in \theta^{-1}(0)$.
Then $\vect{a} \neq \vect{b}$, so there is an $i$ such that $a_i \neq b_i$.
Furthermore, $\{\vect{a}, \vect{b}\} \neq \{\vect{0}, \vect{1}\}$, so there exists a $j$ such that $a_i \neq a_j$ or $b_i \neq b_j$.
If $a_j \neq b_j$, then it holds that $a_i \neq a_j$ and $b_i \neq b_j$, and hence $\tau(\vect{a}) = 1$ and $\tau(\vect{b}) = 0$ for some $\tau \in \{h'_{ij}, h'_{ji}\}$.
If $a_j = b_j$, then we have $\tau(\vect{a}) = 1$ and $\tau(\vect{b}) = 0$ for some $\tau \in \{f'_{ij}, f'_{ji}, g'_{ij}, g'_{ji}\}$.
\end{proof}

\begin{proposition}
\label{prop:Eiio}
Let $k \geq 2$.
\begin{enumerate}[label=\textup{(\roman*)}]
\item\label{prop:Eiio:Eiio}
For any $f, g, h \in \clEiio$ with $f \notin \clOXC$, $g \notin \clXIC$, $h \notin \clEq$, we have $\gen{f, g, h} = \clEiio$.
\item\label{prop:Eiio:Eioi}
For any $f, g, h \in \clEioi$ with $f \notin \clXOC$, $g \notin \clIXC$, $h \notin \clEq$, we have $\gen{f, g, h} = \clEioi$.
\end{enumerate}
\end{proposition}

\begin{proof}
\ref{prop:Eiio:Eiio}
We have $f \in \clII \setminus \clVak$, $g \in \clOO \setminus \clVak$, $h \in \clOI$, so $1 \leq f$, $0 \leq g$, $\id \leq h$, and there exist binary minors $f' \leq f$ and $g' \leq g$ such that $f'(0,0) = f'(1,1) = 1$, $f'(0,1) = 0$, and $g'(0,0) = g'(1,1) = 0$, $g'(0,1) = 1$.
It suffices to show that $\clEiio \setminus \clVak$ is $(G,k)$\hyp{}semibisectable with $G := \{1, \mathord{\id}, f', g'\}$, because from this it follows by Lemma~\ref{lem:helpful} and Proposition~\ref{prop:constants}\ref{prop:constants:both} that
$\clEiio = (\clEiio \setminus \clVak) \cup \clVak \subseteq \gen{1, \mathord{\id}, f', g'} \cup \gen{0, 1} \subseteq \gen{f, g, h} \subseteq \clEiio$.

Let $\theta \in \clEiio \setminus \clVak$.
Condition \ref{helpful:k-true} is obviously satisfied because $1 \in G$.
In order to verify condition \ref{helpful:both}, let $\vect{a} \in \theta^{-1}(1)$ and $\vect{b} \in \theta^{-1}(0)$.
Then $\vect{a} \neq \vect{b}$, so there is an $i$ such that $a_i \neq b_i$.
If there is a $j$ such that $a_j = 1$ and $b_j = 0$, then $\id_j(\vect{a}) = 1$ and $\id_j(\vect{b}) = 0$.
Assume now that there is no $j$ such that $a_j = 1$ and $b_j = 0$; thus $a_i = 0$ and $b_i = 1$.
Since $(\vect{a},\vect{b}) \neq (\vect{0},\vect{1})$, either there exists a $k$ such that $a_k = 1$ and hence $b_k = 1$, or there exists an $\ell$ such that $b_\ell = 0$ and hence $a_\ell = 0$.
In the former case, we have $g'_{ik}(\vect{a}) = 1$ and $g'_{ik}(\vect{b}) = 0$; in the latter case we have $f'_{\ell i}(\vect{a}) = 1$ and $f'_{\ell i}(\vect{b}) = 0$.

\ref{prop:Eiio:Eioi}
This follows from statement \ref{prop:Eiio:Eiio} by taking inner negations and applying Lemma~\ref{lem:stability-nd}.
\end{proof}

\begin{proposition}
\label{prop:All}
Let $k \geq 2$.
For any $f, g, h \in \clAll$ with $f \notin \clEiio$, $g \notin \clEioi$, $h \notin \clEiii$, we have $\gen{f, g, h} = \clAll$.
\end{proposition}

\begin{proof}
We have $f \in \clIO$, $g \in \clOI$, $h \in \clII$, so $\mathord{\neg} \leq f$, $\id \leq g$, $1 \leq h$.
Note also that $0 = \mathord{\wedge}(\mathord{\neg}, \mathord{\id}) \in \clMcUk{k} ( \{f, g\} \clIc ) = \gen{f, g}$ and hence $\clVako \subseteq \gen{f,g}$.
It suffices to show that $\clAll \setminus \clVak$ is $(G,k)$\hyp{}semibisectable with $G := \{\mathord{\id}, \mathord{\neg}, 1\}$, because from this it follows by Lemma~\ref{lem:helpful} and Proposition~\ref{prop:constants}\ref{prop:constants:both} that
$\clAll = (\clAll \setminus \clVak) \cup \clVako \cup \clVaki \subseteq \gen{\mathord{\id}, \mathord{\neg}, 1} \cup \gen{f, g} \cup \gen{1} \subseteq \gen{f, g, h} \subseteq \clAll$.

Let $\theta \in \clAll \setminus \clVak$.
Condition \ref{helpful:k-true} is obviously satisfied because $1 \in G$.
Condition \ref{helpful:both} is also satisfied, because if $\vect{a} \in \theta^{-1}(1)$ and $\vect{b} \in \theta^{-1}(0)$, then $\vect{a} \neq \vect{b}$, so there exists an $i$ such that $a_i \neq b_i$.
Then $\tau(\vect{a}) = 1$ and $\tau(\vect{b}) = 0$ for some $\tau \in \{\id_i, \mathord{\neg}_i\}$.
\end{proof}

\begin{proof}[Proof of Theorem~\ref{thm:IcMcUk-clonoids}]
The fact that the classes listed are $(\clIc,\clMcUk{k})$\hyp{}clonoids follows from Propositions~\ref{prop:stable-1} and \ref{prop:UTheta-stability} and from the fact that the intersection of $(\clIc,\clMcUk{k})$\hyp{}clonoids is a $(\clIc,\clMcUk{k})$\hyp{}clonoid.
That these are the only $(\clIc,\clMcUk{k})$\hyp{}clonoids follows from
Propositions~\ref{prop:UkT-gen:plain}--\ref{prop:UkT-gen:R} and
\ref{prop:empty}--\ref{prop:All}.
\end{proof}

\subsection{Remark: A few covering relations near the bottom of the lattice}

We now make more explicit some of the covering relations described in Theorem~\ref{thm:UkTC-covers}, in particular, ones involving minimal $(k,C)$\hyp{}closed $(\clIc,\clMcUk{k})$\hyp{}clonoids of the form $\clKlik{k}{\Theta}$ ($C \in \{\clXI, \clIX, \clM, \clMneg, \clRefl\}$).
This explains the covering relations near the bottom of the lattice that are shown in Figure~\ref{fig:McUk-stable}.

\begin{proposition}
\label{prop:Uk}
Let $k \geq 2$.
\begin{enumerate}[label=\textup{(\roman*)}]
\item\label{prop:Uk:McUk}
For any $f \in \clMcUk{k}$, we have $\gen{f} = \clMcUk{k}$.
\item\label{prop:Uk:MUk}
For any $f, g \in \clMUk{k}$ with $f \notin \clMcUk{k}$ and $g \notin \clVako$, we have $\gen{f,g} = \clMUk{k}$.
\item\label{prop:Uk:TcUk}
For any $f \in \clTcUk{k}$ with $f \notin \clMcUk{k}$, we have $\gen{f} = \clTcUk{k}$.
\item\label{prop:Uk:TcUkC0}
For any $f, g \in \clTcUkCO{k}$ with $f \notin \clTcUk{k}$ and $g \notin \clMUk{k}$, we have $\gen{f,g} = \clTcUkCO{k}$.
\item\label{prop:Uk:Uk}
For any $f, g \in \clUk{k}$ with $f \notin \clTcUkCO{k}$ and $g \notin \clUkOO{k}$, we have $\gen{f,g} = \clUk{k}$.
\end{enumerate}
\end{proposition}

\begin{remark}
\label{rem:Uk}
The claimed covering relations follow from Proposition~\ref{prop:some-kC-closed} \ref{prop:some-kC-closed:AllEmpty}, \ref{prop:some-kC-closed:C0}, \ref{prop:some-kC-closed:Uk} and Theorem~\ref{thm:UkTC-covers}.
We provide an alternative proof.
\end{remark}

\begin{proof}
\ref{prop:Uk:McUk}
We have $\id \leq f$.
It suffices to show that $\clMcUk{k}$ is $(G,k)$\hyp{}semibisectable with $G := \{\id\}$, because from this it follows by Lemma~\ref{lem:helpful} that
\[
\clMcUk{k} \subseteq \gen{\id} \subseteq \gen{f} \subseteq \clMcUk{k}.
\]

Let $\theta \in \clMcUk{k}$.
In order to verify condition \ref{helpful:k-true}, let $\vect{a}_1, \dots, \vect{a}_k \in \theta^{-1}(1)$.
By Definition~\ref{def:separating}, $\vect{a}_1 \wedge \dots \wedge \vect{a}_k \neq \vect{0}$, so there exists an $i$ such that $a_{1i} = \dots = a_{ki} = 1$.
Consequently, $\id_i(\vect{a}_1) = \dots = \id_i(\vect{a}_k) = 1$.
For condition \ref{helpful:both}, let $\vect{a} \in \theta^{-1}(1)$ and $\vect{b} \in \theta^{-1}(0)$.
We have $\vect{a} \nleq \vect{b}$, so there exists an $i$ such that $a_i = 1$ and $b_i = 0$;
thus $\id_i(\vect{a}) = 1$ and $\id_i(\vect{b}) = 0$.

\ref{prop:Uk:MUk}
We have $f \in \clVako$ and $g \in \clMcUk{k}$.
It follows from part~\ref{prop:Uk:McUk} and Proposition~\ref{prop:constants}\ref{prop:constants:0} that
$\clMUk{k} = \clMcUk{k} \cup \clVako = \gen{g} \cup \gen{0} \subseteq \gen{f,g} \subseteq \clMUk{k}$.

\ref{prop:Uk:TcUk}
We have $\id \leq f$ and there is a ternary minor $f' \leq f$ such that $f'(0,0,0) = 0$, $f'(0,0,1) = 1$, $f'(0,1,1) = 0$, $f'(1,1,1) = 1$.
It suffices to show that $\clTcUk{k}$ is $(G,k)$\hyp{}semibisectable with $G := \{\id, f'\}$, because from this it follows by Lemma~\ref{lem:helpful} that
$\clTcUk{k} \subseteq \gen{\id, f'} \subseteq \gen{f} \subseteq \clTcUk{k}$.

Let $\theta \in \clTcUk{k}$.
In order to verify condition \ref{helpful:k-true}, let $\vect{a}_1, \dots, \vect{a}_k \in \theta^{-1}(1)$.
Then there exists an $i$ such that $a_{1i} = \dots = a_{ki} = 1$.
Consequently, $\id_i(\vect{a}_1) = \dots = \id_i(\vect{a}_k) = 1$.
For condition \ref{helpful:both}, let $\vect{a} \in \theta^{-1}(1)$ and $\vect{b} \in \theta^{-1}(0)$.
Since $\vect{a} \neq \vect{0}$, $\vect{b} \neq \vect{1}$, and $\vect{a} \neq \vect{b}$, there exist $i$, $j$, $k$ such that $a_i = 1$, $b_j = 0$, and $a_k \neq b_k$.
If $a_k = 1$ and $b_k = 0$, then $\id_k(\vect{a}) = 1$ and $\id_k(\vect{b}) = 0$.
Assume that there is no $\ell$ such that $a_\ell = 1$ and $b_\ell = 0$.
Then $a_k = 0$, $b_k = 1$, $b_i = 1$, $a_j = 0$, and we have $f'_{jki}(\vect{a}) = 1$ and $f'_{jki}(\vect{b}) = 0$.

\ref{prop:Uk:TcUkC0}
We have $f \in \clVako$ and $g \in \clTcUk{k} \setminus \clMUk{k}$.
It follows from part \ref{prop:Uk:TcUk} and Proposition~\ref{prop:constants}\ref{prop:constants:0} that
$\clTcUkCO{k} = \gen{g} \cup \gen{0} \subseteq \gen{f,g} \subseteq \clTcUkCO{k}$.

\ref{prop:Uk:Uk}
We have $f \in \clUkOO{k} \setminus \clVako$ and $g \in \clTcUk{k}$,
so $0 \leq f$, $\id \leq g$, and there is a binary minor $f' \leq f$ such that $f'(0,0) = f'(1,1) = f'(0,1) = 0$, $f'(1,0) = 1$, i.e., $f' = \mathord{\nrightarrow}$.
It suffices to show that $\clUk{k} \setminus \clVako$ is $(G,k)$\hyp{}semibisectable with $G := \{\id, \mathord{\nrightarrow}\}$, because from this it follows by Lemma~\ref{lem:helpful} and Proposition~\ref{prop:constants}\ref{prop:constants:0} that
$\clUk{k} = (\clUk{k} \setminus \clVako) \cup \clVako \subseteq \gen{\id, \mathord{\nrightarrow}} \cup \gen{0} \subseteq \gen{f,g} \subseteq \clUk{k}$.

Let $\theta \in \clUk{k} \setminus \clVako$.
In order to verify condition \ref{helpful:k-true}, let $\vect{a}_1, \dots, \vect{a}_k \in \theta^{-1}(1)$.
Then there exists an $i$ such that $a_{1i} = \dots = a_{ki} = 1$.
Consequently, $\id_i(\vect{a}_1) = \dots = \id_i(\vect{a}_k) = 1$.
For condition \ref{helpful:both}, let $\vect{a} \in \theta^{-1}(1)$ and $\vect{b} \in \theta^{-1}(0)$.
Since $\vect{a} \neq \vect{0}$ and $\vect{a} \neq \vect{b}$, there exist $i$ and $j$ such that $a_i = 1$ and $a_j \neq b_j$.
If $a_j = 1$ and $b_j = 0$, then $\id_j(\vect{a}) = \id_j(\vect{b})$.
Assume that there is no $\ell$ such that $a_\ell = 1$ and $b_\ell = 0$.
Then $a_j = 0$, $b_j = 1$, $b_i = 1$, and we have $\mathord{\nrightarrow}_{ij}(\vect{a}) = 1$ and $\mathord{\nrightarrow}_{ij}(\vect{b}) = 0$.
\end{proof}

\begin{proposition}
\label{prop:Wkneg}
Let $k \geq 2$.
\begin{enumerate}[label=\textup{(\roman*)}]
\item\label{prop:Wkneg:McWkneg}
For any $f \in \clMcWkneg{k}$, we have $\gen{f} = \clMcWkneg{k}$.
\item\label{prop:Wkneg:MWkneg}
For any $f, g \in \clMWkneg{k}$ with $f \notin \clMcWkneg{k}$ and $g \notin \clVako$, we have $\gen{f,g} = \clMWkneg{k}$.
\item\label{prop:Wkneg:TcWkneg}
For any $f \in \clTcWkneg{k}$ with $f \notin \clMcWkneg{k}$, we have $\gen{f} = \clTcWkneg{k}$.
\item\label{prop:Wkneg:TcWknegC0}
For any $f, g \in \clTcWknegCO{k}$ with $f \notin \clTcWkneg{k}$ and $g \notin \clMWkneg{k}$, we have \linebreak $\gen{f,g} = \clTcWknegCO{k}$.
\item\label{prop:Wkneg:Wkneg}
For any $f, g \in \clWkneg{k}$ with $f \notin \clTcWknegCO{k}$ and $g \notin \clWknegOO{k}$, we have \linebreak $\gen{f,g} = \clWkneg{k}$.
\end{enumerate}
\end{proposition}

\begin{proof}
This follows from Proposition~\ref{prop:Uk} by taking inner negations and applying Lemma~\ref{lem:stability-nd}.
\end{proof}

\begin{proposition}
\label{prop:UkWkneg}
Let $k \geq 2$.
For any $f \in \clUk{k} \cap \clWkneg{k}$ with $f \notin \clVako$, we have $\gen{f} = \clUk{k} \cap \clWkneg{k}$.
\end{proposition}

\begin{remark}
The fact that $\clVako$ is the unique lower cover of $\clUk{k} \cap \clWkneg{k}$ follows from Theorem~\ref{thm:UkTC-covers} by making use of the following facts:
\begin{itemize}
\item $\clVako = \clKlik{k}{\{0\}}$ and $\clUk{k} \cap \clWkneg{k} = \clKlik{k}{\{\mathord{\nrightarrow}\}}$,
\item $0$ is the least element and $\mathord{\nrightarrow}$ is the unique atom of $\posetAllleq{k}$ (see Proposition~\ref{prop:minmin-10-atom}),
\item ${\downarrow^{[\leq k]}} \{ \mathord{\nrightarrow} \}$ is not $(k,C)$\hyp{}closed for any $C \in \{\clXI, \clIX, \clM, \clMneg, \clRefl\}$
because each one of
$\mathord{\nrightarrow}^\clXI = \mathord{\nrightarrow}^\clM = \pr^{(2)}_1 \eqminmin \id$,
$\mathord{\nrightarrow}^\clIX = \mathord{\nrightarrow}^\clMneg = \neg(\pr^{(2)}_1) \eqminmin \neg$,
and
$\mathord{\nrightarrow}^\clRefl = \mathord{+}$
has at most $k$ true points and is strictly above $\nrightarrow$ in the minorant\hyp{}minor order (see Definitions~\ref{def:C-closure} and \ref{def:kC-closed}).
\end{itemize}
\end{remark}

\begin{proof}
Since $f \in (\clUk{k} \cap \clWkneg{k}) \setminus \clVako$, we have $f(\vect{0}) = f(\vect{1}) = 0$ and there exists an $\vect{a}$ with $f(\vect{a}) = 1$; since $f \in \clUk{k}$, we must have $f(\overline{\vect{a}}) = 0$.
Thus there is a binary minor $f' \leq f$ with $f'(0,0) = f'(1,1) = f'(0,1) = 0$ and $f'(1,0) = 1$; in other words, $f' = \mathord{\nrightarrow}$.
We also have $0 \leq f$.
It suffices to show that $(\clUk{k} \cap \clWkneg{k}) \setminus \clVako$ is $(G,k)$\hyp{}semibisectable with $G := \{0, \mathord{\nrightarrow}\}$, because from this it follows by Lemma~\ref{lem:helpful} and Proposition~\ref{prop:constants}\ref{prop:constants:0} that
$\clUk{k} \cap \clWkneg{k} = ((\clUk{k} \cap \clWkneg{k}) \setminus \clVako) \cup \clVako \subseteq \gen{\mathord{\nrightarrow}} \cup \gen{0} \subseteq \gen{f} \subseteq \clUk{k} \cap \clWkneg{k}$.

Let $\theta \in (\clUk{k} \cap \clWkneg{k}) \setminus \clVako$.
In order to verify condition \ref{helpful:k-true}, let $\vect{a}_1, \dots, \vect{a}_k \in \theta^{-1}(1)$.
Then there exist $i$ and $j$ such that $a_{1i} = \dots = a_{ki} = 1$ and $a_{1j} = \dots = a_{kj} = 0$.
Consequently, $\mathord{\nrightarrow}_{ij}(\vect{a}_1) = \dots = \mathord{\nrightarrow}_{ij}(\vect{a}_k) = 1$.
For condition \ref{helpful:both}, let $\vect{a} \in \theta^{-1}(1)$ and $\vect{b} \in \theta^{-1}(0)$.
Since $\vect{a} \notin \{\vect{0}, \vect{1}\}$ and $\vect{a} \neq \vect{b}$, there exist $i$, $j$, $k$ such that $a_i = 0$, $a_j = 1$, $a_k \neq b_k$.
If $a_k = 1$ and $b_k = 0$, then $\mathord{\nrightarrow}_{ki}(\vect{a}) = 1$ and $\mathord{\nrightarrow}_{ki}(\vect{b}) = 0$.
Assume now that there is no $\ell$ such that $a_\ell = 1$ and $b_\ell = 0$.
Then $a_k = 0$, $b_k = 1$, $b_j = 1$, and we have $\mathord{\nrightarrow}_{jk}(\vect{a}) = 1$ and $\mathord{\nrightarrow}_{jk}(\vect{b}) = 0$.
\end{proof}

\begin{proposition}
Let $k \geq 2$.
For any $f \in (\clKlik{k}{\{+\}} \cap \clReflOO) \setminus \clVako$, we have $\gen{f} = \clKlik{k}{\{+\}} \cap \clReflOO$.
\end{proposition}

\begin{proof}
Since $f \in \clReflOO \setminus \clVako$, we have $f(\vect{0}) = f(\vect{1}) = 0$ and there is a $\vect{u}$ such that $f(\vect{u}) = f(\overline{\vect{u}}) = 1$. (Clearly $\vect{u} \notin \{\vect{0}, \vect{1}\}$.)
Consequently, $\mathord{+} \minor f$ and also $0 \minor f$.
It suffices to show that
$(\clKlik{k}{\{+\}} \cap \clReflOO) \setminus \clVako$ is $(G,k)$\hyp{}semibisectable with $G := \{\mathord{+}\}$, because from this it follows by Lemma~\ref{lem:helpful} and Proposition~\ref{prop:constants}\ref{prop:constants:0} that
$\clKlik{k}{\{+\}} \cap \clReflOO
= ((\clKlik{k}{\{+\}} \cap \clReflOO) \setminus \clVako) \cup \clVako
\subseteq \gen{\mathord{+}} \cup \gen{0}
\subseteq \gen{f}
\subseteq \clKlik{k}{\{+\}} \cap \clReflOO$.

Let $\theta \in (\clKlik{k}{\{+\}} \cap \clReflOO) \setminus \clVako$.
In order to verify condition \ref{helpful:k-true}, let $T = \{\vect{a}_1, \dots, \vect{a}_k \} \subseteq \theta^{-1}(1)$.
Since $\theta \in \clKlik{k}{\{+\}}$, there exist $i$ and $j$ such that $\mathord{+}_{ij}(\vect{a}) = 1$ for all $\vect{a} \in T$.
This shows that condition \ref{helpful:k-true} is satisfied.
For condition \ref{helpful:both}, let $\vect{a} \in \theta^{-1}(1)$ and $\vect{b} \in \theta^{-1}(0)$.
Since $\vect{a} \notin \{\vect{0}, \vect{1}\}$ and $\vect{a} \neq \vect{b}$, there exist $i$, $j$, $k$ such that $a_i = 0$, $a_j = 1$, and $a_k \neq b_k$.
Since $\theta \in \clRefl$, we have $\vect{a} \neq \overline{\vect{b}}$, so there exists an $\ell$ such that $a_\ell = b_\ell$.
It is not difficult to verify that the indices $i$, $j$, $k$, $\ell$ can be chosen in such a way that $\{i, j\} = \{k, \ell\}$.
(If for all $p$ such that $a_p = 0$ we have $a_p = b_p$, then $a_k = 1$ and $b_k = 0$, and we can take $i = \ell$ and $j = k$.
If for all $p$ such that $a_p = 1$ we have $a_p = b_p$, then $a_k = 0$ and $b_k = 1$, and we can take $j = \ell$ and $i = k$.
Otherwise, there are $p$ and $q$ such that $a_p = 0$, $b_p = 1$, $a_q = 1$, $b_q = 0$.
If $a_\ell = b_\ell = 0$, then we can take $i = \ell$ and $j = k := q$.
If $a_\ell = b_\ell = 1$, then we can take $j = \ell$ and $i = k := p$.)
Then $a_i \neq a_j$ and $b_i = b_j$, and we have $\mathord{+}_{ij}(\vect{a}) = 1$ and $\mathord{+}_{ij}(\vect{b}) = 0$.
\end{proof}

\subsection{Special case $k = 2$}

Our main result, Theorem~\ref{thm:IcMcUk-clonoids}, is, regretfully, somewhat implicit, because the description of the $(\clIc,\clMcUk{k})$\hyp{}clonoids is given in terms of the lattice $\Ideals(\posetAllleq{k})$ of ideals of the minorant\hyp{}minor poset $\posetAllleq{k}$ of Boolean functions with at most $k$ true points and the $(k,C)$\hyp{}closed ideals. The minorant\hyp{}minor poset is not well understood (at least by the current author), and we do not have an explicit description of the ideal lattice $\Ideals(\posetAllleq{k})$, except for small values of $k$. In the case when $k = 2$, we do obtain an explicit description of the $(\clIc,\clMcUk{2})$\hyp{}clonoids with the help of the explicit description of the ideal lattice $\Ideals(\posetAllleq{2})$ (see Figure~\ref{fig:minmin2}). It only remains to identify the $(2,C)$\hyp{}closed ideals for each $C \in \{\clXI, \clIX, \clM, \clMneg, \clRefl\}$.
We leave it as an exercise to the reader to verify that
\begin{itemize}
\item the $(2,\clXI)$\hyp{}closed ideals are the downsets of $\{1\}$, $\{\id, \mathord{+}\}$, $\{\id, \lambda_{30}\}$, $\{\id\}$, $\{0\}$, and $\emptyset$;
\item the $(2,\clIX)$\hyp{}closed ideals are the downsets of $\{1\}$, $\{\neg, \mathord{+}\}$, $\{\neg, \lambda_{31}\}$, $\{\neg\}$, $\{0\}$, and $\emptyset$;
\item the $(2,\clM)$\hyp{}closed ideals are the downsets of $\{1\}$, $\{\id, \mathord{+}\}$, $\{\id\}$, $\{0\}$, and $\emptyset$;
\item the $(2,\clMneg)$\hyp{}closed ideals are the downsets of $\{1\}$, $\{\neg, \mathord{+}\}$, $\{\neg\}$, $\{0\}$, and $\emptyset$;
\item the $(2,\clRefl)$\hyp{}closed ideals are the downsets of $\{1\}$, $\{\mathord{+}\}$, $\{0\}$, and $\emptyset$.
\end{itemize}
We thus find all $(\clIc,\clMcUk{2})$\hyp{}clonoids; they are listed in Table~\ref{table:McU2-stable} and shown in Figure~\ref{fig:McU2-stable}.
This result is, of course, not new, as the $(\clIc,\clMcUk{2})$\hyp{}clonoids were described in our earlier paper \cite[Corollary~5.28(a)]{Lehtonen-SM}.
Note that the classes of the form $\clKlik{2}{\Theta}$, the yellow\hyp{}shaded part of the diagram, form a sublattice isomorphic to $\Ideals(\posetAllleq{2})$.
For these classes, we use the ``more familiar'' notation introduced in Section~\ref{sec:Boolean}; see Example~\ref{ex:classes} for a translation.

\begin{figure}
\begin{center}
\scalebox{0.28}{
\tikzstyle{every node}=[circle, draw, fill=black, scale=1, font=\huge]
\tikzstyle{minclosed}=[fill=red]
\tikzstyle{forallk}=[fill=blue]
\tikzstyle{Icl}=[fill=yellow]
\tikzstyle{Ocl}=[fill=orange]
\tikzstyle{Mcl}=[fill=pink]
\tikzstyle{Mnegcl}=[fill=purple]
\tikzstyle{Rcl}=[fill=green]
\pgfdeclarelayer{poset}
\pgfdeclarelayer{blobs}
\pgfsetlayers{blobs,poset}
\begin{tikzpicture}[baseline, scale=1]
\begin{pgfonlayer}{poset}
\coordinate (bottom) at (0,0);
\coordinate (sw) at ($(bottom)+(-20,9)$);
\coordinate (se) at ($(bottom)+(20,9)$);
\coordinate (top) at ($(bottom)+(0,50)$);
\coordinate (nw) at ($(top)+(0,-9)$);
\coordinate (ne) at ($(top)+(20,-9)$);
\node[minclosed] (empty) at ($(bottom)+(-9,0)$) {};
\draw ($(empty)+(270:0.7)$) node[draw=none,fill=none]{$\clEmpty$};
\node[minclosed] (D0C0) at ($(bottom)+(-13,4)$) {};
\draw ($(D0C0)+(270:0.7)$) node[draw=none,fill=none]{$\clVako$};
\node[forallk] (D0C1) at ($(bottom)+(16,4)$) {};
\draw ($(D0C1)+(270:0.7)$) node[draw=none,fill=none]{$\clVaki$};
\node[forallk] (D0) at ($(bottom)+(12,8)$) {};
\draw ($(D0)+(270:0.7)$) node[draw=none,fill=none]{$\clVak$};
\node[minclosed] (All) at ($(top)+(4,0)$) {};
\draw ($(All)+(90:0.7)$) node[draw=none,fill=none]{$\clAll$};
\node[forallk] (Eiio) at ($(top)+(4,-4)$) {};
\draw ($(Eiio)+(135:0.8)$) node[draw=none,fill=none]{$\clEiio$};
\node[minclosed] (Eiii) at ($(top)+(-8,-11)$) {};
\draw ($(Eiii)+(135:1)$) node[draw=none,fill=none]{$\clEiii$};
\node[forallk] (Eioi) at ($(top)+(12,-4)$) {};
\draw ($(Eioi)+(45:0.8)$) node[draw=none,fill=none]{$\clEioi$};
\node[forallk] (C0D0) at ($(top)+(0,-8)$) {};
\draw ($(C0D0)+(160:1.2)$) node[draw=none,fill=none]{$\clOXC$};
\node[forallk] (E1D0) at ($(top)+(4,-8)$) {};
\draw ($(E1D0)+(160:1.2)$) node[draw=none,fill=none]{$\clXIC$};
\node[forallk] (P0) at ($(top)+(8,-8)$) {};
\draw ($(P0)+(180:0.7)$) node[draw=none,fill=none]{$\clEq$};
\node[forallk] (E0D0) at ($(top)+(12,-8)$) {};
\draw ($(E0D0)+(0:1.3)$) node[draw=none,fill=none]{$\clXOC$};
\node[forallk] (C1D0) at ($(top)+(16,-8)$) {};
\draw ($(C1D0)+(20:1.2)$) node[draw=none,fill=none]{$\clIXC$};
\node[forallk] (C1) at ($(C1D0)+(2,-2)$) {};
\draw ($(C1)+(0:0.7)$) node[draw=none,fill=none]{$\clIX$};
\node[forallk] (E1) at ($(E1D0)+(0,-2)$) {};
\draw ($(E1)+(10:0.8)$) node[draw=none,fill=none]{$\clXI$};
\node[forallk] (C1E1D0) at ($(top)+(18,-16)$) {};
\draw ($(C1E1D0)+(80:1.2)$) node[draw=none,fill=none,rotate=284]{$\clIIC$};
\node[forallk] (C1E1) at ($(C1E1D0)+(2,-2)$) {};
\draw ($(C1E1)+(0:0.8)$) node[draw=none,fill=none]{$\clII$};
\node[forallk] (C0E0D0) at ($(E1)+1.5*(0,-2)$) {};
\draw ($(C0E0D0)+(180:1.25)$) node[draw=none,fill=none]{$\clOOC$};
\node[Icl] (C0E1) at ($(bottom)+(-14,19)$) {};
\draw ($(C0E1)+(90:0.8)$) node[draw=none,fill=none]{$\clOI$};
\node[Icl] (C0E1D000) at ($(C0E1)+(-1.2,1.2)$) {};
\draw ($(C0E1D000)+(102:1.1)$) node[draw=none,fill=none,rotate=77]{$\clOICO$};
\node[forallk] (C0E1D011) at ($(C0E1)+(1.2+0.5,1.2)$) {};
\draw ($(C0E1D011)+(350:1.4)$) node[draw=none,fill=none]{$\clOICI$};
\node[forallk] (C0E1D0) at ($(C0E1)+(0+0.5,2*1.2)$) {};
\draw ($(C0E1D0)+(102:1)$) node[draw=none,fill=none,rotate=77]{$\clOIC$};
\node[minclosed] (C0) at ($(C0E1D000)+4.5*(0.6,2.4)$) {};
\draw ($(C0)+(180:0.7)$) node[draw=none,fill=none]{$\clOX$};
\node[Icl] (SminC0E1) at ($(C0E1)+(-1.2-0.5,-1.2)$) {};
\draw ($(SminC0E1)+(280:0.7)$) node[draw=none,fill=none]{$\clSminOI$};
\node[Icl] (SminC0E1D000) at ($(SminC0E1)+(-1.2,1.2)$) {};
\draw ($(SminC0E1D000)+(102:1.1)$) node[draw=none,fill=none,rotate=77]{$\clSminOICO$};
\node[minclosed] (SminC0) at ($(SminC0E1D000)+4*(0.6,2.4)$) {};
\draw ($(SminC0)+(180:0.7)$) node[draw=none,fill=none]{$\clSminOX$};
\node[Icl] (TcU2) at ($(SminC0E1)+(-1.2-0.5,-1.2)$) {};
\draw ($(TcU2)+(345:0.8)$) node[draw=none,fill=none]{$\clTcU$};
\node[Icl] (TcU2D0) at ($(TcU2)+(-1.2,1.2)$) {};
\draw ($(TcU2D0)+(180:1.3)$) node[draw=none,fill=none]{$\clTcUCO$};
\node[minclosed] (U2) at ($(TcU2D0)+3.5*(0.6,2.4)$) {};
\draw ($(U2)+(180:0.6)$) node[draw=none,fill=none]{$\clU$};
\node[Mcl] (Mc) at ($(C0E1)+1.5*(0.6,-2.4)$) {};
\draw ($(Mc)+(300:0.8)$) node[draw=none,fill=none]{$\clMc$};
\node[Mcl] (M0) at ($(C0E1D000)+1.5*(0.6,-2.4)$) {};
\draw ($(M0)+(270:0.8)$) node[draw=none,fill=none]{$\clMo$};
\node[forallk] (M1) at ($(C0E1D011)+1.5*(0.6,-2.4)$) {};
\draw ($(M1)+(290:0.8)$) node[draw=none,fill=none]{$\clMi$};
\node[forallk] (M) at ($(C0E1D0)+1.5*(0.6,-2.4)$) {};
\draw ($(M)+(260:0.7)$) node[draw=none,fill=none]{$\clM$};
\node[Mcl] (MU2) at ($(TcU2D0)+1.5*(0.6,-2.4)$) {};
\draw ($(MU2)+(180:0.9)$) node[draw=none,fill=none]{$\clMU$};
\node[Mcl] (McU2) at ($(TcU2)+1.5*(0.6,-2.4)$) {};
\draw ($(McU2)+(350:1.1)$) node[draw=none,fill=none]{$\clMcU$};
\node[Ocl] (C1E0) at ($(bottom)+(8,19)$) {};
\draw ($(C1E0)+(75:0.8)$) node[draw=none,fill=none]{$\clIO$};
\node[forallk] (C1E0D011) at ($(C1E0)+(1.2+0.5,1.2)$) {};
\draw ($(C1E0D011)+(0:1.4)$) node[draw=none,fill=none]{$\clIOCI$};
\node[Ocl] (C1E0D000) at ($(C1E0)+(-1.2,1.2)$) {};
\draw ($(C1E0D000)+(140:1.4)$) node[draw=none,fill=none,rotate=315]{$\clIOCO$};
\node[forallk] (C1E0D0) at ($(C1E0)+(0+0.5,2*1.2)$) {};
\draw ($(C1E0D0)+(140:1.4)$) node[draw=none,fill=none,rotate=315]{$\clIOC$};
\node[minclosed] (E0) at ($(C1E0D000)+4.5*(-0.6,2.4)$) {};
\draw ($(E0)+(0:0.8)$) node[draw=none,fill=none]{$\clXO$};
\node[Ocl] (SminC1E0) at ($(C1E0)+(-1.2-0.5,-1.2)$) {};
\draw ($(SminC1E0)+(280:0.7)$) node[draw=none,fill=none]{$\clSminIO$};
\node[Ocl] (SminC1E0D000) at ($(SminC1E0)+(-1.2,1.2)$) {};
\draw ($(SminC1E0D000)+(140:1.4)$) node[draw=none,fill=none,rotate=315]{$\clSminIOCO$};
\node[minclosed] (SminE0) at ($(SminC1E0D000)+4*(-0.6,2.4)$) {};
\draw ($(SminE0)+(0:0.7)$) node[draw=none,fill=none]{$\clSminXO$};
\node[Ocl] (TcW2neg) at ($(SminC1E0)+(-1.2-0.5,-1.2)$) {};
\draw ($(TcW2neg)+(335:1.0)$) node[draw=none,fill=none]{$\clTcWneg$};
\node[Ocl] (TcW2negD0) at ($(TcW2neg)+(-1.2,1.2)$) {};
\draw ($(TcW2negD0)+(195:1.4)$) node[draw=none,fill=none]{$\clTcWnegCO$};
\node[minclosed] (W2neg) at ($(TcW2negD0)+3.5*(-0.6,2.4)$) {};
\draw ($(W2neg)+(0:0.7)$) node[draw=none,fill=none]{$\clWneg$};
\node[Mnegcl] (Mcneg) at ($(C1E0)+1.5*(0.6,-2.4)$) {};
\draw ($(Mcneg)+(280:0.8)$) node[draw=none,fill=none]{$\clMcneg$};
\node[forallk] (M0neg) at ($(C1E0D011)+1.5*(0.6,-2.4)$) {};
\draw ($(M0neg)+(10:0.8)$) node[draw=none,fill=none]{$\clMoneg$};
\node[Mnegcl] (M1neg) at ($(C1E0D000)+1.5*(0.6,-2.4)$) {};
\draw ($(M1neg)+(270:0.8)$) node[draw=none,fill=none]{$\clMineg$};
\node[forallk] (Mneg) at ($(C1E0D0)+1.5*(0.6,-2.4)$) {};
\draw ($(Mneg)+(250:0.8)$) node[draw=none,fill=none]{$\clMneg$};
\node[Mnegcl] (MW2neg) at ($(TcW2negD0)+1.5*(0.6,-2.4)$) {};
\draw ($(MW2neg)+(165:1.1)$) node[draw=none,fill=none]{$\clMWneg$};
\node[Mnegcl] (McW2neg) at ($(TcW2neg)+1.5*(0.6,-2.4)$) {};
\draw ($(McW2neg)+(320:1.0)$) node[draw=none,fill=none]{$\clMcWneg$};
\node[minclosed] (C0E0) at ($(C0)!0.5!(E0)+(0,-4)$) {};
\draw ($(C0E0)+(180:0.8)$) node[draw=none,fill=none]{$\clOO$};
\node[minclosed] (SminC0E0) at ($(C0E0)+(0,-3)+3*(-0.15,0)$) {};
\draw ($(SminC0E0)+(180:0.8)$) node[draw=none,fill=none]{$\clSminOO$};
\node[minclosed] (U2E0) at ($(SminC0E0)+(-2,-2)+2*(-0.15,0)$) {};
\draw ($(U2E0)+(200:0.7)$) node[draw=none,fill=none]{$\clUOO$};
\node[minclosed] (W2negC0) at ($(SminC0E0)+(2,-2)+2*(-0.15,0)$) {};
\draw ($(W2negC0)+(335:1.0)$) node[draw=none,fill=none]{$\clWnegOO$};
\node[minclosed] (U2W2neg) at ($(SminC0E0)+(0,-4)+4*(-0.15,0)$) {};
\draw ($(U2W2neg)+(340:1.3)$) node[draw=none,fill=none]{$\clUWneg$};
\node[minclosed] (Smin) at ($(SminC0)!0.5!(SminE0)+(-2,4)$) {};
\draw ($(Smin)+(270:0.6)$) node[draw=none,fill=none]{$\clSmin$};
\node[forallk] (X1P0) at ($(bottom)+(17,31)$) {};
\draw ($(X1P0)+(45:0.6)$) node[draw=none,fill=none]{$\clRefl$};
\node[forallk] (X1C0E0D0) at ($(X1P0)+2*(-1,-12/17)$) {};
\draw ($(X1C0E0D0)+(345:1.2)$) node[draw=none,fill=none]{$\clReflOOC$};
\node[Rcl] (X1C0E0) at ($(X1P0)+16*(-1,-12/17)$) {};
\draw ($(X1C0E0)+(180:0.8)$) node[draw=none,fill=none]{$\clReflOO$};
\node[forallk] (X1C1E1D0) at ($(X1P0)+2*(1,-12/17)$) {};
\draw ($(X1C1E1D0)+(80:1.2)$) node[draw=none,fill=none,rotate=284]{$\clReflIIC$};
\node[forallk] (X1C1E1) at ($(X1P0)+4*(1,-12/17)$) {};
\draw ($(X1C1E1)+(0:0.8)$) node[draw=none,fill=none]{$\clReflII$};
\foreach \u/\v in {
   empty/D0C0, empty/D0C1, D0C0/D0, D0C1/D0,
   D0/X1C0E0D0, D0/X1C1E1D0,
   D0C0/MU2,
   D0C0/U2W2neg, U2W2neg/U2E0, U2W2neg/W2negC0,
   X1C0E0/X1C0E0D0, X1C0E0D0/X1P0, X1C0E0/C0E0, X1C1E1/X1C1E1D0, X1C1E1D0/X1P0, X1C1E1/C1E1, 
   D0C0/X1C0E0, D0C1/X1C1E1,
   McU2/Mc, McU2/MU2, McU2/TcU2, MU2/M0, MU2/TcU2D0, TcU2/TcU2D0, TcU2D0/U2,
   Mc/M0, Mc/M1, M0/M, M1/M, U2E0/U2,
   McW2neg/Mcneg, McW2neg/MW2neg, McW2neg/TcW2neg, MW2neg/M1neg, MW2neg/TcW2negD0, TcW2neg/TcW2negD0, TcW2negD0/W2neg, Mcneg/M0neg, Mcneg/M1neg, M0neg/Mneg, M1neg/Mneg,
   W2negC0/W2neg,
   Mc/C0E1, U2/SminC0,
   U2E0/SminC0E0, W2negC0/SminC0E0, SminC0E0/C0E0,
   Mcneg/C1E0,
   W2neg/SminE0,
   M0/C0E1D000, M1/C0E1D011, M/C0E1D0, M0neg/C1E0D011, M1neg/C1E0D000, Mneg/C1E0D0,
   X1C0E0D0/C0E0D0, X1C1E1D0/C1E1D0,
   X1P0/P0,
   C0E0/C0, C0E0/E0, C0E1D000/C0, C0E1D011/E1,
   C1E0D011/C1, C1E0D000/E0,
   C1E1/C1, C1E1/E1,
   C0E0/C0E0D0, C1E1/C1E1D0, C0E1/C0E1D000, C0E1/C0E1D011, C0E1D000/C0E1D0, C0E1D011/C0E1D0, C1E0/C1E0D000, C1E0/C1E0D011, C1E0D000/C1E0D0, C1E0D011/C1E0D0,
   C0/C0D0, C1/C1D0, E0/E0D0, E1/E1D0,
   C0E0D0/C0D0, C0E0D0/E0D0, C1E1D0/C1D0, C1E1D0/E1D0, C0E1D0/C0D0, C0E1D0/E1D0, C1E0D0/C1D0, C1E0D0/E0D0,
   C0E0D0/P0, C1E1D0/P0,
   C0D0/Eiio, C1D0/Eioi, E0D0/Eioi, E1D0/Eiio,
   C0/Eiii, E0/Eiii,
   P0/Eioi,
   Eiio/All, Eiii/All,
   Eioi/All,
   TcU2/SminC0E1,
   TcW2neg/SminC1E0,
   TcU2D0/SminC0E1D000,
   TcW2negD0/SminC1E0D000,
   SminC0E0/SminC0, SminC0E0/SminE0,
   SminC0E1/SminC0E1D000,
   SminC1E0/SminC1E0D000,
   SminC0E1/C0E1, SminC0E1D000/C0E1D000, SminC0E1D000/SminC0,
   SminC1E0/C1E0, SminC1E0D000/C1E0D000, SminC1E0D000/SminE0,
   SminC0/C0, SminE0/E0,
   SminC0/Smin, SminE0/Smin,
   Smin/Eiii,
   D0/M, D0/Mneg,
   D0C0/MW2neg,
   empty/McU2, empty/McW2neg,
   D0C1/M1, D0C1/M0neg, D0/Mneg, P0/Eiio%
}
{
   \draw [thick] (\u) -- (\v);
}
\end{pgfonlayer}{poset}
\begin{pgfonlayer}{blobs}
    \draw[thin,draw=yellow!60!black,fill=yellow!25] plot[smooth cycle] coordinates{($(All)+(-1,1)$) ($(Eiii)+(-1.5,1)$) ($(U2)+(-2,-0.5)$) ($(U2E0)+(-1.5,-1.5)$) ($(D0C0)+(-1.5,-0.25)$) ($(empty)+(1,-1)$) ($(X1C0E0)+(-3,0)$) ($(W2neg)+(2.5,0)$) ($(E0)+(2,1)$) ($(C0D0)+(-2.75,-0.5)$) ($(All)+(1,-1)$)};
    \draw[thin,draw=green!60!black,fill=green!25] (X1C0E0) circle [radius=0.75];
    \draw[thin,draw=green!60!black,fill=green!25] (C0E0) circle [radius=0.75];
    \draw[thin,draw=green!60!black,fill=magenta!22] (D0C0) circle [radius=1];
    \draw[thin,draw=green!60!black,fill=magenta!22] (empty) circle [radius=1];
    \draw[thin,draw=green!60!black,fill=green!25] (D0C0) circle [radius=0.66];
    \draw[thin,draw=green!60!black,fill=green!25] (empty) circle [radius=0.66];
    \draw[thin,draw=blue!60!black,fill=blue!25] ($(C0E1)+(35:0.75)$) -- ($(C0E1D000)+(35:0.75)$) arc (35:135:0.75) -- ($(TcU2D0)+(135:0.75)$) arc (135:215:0.75) -- ($(TcU2)+(215:0.75)$) arc (215:315:0.75) -- ($(C0E1)+(315:0.75)$) arc (315:395:0.75);
    \draw[thin,draw=magenta!60!black,fill=magenta!22] ($(Mc)+(35:0.75)$) -- ($(M0)+(35:0.75)$) arc (35:135:0.75) -- ($(MU2)+(135:0.75)$) arc (135:215:0.75) -- ($(McU2)+(215:0.75)$) arc (215:315:0.75) -- ($(Mc)+(315:0.75)$) arc (315:395:0.75);
    \draw[thin,draw=blue!60!black,fill=blue!25] ($(C1E0)+(35:0.75)$) -- ($(C1E0D000)+(35:0.75)$) arc (35:135:0.75) -- ($(TcW2negD0)+(135:0.75)$) arc (135:215:0.75) -- ($(TcW2neg)+(215:0.75)$) arc (215:315:0.75) -- ($(C1E0)+(315:0.75)$) arc (315:395:0.75);
    \draw[thin,draw=magenta!60!black,fill=magenta!22] ($(Mcneg)+(35:0.75)$) -- ($(M1neg)+(35:0.75)$) arc (35:135:0.75) -- ($(MW2neg)+(135:0.75)$) arc (135:215:0.75) -- ($(McW2neg)+(215:0.75)$) arc (215:315:0.75) -- ($(Mcneg)+(315:0.75)$) arc (315:395:0.75);
    \draw[thin,draw=blue!60!black,fill=blue!25] plot[smooth cycle] coordinates{($(C0)+(135:1)$) ($(U2)+(135:1)$) ($(U2)+(225:1)$) ($(U2)+(315:1)$) ($(C0)+(315:1)$) ($(C0)+(405:1)$)};
    \draw[thin,draw=magenta!60!black,fill=magenta!22] (C0) circle [radius=0.75];
    \draw[thin,draw=magenta!60!black,fill=magenta!22] (U2) circle [radius=0.75];
    \draw[thin,draw=blue!60!black,fill=blue!25] plot[smooth cycle] coordinates{($(E0)+(155:1)$) ($(W2neg)+(155:1)$) ($(W2neg)+(245:1)$) ($(W2neg)+(325:1)$) ($(E0)+(325:1)$) ($(E0)+(425:1)$)};
    \draw[thin,draw=magenta!60!black,fill=magenta!22] (E0) circle [radius=0.75];
    \draw[thin,draw=magenta!60!black,fill=magenta!22] (W2neg) circle [radius=0.75];
\end{pgfonlayer}{blobs}
\end{tikzpicture}
}
\end{center}
\caption{$(\clIc,\protect\clMcUk{2})$\hyp{}stable classes.}
\label{fig:McU2-stable}
\end{figure}


\section{$(C_1,C_2)$\hyp{}stability for $C_2 \supseteq \mathsf{MU}^k_{01}$}
\label{sec:C1C2}

The complete description of $(\clIc,\clMcUk{k})$\hyp{}clonoids, Theorem~\ref{thm:IcMcUk-clonoids}, can be used to find also all $(C_1,C_2)$\hyp{}clonoids for other pairs of clones $C_1$ and $C_2$.

\subsection{Auxiliary results}

By the monotonicity of function class composition, if $C_1$ is an arbitrary clone and $C_2 \supseteq \clMcUk{k}$ for some $k \geq 2$, then every $(C_1,C_2)$\hyp{}clonoid is also a $(\clIc,\clMcUk{k})$\hyp{}clonoid.
Therefore, it is just a matter of checking the stability of each $(\clIc,\clMcUk{k})$\hyp{}clonoid under right and left composition with clones $C_1$ and $C_2$.
In fact, we only need to identify the largest clones $C_1$ and $C_2$ for which we have stability, as explained by the following result.

\begin{theorem}[{\cite[Proposition~2.8]{Lehtonen-SM}}]
\label{thm:CK1CK2}
For every minion $K \subseteq \mathcal{F}_{AB}$, there exist clones $C^K_1 \subseteq \mathcal{O}_A$ and $C^K_2 \subseteq \mathcal{O}_B$
such that for all clones $C_1 \subseteq \mathcal{O}_A$ and $C_2 \subseteq \mathcal{O}_B$ it holds that $K$ is $(C_1,C_2)$\hyp{}stable if and only if $C_1 \subseteq C^K_1$ and $C_2 \subseteq C^K_2$.
\end{theorem}

Note that in the notations $C^K_1$ and $C^K_2$, the superscript $K$ indicates dependency from the minion $K$.

\begin{lemma}[{\cite[Lemma~5.3]{Lehtonen-SM}}]
\label{lem:intersections}
Let $K_1, K_2 \subseteq \mathcal{F}_{AB}$, $C_1, C_2 \subseteq \mathcal{O}_A$, and $D_1, D_2 \subseteq \mathcal{O}_B$.
\begin{enumerate}[label={\upshape(\roman*)}]
\item If $K_1 C_1 \subseteq K_1$ and $K_2 C_2 \subseteq K_2$, then $(K_1 \cap K_2)(C_1 \cap C_2) \subseteq K_1 \cap K_2$.
\item If $D_1 K_1 \subseteq K_1$ and $D_2 K_2 \subseteq K_2$, then $(D_1 \cap D_2)(K_1 \cap K_2) \subseteq K_1 \cap K_2$.
\end{enumerate}
\end{lemma}

The task of describing the $(C_1,C_2)$\hyp{}clonoids for $C_2 \supseteq \clMcUk{k}$ was partly done in \cite{Lehtonen-SM}, where we determined all $(C_1,C_2)$\hyp{}clonoids for $C_2 \supseteq \clSM$.
In particular, the $(\clIc,\clMcUk{2})$\hyp{}clonoids are listed in Table~\ref{table:McU2-stable}; these are also $(\clIc,\clMcUk{k})$\hyp{}clonoids for every $k \geq 2$.
For each such clonoid $K$, the table also indicates the clones $C^K_1$ and $C^K_2$ as in Theorem~\ref{thm:CK1CK2}.

\begin{table}
\setlength{\tabcolsep}{3.9pt}
\scalebox{0.92}{\footnotesize%
\begin{tabular}{cccc@{\quad\qquad}cccc}
\toprule
$K$ & $K C \subseteq K$ & $C K \subseteq K$ & $K$ & $K C \subseteq K$ & $C K \subseteq K$  \\
& iff $C \subseteq \ldots$ & iff $C \subseteq \ldots$ & & iff $C \subseteq \ldots$ & iff $C \subseteq \ldots$ \\
\midrule
$\clAll$ & $\clAll$ & $\clAll$ & & & \\
$\clEiio$ & $\clOI$ & $\clM$ & $\clEioi$ & $\clOI$ & $\clM$ \\
$\clEiii$ & $\clOI$ & $\clU$ & & & \\
$\clEq$ & $\clOI$ & $\clAll$ & & & \\
$\clOXC$ & $\clOX$ & $\clM$ & $\clXOC$ & $\clXI$ & $\clM$ \\
$\clIXC$ & $\clOX$ & $\clM$ & $\clXIC$ & $\clXI$ & $\clM$ \\
$\clOX$ & $\clOX$ & $\clOX$ & $\clXO$ & $\clXI$ & $\clOX$ \\
$\clIX$ & $\clOX$ & $\clXI$ & $\clXI$ & $\clXI$ & $\clXI$ \\
$\clOOC$ & $\clOI$ & $\clM$ & $\clIIC$ & $\clOI$ & $\clM$ \\
$\clOO$ & $\clOI$ & $\clOX$ & $\clII$ & $\clOI$ & $\clXI$ \\
$\clOIC$ & $\clOI$ & $\clM$ & $\clIOC$ & $\clOI$ & $\clM$ \\
$\clOICO$ & $\clOI$ & $\clMo$ & $\clIOCI$ & $\clOI$ & $\clMi$ \\
$\clOICI$ & $\clOI$ & $\clMi$ & $\clIOCO$ & $\clOI$ & $\clMo$ \\
$\clOI$ & $\clOI$ & $\clOI$ & $\clIO$ & $\clOI$ & $\clOI$ \\
$\clSmin$ & $\clS$ & $\clU$ & & & \\
$\clSminOX$ & $\clSc$ & $\clU$ &
$\clSminXO$ & $\clSc$ & $\clU$ \\
$\clSminOICO$ & $\clSc$ & $\clMU$ & 
$\clSminIOCO$ & $\clSc$ & $\clMU$ \\
$\clSminOI$ & $\clSc$ & $\clTcU$ &
$\clSminIO$ & $\clSc$ & $\clTcU$ \\
$\clSminOO$ & $\clSc$ & $\clU$ & & & \\
$\clM$ & $\clM$ & $\clM$ & $\clMneg$ & $\clM$ & $\clM$ \\
$\clMo$ & $\clMo$ & $\clMo$ & $\clMineg$ & $\clMi$ & $\clMo$ \\
$\clMi$ & $\clMi$ & $\clMi$ & $\clMoneg$ & $\clMo$ & $\clMi$ \\
$\clMc$ & $\clMc$ & $\clMc$ & $\clMcneg$ & $\clMc$ & $\clMc$ \\
$\clU$ & $\clU$ & $\clU$ &
$\clWneg$ & $\clW$ & $\clU$ \\
$\clTcUCO$ & $\clTcU$ & $\clMU$ & 
$\clTcWnegCO$ & $\clTcW$ & $\clMU$ \\
$\clTcU$ & $\clTcU$ & $\clTcU$ &
$\clTcWneg$ & $\clTcW$ & $\clTcU$ \\
$\clMU$ & $\clMU$ & $\clMU$ &
$\clMWneg$ & $\clMW$ & $\clMU$ \\
$\clMcU$ & $\clMcU$ & $\clMcU$ &
$\clMcWneg$ & $\clMcW$ & $\clMcU$ \\
$\clUOO$ & $\clTcU$ & $\clU$ &
$\clWnegOO$ & $\clTcW$ & $\clU$ \\
$\clUWneg$ & $\clSM$ & $\clU$ & & & \\
$\clRefl$ & $\clS$ & $\clAll$ & & & \\
$\clReflOOC$ & $\clSc$ & $\clM$ & $\clReflIIC$ & $\clSc$ & $\clM$ \\
$\clReflOO$ & $\clSc$ & $\clOX$ & $\clReflII$ & $\clSc$ & $\clXI$ \\
$\clVak$ & $\clAll$ & $\clAll$ & & & \\
$\clVako$ & $\clAll$ & $\clOX$ & $\clVaki$ & $\clAll$ & $\clXI$ \\
$\clEmpty$ & $\clAll$ & $\clAll$ & & & \\
\bottomrule
\end{tabular}}
\bigskip
\caption{$(\clIc,\mathsf{MU}^2_{01})$\hyp{}clonoids and their stability under right and left composition with clones of Boolean functions.}
\label{table:McU2-stable}
\end{table}

It only remains to determine the $(C_1,C_2)$\hyp{}clonoids, where $C_1$ is arbitrary and $\clMcUk{k} \subseteq C_2 \subseteq \clUk{k}$, for each $k \geq 3$.
The $(\clIc,\clMcUk{2})$\hyp{}clonoids listed in Table~\ref{table:McU2-stable} are also $(\clIc,\clMcUk{k})$\hyp{}clonoids for every $k \geq 2$, so we only need to consider the remaining $(\clIc,\clMcUk{k})$\hyp{}clonoids, which are of the form $\clKlik{k}{\Theta}$ or $\clKlik{k}{\Theta} \cap C$,
where
\[
C \in \{\clOICO, \clIOCO, \clOI, \clIO, \clMo, \clMc, \clMineg, \clMcneg, \clReflOO\}.
\]

\subsection{Stability under left composition}

We start with stability under left composition with $\clMUk{k}$, $\clTcUk{k}$, and $\clUk{k}$, for $k \geq 3$.
We will make use of Lemmata~\ref{lem:right-stable} and \ref{lem:left-stable}
and the generating sets of the clones $\clMcUk{k}$, $\clMUk{k}$, $\clTcUk{k}$, and $\clUk{k}$ mentioned in Subsection~\ref{subsec:gen-nu}.

\begin{lemma}
\label{lem:left-stability-apu}
Let $k \in \IN \cup \{\infty\}$, and let $\Theta \subseteq \clAllleq{k}$ with $\Theta \neq \emptyset$.
Assume that $k$ is the least element $\ell \in \IN \cup \{\infty\}$ such that $\clKlik{k}{\Theta} = \clKlik{\ell}{\Theta'}$ for some $\Theta' \subseteq \clAllleq{\ell}$.
Assume that $k \geq 2$.
\begin{enumerate}[label={\upshape(\roman*)}]
\item\label{lem:left-stability-apu:1}
There exists a set $T \subseteq \{0,1\}^n$ for some $n$ with $\card{T} = k$ such that $\chi_T \notin {\downarrow} \Theta$ \textup{(}and hence $\chi_T \notin \clKlik{k}{\Theta}$\textup{)} but for all proper subsets $S$ of $T$, $\chi_S \in {\downarrow} \Theta$ \textup{(}and hence $\chi_S \in \clKlik{k}{\Theta}$\textup{)}.
\item\label{lem:left-stability-apu:2}
For the set $T$ prescribed in \ref{lem:left-stability-apu:1}, we have $\chi_T \in \clMcUk{k-1} \, \clKlik{k}{\Theta}$ but $\chi_T \notin \clKlik{k}{\Theta}$.
Consequently, $\clMcUk{k-1} \, \clKlik{k}{\Theta} \nsubseteq \clKlik{k}{\Theta}$.
\item\label{lem:left-stability-apu:C}
For $C \in \{\clXI, \clIX, \clM, \clMneg, \clRefl\}$,
if $\Theta$ is $(k,C)$\hyp{}closed, then, for the set $T$ prescribed in \ref{lem:left-stability-apu:1}, we have $\chi_T^C \notin \clKlik{k}{\Theta} \cap C$ and for every proper subset $S$ of $T$, $\chi_S^C \in \clKlik{k}{\Theta} \cap C$.
\item\label{lem:left-stability-apu:C2}
For $C \in \{\clXI, \clIX, \clM, \clMneg, \clRefl\}$,
if $\Theta$ is $(k,C)$\hyp{}closed, then $\clMcUk{k-1} (\clKlik{k}{\Theta} \cap C) \nsubseteq \clKlik{k}{\Theta} \supseteq \clKlik{k}{\Theta} \cap C$.
\end{enumerate}
\end{lemma}

\begin{proof}
\ref{lem:left-stability-apu:1}
Because of the minimality of $k$, we have that $\clKlik{k}{\Theta} \subsetneq \clKlik{k-1}{\Theta}$, so there is a $\varphi \in \clKlik{k-1}{\Theta} \setminus \clKlik{k}{\Theta}$.
Thus, there exists a $T \subseteq \varphi^{-1}(1)$ with $\card{T} = k$ such that $\chi_T \notin {\downarrow} \Theta$ -- and hence $\chi_T \notin \clKlik{k}{\Theta}$ -- but for every proper subset $S$ of $T$, we have $\chi_S \in {\downarrow} \Theta$.

\ref{lem:left-stability-apu:2}
Let $T = \{\vect{a}_1, \dots, \vect{a}_k\}$ be the set prescribed in \ref{lem:left-stability-apu:1},
and for $i \in \nset{k}$, let $S_i := T \setminus \{\vect{a}_i\}$.
Therefore. $\chi_{S_i} \in {\downarrow} \Theta$ and hence $\chi_{S_1}, \dots, \chi_{S_k} \in \clKlik{k}{\Theta}$ for all $i \in \nset{k}$.
Let $\tau := \threshold{k}{k-1}(\chi_{S_1}, \dots, \chi_{S_k})$.
Because $\threshold{k}{k-1} \in \clMcUk{k-1}$, we have $\tau \in \clMcUk{k-1} \, \clKlik{k}{\Theta}$.
Moreover, it is easy to verify that $\tau = \chi_T$; hence $\tau \notin \clKlik{k}{\Theta}$.
This shows that $\clMcUk{k-1} \, \clKlik{k}{\Theta} \nsubseteq \clKlik{k}{\Theta}$.

\ref{lem:left-stability-apu:C}
If $\Theta$ is $(k,C)$\hyp{}closed, then
$\chi_T^C \notin \clKlik{k}{\Theta}$
but $\chi_{S_i}^C \in \clKlik{k}{\Theta} \cap C$ for all $i \in \nset{k}$
by Lemma~\ref{lem:k-1MR-closed}.

\ref{lem:left-stability-apu:C2}
We are going to show that $\threshold{k}{k-1}(\chi_{S_1}^C, \dots, \chi_{S_k}^C) = \chi_T^C$.
Since $\clonegen{\threshold{k}{k-1}} = \clMcUk{k-1}$ and
$\chi_{S_i}^C \in \clKlik{k}{\Theta} \cap C$ and
$\chi_T^C \notin \clKlik{k}{\Theta} \supseteq \clKlik{k}{\Theta} \cap C$ by \ref{lem:left-stability-apu:C},
it follows from Lemma~\ref{lem:left-stable} that $\clMcUk{k-1} (\clKlik{k}{\Theta} \cap C) \nsubseteq \clKlik{k}{\Theta} \cap C$.

The argument is slightly different for different clones $C$, and we need to consider different cases.

Case~1: $C = \clXI$.
Because $\chi_{S_i}^\clXI(\vect{1}) = 1$ and $\chi_{S_i}^\clXI(\vect{a}) = \chi_{S_i}(\vect{a})$ whenever $\vect{a} \neq \vect{1}$, it is easy to verify that $\threshold{k}{k-1}(\chi_{S_1}^\clXI, \dots, \chi_{S_k}^\clXI) = \chi_T^\clXI$.

Case~2: $C = \clIX$.
The argument is similar to Case~1.

Case~3: $C = \clM$.
Assume $\vect{b} \in (\chi_T^\clM)^{-1}(1)$.
Then there is an $\vect{a}_j \in T$ such that $\vect{a}_j \leq \vect{b}$.
But then we have for all $i \in \nset{k} \setminus \{j\}$ that $\chi_{S_i}(\vect{a}_j) = 1$, so $\chi_{S_i}^\clM(\vect{b}) = 1$.
Therefore $\threshold{k}{k-1}(\chi_{S_1}^\clM, \dots, \chi_{S_k}^\clM)(\vect{b}) = \threshold{k}{k-1}(1, \dots, 1, \chi_{S_j}^\clM(\vect{b}), 1, \dots, 1) = 1$.
Assume now that $\vect{b} \in (\chi_T^\clM)^{-1}(0)$.
Then $\vect{b} \nleq \vect{a}_i$ for all $\vect{a}_i \in T$.
Consequently, $\chi_{S_i}(\vect{b}) = 0$ for all $i \in \nset{k}$, and so $\threshold{k}{k-1}(\chi_{S_1}^\clM, \dots, \chi_{S_k}^\clM)(\vect{b}) = \threshold{k}{k-1}(0, \dots, 0) = 0$.

Case~4: $C = \clMneg$.
The argument is similar to Case~3.

Case~5: $C = \clRefl$.
Straightforward verification.
\end{proof}

\begin{proposition}
\label{prop:left-stability}
Let $k \in \IN \cup \{\infty\}$, and let $\Theta \subseteq \clAllleq{k}$ with $\Theta \neq \emptyset$.
Assume that $k$ is the least element $\ell \in \IN \cup \{\infty\}$ such that $\clKlik{k}{\Theta} = \clKlik{\ell}{\Theta'}$ for some $\Theta' \subseteq \clAllleq{\ell}$.
Assume that $k \geq 2$.
Let $C$ be a clone of Boolean functions.
\begin{enumerate}[label={\upshape(\roman*)}]
\item\label{prop:left-stability:Klik}
$C \clKlik{k}{\Theta} \subseteq \clKlik{k}{\Theta}$ if and only if $C \subseteq \clUk{k}$.

\item\label{prop:left-stability:Klik:XI}
Assume $\Theta$ is $(k,\clXI)$\hyp{}closed.
Then
\begin{enumerate}[label={\upshape(\alph*)}]
\item\label{prop:left-stability:Klik:XI:OICO}
$C (\clKlik{k}{\Theta} \cap (\clOICO)) \subseteq \clKlik{k}{\Theta} \cap (\clOICO)$ if and only if $C \subseteq \clMUk{k}$;
\item\label{prop:left-stability:Klik:XI:OI}
$C (\clKlik{k}{\Theta} \cap \clOI) \subseteq \clKlik{k}{\Theta} \cap \clOI$ if and only if $C \subseteq \clTcUk{k}$.
\end{enumerate}

\item\label{prop:left-stability:Klik:IX}
Assume $\Theta$ is $(k,\clIX)$\hyp{}closed.
Then
\begin{enumerate}[label={\upshape(\alph*)}]
\item\label{prop:left-stability:Klik:IX:IOCO}
$C (\clKlik{k}{\Theta} \cap (\clIOCO)) \subseteq \clKlik{k}{\Theta} \cap (\clIOCO)$ if and only if $C \subseteq \clMUk{k}$;
\item\label{prop:left-stability:Klik:IX:IO}
$C (\clKlik{k}{\Theta} \cap \clIO) \subseteq \clKlik{k}{\Theta} \cap \clIO$ if and only if $C \subseteq \clTcUk{k}$.
\end{enumerate}

\item\label{prop:left-stability:Klik:M}
Assume $\Theta$ is $(k,\clM)$\hyp{}closed.
Then
\begin{enumerate}[label={\upshape(\alph*)}]
\item\label{prop:left-stability:Klik:M:Mo}
$C (\clKlik{k}{\Theta} \cap \clMo) \subseteq \clKlik{k}{\Theta} \cap \clMo$ if and only if $C \subseteq \clMUk{k}$.
\item\label{prop:left-stability:Klik:M:Mc}
$C (\clKlik{k}{\Theta} \cap \clMc) \subseteq \clKlik{k}{\Theta} \cap \clMc$ if and only if $C \subseteq \clMcUk{k}$.
\end{enumerate}

\item\label{prop:left-stability:Klik:Mneg}
Assume $\Theta$ is $(k,\clMneg)$\hyp{}closed.
Then
\begin{enumerate}[label={\upshape(\alph*)}]
\item\label{prop:left-stability:Klik:Mneg:Mineg}
$C (\clKlik{k}{\Theta} \cap \clMineg) \subseteq \clKlik{k}{\Theta} \cap \clMineg$ if and only if $C \subseteq \clMUk{k}$.
\item\label{prop:left-stability:Klik:Mneg:Mcneg}
$C (\clKlik{k}{\Theta} \cap \clMcneg) \subseteq \clKlik{k}{\Theta} \cap \clMcneg$ if and only if $C \subseteq \clMcUk{k}$.
\end{enumerate}

\item\label{prop:left-stability:Klik:R}
Assume $\Theta$ is $(k,\clRefl)$\hyp{}closed.
Then
$C (\clKlik{k}{\Theta} \cap \clReflOO) \subseteq \clKlik{k}{\Theta} \cap \clReflOO$ if and only if $C \subseteq \clUk{k}$.
\end{enumerate}
\end{proposition}

\begin{proof}
\ref{prop:left-stability:Klik}
First let us verify the claim for $k = 2$.
By reading off from Figure~\ref{fig:minmin1} and using the translation table from Example~\ref{ex:classes}, we see that the classes of the form $\clKlik{1}{\Theta}$ are
\begin{align*}
& \clKlik{1}{\{\id, \neg\}} = \clAll = \clKlik{2}{\{1\}}, &
& \clKlik{1}{\{\id\}} = \clOX = \clKlik{2}{\{\id, \mathord{+}\}}, &
& \clKlik{1}{\{\neg\}} = \clXO = \clKlik{2}{\{\neg, \mathord{+}\}}, \\
& \clKlik{1}{\{\mathord{\nrightarrow}\}} = \clOO = \clKlik{2}{\{\mathord{+}\}}, &
& \clKlik{1}{\{0\}} = \clVako = \clKlik{2}{\{0\}}, &
& \clKlik{1}{\emptyset} = \clEmpty = \clKlik{2}{\emptyset}.
\end{align*}
Now, by reading from Figure~\ref{fig:minmin2}, we see that the classes of the form $\clKlik{2}{\Theta}$ for which $2$ is the least number $\ell$ for which $\clKlik{2}{\Theta} = \clKlik{\ell}{\Theta'}$ for some $\Theta'$ are
\begin{align*}
& \clKlik{2}{\{\id, \neg, \mathord{+}\}} = \clEiii, &
& \clKlik{2}{\{\id, \neg\}} = \clSmin, &
& \clKlik{2}{\{\id, \lambda_{30}\}} = \clSminOX, &
& \clKlik{2}{\{\neg, \lambda_{31}\}} = \clSminXO, \\
& \clKlik{2}{\{\lambda_{30}, \lambda_{31}\}} = \clSminOO, &
& \clKlik{2}{\{\id\}} = \clUk{2}, &
& \clKlik{2}{\{\neg\}} = \clWkneg{2}, &
& \clKlik{2}{\{\lambda_{30}\}} = \clWknegOO{2}, \\
& \clKlik{2}{\{\lambda_{31}\}} = \clUkOO{2}, &
& \clKlik{2}{\{\mathord{\rightarrow}\}} = \clUk{2} \cap \clWk{2}.
\end{align*}
Now we can check from Table~\ref{table:McU2-stable} that for each of these classes $\clKlik{2}{\Theta}$, it holds that $C \clKlik{2}{\Theta} \subseteq \clKlik{2}{\Theta}$ if and only if $C \subseteq \clUk{2}$.

Assume now that $k \geq 3$.
Observe first that $\{\mathord{\nrightarrow}\} \clKlik{k}{\Theta} \subseteq \clKlik{k}{\Theta}$ because $\mathord{\nrightarrow}(f,g)$ is a minorant of $f$ and $\clKlik{k}{\Theta}$ is closed under minorants.
Moreover, $\clMcUk{k} \clKlik{k}{\Theta} \subseteq \clKlik{k}{\Theta}$ by Proposition~\ref{prop:UTheta-stability}.
Since $\clUk{k} = \clonegen{\clMcUk{k} \cup \{\mathord{\nrightarrow}\}}$, it now follows from Lemma~\ref{lem:left-stable} that $\clUk{k} \clKlik{k}{\Theta} \subseteq \clKlik{k}{\Theta}$.

Since every proper superclone of $\clUk{k}$ is a superclone of $\clMcUk{k-1}$, the claim will follow
if we show that $C \clKlik{k}{\Theta} \nsubseteq \clKlik{k}{\Theta}$ for every clone $C$ containing $\clMcUk{k-1}$.
By the monotonicity of function class composition, it suffices to show this for $C = \clMcUk{k-1}$; this holds by Lemma~\ref{lem:left-stability-apu}\ref{lem:left-stability-apu:2}.

\smallskip
Now, for the proof of statements \ref{prop:left-stability:Klik:XI}--\ref{prop:left-stability:Klik:R}, we can read from Table~\ref{table:McU2-stable} that
$C K \subseteq K$ if and only if $C \subseteq C^K$, where
\begin{itemize}
\item $C^K = \clMo$ for $K \in \{\clOICO, \clIOCO, \clMo, \clMineg\}$,
\item $C^K = \clOI$ for $K \in \{\clOI, \clIO\}$,
\item $C^K = \clMc$ for $K \in \{\clMc, \clMcneg\}$,
\item $C^K = \clOX$ for $K \in \{\clReflOO\}$.
\end{itemize}
By part~\ref{prop:left-stability:Klik} and Lemma~\ref{lem:intersections},
$C (\clKlik{k}{\Theta} \cap K) \subseteq \clKlik{k}{\Theta} \cap K$ whenever $C \subseteq \clUk{k} \cap C^K$.
This proves the sufficiency of the conditions.
It remains to prove necessity.
By Theorem~\ref{thm:CK1CK2}, we need to show that $C (\clKlik{k}{\Theta} \cap K) \nsubseteq \clKlik{k}{\Theta} \cap K$ whenever $C$ is a proper superclone of $C^K$; it suffices to show this for the upper covers of $C^K$ in Post's lattice.

Lemma~\ref{lem:left-stability-apu}\ref{lem:left-stability-apu:C2} shows that $\clMcUk{k-1} (\clKlik{k}{\Theta} \cap K) \subseteq \clKlik{k}{\Theta} \cap K$ and hence $C (\clKlik{k}{\Theta} \cap K) \subseteq \clKlik{k}{\Theta} \cap K$ for every superclone $C$ of $\clMcUk{k-1}$.
It only remains to consider the upper covers of $C^K$ contained in $\clUk{k}$.

\ref{prop:left-stability:Klik:R}:
There remains nothing to show for this statement.

\ref{prop:left-stability:Klik:XI}\ref{prop:left-stability:Klik:XI:OICO}\ref{prop:left-stability:Klik:XI:OI} and \ref{prop:left-stability:Klik:M}\ref{prop:left-stability:Klik:M:Mo}:
In order to show that $\clUk{k} (\clKlik{k}{\Theta} \cap K) \nsubseteq \clKlik{k}{\Theta} \cap K$ when $K$ equals $\clOICO$, $\clOI$, or $\clMo$, we apply Lemma~\ref{lem:left-stable} with the fact that $\clUk{k} = \clonegen{\threshold{k+1}{k}, \mathord{\nrightarrow}}$.
Let $f \in \clKlik{k}{\Theta} \cap K$ be a nonconstant function, and let $\vect{u}$ be a (minimal) true point of $f$.
We may assume that $\vect{u} \neq \vect{1}$ by introducting a fictitious argument.
Let now $f'$ be the minorant of $f$ with $f'(\vect{u}) = 0$ and $f'(\vect{a}) = f(\vect{a})$ for $\vect{a} \neq \vect{u}$.
Then $f' \in \clKlik{k}{\Theta} \cap K$.
Let $g := \mathord{\nrightarrow}(f,f')$.
We have
\begin{align*}
g(\vect{u}) &= f(\vect{u}) \nrightarrow f'(\vect{u}) = 1 \nrightarrow 0 = 1, \\
g(\vect{0}) &= f(\vect{0}) \nrightarrow f'(\vect{0}) = 0 \nrightarrow 0 = 0, \\
g(\vect{1}) &= f(\vect{1}) \nrightarrow f'(\vect{1}) = 1 \nrightarrow 1 = 0.
\end{align*}
Thus $g \in \clOO \setminus \clVako$; hence $g \notin K \supseteq \clKlik{k}{\Theta} \cap K$.

\ref{prop:left-stability:Klik:M}\ref{prop:left-stability:Klik:M:Mc}
In order to show that $C (\clKlik{k}{\Theta} \cap \clMc) \nsubseteq \clKlik{k}{\Theta} \cap \clMc$ when $C$ is $\clMUk{k}$ or $\clTcUk{k}$, we apply Lemma~\ref{lem:left-stable} with the fact that $\clMUk{k} = \clonegen{\threshold{k+1}{k}, 0}$ and $\clTcUk{k} = \clonegen{\threshold{k+1}{k}, x \wedge (y \rightarrow z)}$.
Firstly, $0(g_1, \dots, g_n) = 0 \notin \clMc \supseteq \clKlik{k}{\Theta} \cap \clMc$.
Let now $f \in \clKlik{k}{\Theta} \cap \clMc$ be a nonconstant function.
Then $f$ is monotone; let $\vect{u}$ be a minimal true point of $f$.
As above, we may assume that $\vect{u} \neq \vect{1}$ by adding a fictitious argument if necessary.
Now add a fictitious argument to $f$ to get a minor $f^*$ of $f$; we have $f^* \in \clKlik{k}{\Theta} \cap \clMc$.
Then $\vect{u}0$ is a minimal true point of $f^*$.
Let $f'$ and $f''$ be the minorants of $f^*$ with $f'(\vect{u}0) = f''(\vect{u}0) = 0$, $f'(\vect{u}1) = 1$, $f''(\vect{u}1) = 0$ and $f'(\vect{a}) = f''(\vect{a}) = f^*(\vect{a})$ for all $\vect{a} \notin \{\vect{u}0, \vect{u}1\}$.
Now $f', f'' \in \clKlik{k}{\Theta} \cap \clMc$.
Let $g := f^* \wedge (f' \rightarrow f'')$.
We have
\begin{align*}
g(\vect{u}0) &= f^*(\vect{u}0) \wedge (f'(\vect{u}0) \rightarrow f''(\vect{u}0)) = 1 \wedge (0 \rightarrow 0) = 1, \\
g(\vect{u}1) &= f^*(\vect{u}1) \wedge (f'(\vect{u}1) \rightarrow f''(\vect{u}1)) = 1 \wedge (1 \rightarrow 0) = 0.
\end{align*}
Consequently, $g \notin \clM \supseteq \clKlik{k}{\Theta} \cap \clMc$.

\ref{prop:left-stability:Klik:IX} and \ref{prop:left-stability:Klik:Mneg}:
These statements follow from \ref{prop:left-stability:Klik:XI} and \ref{prop:left-stability:Klik:M} by taking inner negations and applying Lemma~\ref{lem:stability-nd}.
\end{proof}

\subsection{Stability under right composition}

Unfortunately, we were not able to find such an explicit result on the stability of $(\clIc,\clMcUk{k})$\hyp{}clonoids of the form $\clKlik{k}{\Theta}$ under right composition with an arbitrary clone $C$.
We can, however, provide a necessary and sufficient condition in terms of $\Theta C$.

\begin{proposition}
Let $k \geq 2$, $\Theta \subseteq \clAllleq{k}$, and let $C$ be a clone.
We have $\clKlik{k}{\Theta} C \subseteq \clKlik{k}{\Theta}$ if and only if $({\downarrow} (\Theta C))^{[\leq k]} \subseteq ({\downarrow} \Theta)^{[\leq k]}$.
\end{proposition}

\begin{proof}
``$\Rightarrow$''
Assume $\clKlik{k}{\Theta} C \subseteq \clKlik{k}{\Theta}$.
Then
$\Theta C \subseteq \clKlik{k}{\Theta} C \subseteq \clKlik{k}{\Theta}$,
so
$({\downarrow} ( \Theta C ) )^{[\leq k]} \subseteq ( {\downarrow} (  \clKlik{k}{\Theta} ) )^{[\leq k]}$.
By Proposition~\ref{prop:UTheta-stability}, ${\downarrow} (  \clKlik{k}{\Theta} ) = \clKlik{k}{\Theta}$, so $({\downarrow} (  \clKlik{k}{\Theta} ))^{[\leq k]} = (\clKlik{k}{\Theta})^{[\leq k]}$.
Moreover, $(\clKlik{k}{\Theta})^{[\leq k]} = ({\downarrow} \Theta)^{[\leq k]}$.
(We have $\Theta \subseteq \clKlik{k}{\Theta}$, so ${\downarrow} \Theta \subseteq {\downarrow} \clKlik{k}{\Theta} = \clKlik{k}{\Theta}$; hence $({\downarrow} \Theta)^{[\leq k]} \subseteq (\clKlik{k}{\Theta})^{[\leq k]}$.
Conversely, if $f \in (\clKlik{k}{\Theta})^{[\leq k]}$, then $\card{f^{-1}(1)} \leq k$ and so $f = \chi_{f^{-1}(1)} \in {\downarrow} \Theta$; in fact, $f \in ({\downarrow} \Theta)^{[\leq k]}$.)
Putting these inclusions and equalities together, we obtain $({\downarrow} (\Theta C)^{[\leq k]}) \subseteq ({\downarrow} \Theta)^{[\leq k]}$.

``$\Leftarrow$''
Assume now that $({\downarrow} (\Theta C))^{[\leq k]} \subseteq ({\downarrow} \Theta)^{[\leq k]}$.
Let $h \in \clKlik{k}{\Theta} C$.
Then $h = f(g_1, \dots, g_n)$ for some $f \in \clKlik{k}{\Theta}$ and $g_1, \dots, g_n \in C$.
In order to show that $h \in \clKlik{k}{\Theta}$, let $T \subseteq h^{-1}(1)$ with $\card{T} \leq k$.
Then we have, for every $\vect{a} \in T$, that $g(\vect{a}) := (g_1(\vect{a}), \dots, g_n(\vect{a})) \in f^{-1}(1)$, so $g(T) \subseteq f^{-1}(1)$ and $\card{g(T)} \leq k$.
Since $f \in \clKlik{k}{\Theta}$, there exists a $\theta \in \Theta$ and a $\sigma \colon \nset{\arity \theta} \to \nset{m}$ such that for all $\vect{a} \in T$ it holds that $1 = \theta(g(\vect{a}) \sigma) = \theta_\sigma(g(\vect{a}) = (\theta_\sigma \circ g)(\vect{a})$.
Thus, $\chi_T \minorant \theta_\sigma \circ g$.
We have $\theta_\sigma \circ g \in \Theta \clIc C = \Theta C$, so $\chi_T \in ({\downarrow} \Theta C)^{[\leq k]} \subseteq ({\downarrow} \Theta)^{[\leq k]} \subseteq {\downarrow} \Theta$.
This shows that $h \in \kLoc {\downarrow} \Theta = \clKlik{k}{\Theta}$.
\end{proof}


\section{Final remarks}
\label{sec:final}
\numberwithin{theorem}{section}

The description of the $(\clIc,\clMcUk{k})$\hyp{}clonoids provided by Theorem~\ref{thm:IcMcUk-clonoids} is not entirely explicit because it relies on the minorant\hyp{}minor order and its ideals, which are not well understood (at least by the current author).
Understanding the minorant\hyp{}minor poset $(\clAll / \mathord{\eqminmin}, \mathord{\minmin})$ and its ideals (and their restriction to Boolean functions with at most $k$ true points) better would help understand also clonoids better.
Even if a complete description were unattainable, it might be possible to describe certain interesting parts of the ideal lattice $\Ideals(\posetAllleq{k})$, such as the $(k,C)$\hyp{}closed ideals for $C \in \{\clXI, \clIX, \clM, \clMneg, \clRefl\}$.

In this paper, we described the $(C_1,C_2)$\hyp{}clonoids of Boolean functions in the case when $C_1 = \clIc$ and $C_2 \supseteq \clMcUk{k}$ for some $k \geq 2$.
By making use of Lemma~\ref{lem:stable-monotonicity}, it would, in principle, be possible to identify the $(C_1,C_2)$\hyp{}clonoids, for an arbitrary $C_1$ and $C_2 \supseteq \clMcUk{k}$, among the $(\clIc,C_2)$\hyp{}clonoids.

More generally, we would be interested in describing the $(C_1,C_2)$\hyp{}clonoids for arbitrary pairs of clones $C_1$ and $C_2$ of Boolean functions.
If $C_2$ does not contain a near\hyp{}unanimity operation, the lattice $\closys{(C_1,C_2)}$ is infinite, and an explicit description may not be possible.
Nevertheless, finding the cardinality of $\closys{(C_1,C_2)}$ for arbitrary $C_1$ and $C_2$ would be a worthwhile goal, as this would offer a substantial extension of Sparks's Theorem~\ref{thm:Sparks}.

Our focus has been on $(C_1,C_2)$\hyp{}clonoids of Boolean functions.
Moving forward, the eventual goal of this line of research is to generalize these results and questions to clonoids on arbitrary finite sets.


\section*{Acknowledgments}

The author would like to thank Miguel Couceiro and Sebastian Kreinecker for inspiring discussions.


\end{document}